\documentclass[11pt,reqno]{amsart}
\usepackage{url}
\usepackage{bbm}
\usepackage[T1]{fontenc}
\usepackage[table,dvipsnames]{xcolor}

\usepackage{hyperref}
\hypersetup{
	colorlinks   = true,
	citecolor    = blue,
	linkcolor    = blue,
	urlcolor     = Blue
}

\usepackage{amsmath,amsfonts,amsthm,amssymb, comment}
\usepackage{mathtools}
\mathtoolsset{showonlyrefs=true}
\usepackage{algorithm} 
\usepackage{algpseudocode} 
\usepackage{soul}


\usepackage[numbers]{natbib}

\usepackage{mathrsfs}   
\usepackage[normalem]{ulem}
\usepackage{xfrac}
\usepackage[font={scriptsize}]{caption}
\usepackage{environ}
\usepackage{tikz}
\usepackage[margin = 1in]{geometry}
\usepackage{graphicx}
\usetikzlibrary{positioning,arrows.meta}
	\usepackage{tcolorbox}
	
	\usepackage{array}
	\newcolumntype{C}[1]{>{\centering\arraybackslash}m{#1}}
	\newcolumntype{L}[1]{>{\raggedright\arraybackslash}p{#1}}
	\usepackage[inline]{enumitem}
	\usepackage{pifont}
	
	\newcommand{\cmark}{\ding{51}} 
	\newcommand{\xmark}{\ding{55}} 
	\usepackage[justification=centering]{subcaption}
	\usepackage{multicol,verbatim}
	\usepackage{float}

\newcommand{\pl}{\partial}

\usepackage{accents}

\newcommand{\hb}[1]{\color{blue}{#1}}
\newcommand{\mr}[1]{{\color{blue}#1}}
\newcommand{\pc}[1]{{\color{cyan}#1}}
\newcommand{\us}[1]{\textcolor{red}{#1}}
\newcommand{\bxi}{\boldsymbol{\xi}}
\newcommand{\bzeta}{\boldsymbol{\zeta}}

\newcommand{\E}{\mathbb E}
\newcommand{\R}{\mathbb R}

\newcommand{\Q}{\mathbb Q}

\newcommand{\cd}{\mathcal D}
\newcommand{\ce}{\mathcal E}
\newcommand{\cf}{\mathcal F}

\newcommand{\cO}{\mathcal O}

\newcommand{\al}{\alpha}

\newcommand{\laa}{\Lambda}

\newcommand{\lp}{\left(}
\newcommand{\rp}{\right)}
\newcommand{\lc}{\left[}
\newcommand{\rc}{\right]}

\newcommand{\lln}{\left|}
\newcommand{\rrn}{\right|}

\newcommand{\lnn}{\left\|}
\newcommand{\rnn}{\right\|}  

\numberwithin{equation}{section}
\newtheorem{theorem}{Theorem}[section]
\newtheorem{lemma}{Lemma}[section]

\newtheorem{proposition}{Proposition}[section]
\newtheorem{corollary}{Corollary}[section]

\newtheorem{definition}{Definition}[section]

\newtheorem{assumption}{Assumption}[section]
\theoremstyle{remark}
\newtheorem{remark}{Remark}[section]

\newcommand{\bean}{\begin{eqnarray*}}
	\newcommand{\eean}{\end{eqnarray*}}
\newcommand{\ben}{\begin{enumerate}}
	\newcommand{\een}{\end{enumerate}}
\newcommand{\beq}{\begin{equation}}
	\newcommand{\eeq}{\end{equation}}

\title[ESGS: Mitigating Dimension-dependence and Incorporating Decision-dependence]{Complexity Guarantees for Zeroth-order Methods via Exponentially-shifted Gaussian Smoothing: Mitigating Dimension-dependence and Incorporating Decision-dependence}

\author[MINGRUI WANG]{MINGRUI WANG}
\address{(M. Wang) Harold and Inge Marcus Department of Industrial and Manufacturing
	Engineering, The Pennsylvania State University, University Park, PA 16802 United States\\
}
\email{mvw5822@psu.edu}

\author[PRAKASH CHAKRABORTY]{PRAKASH CHAKRABORTY}
\address{(P. Chakraborty) Harold and Inge Marcus Department of Industrial and Manufacturing
	Engineering, The Pennsylvania State University, University Park, PA 16802 United States\\
}
\email{prakashc@psu.edu}

\author[UDAY V. SHANBHAG]{UDAY V. SHANBHAG}
\address{(U. V. Shanbhag) Department of Industrial and Operations Engineering, University of Michigan, Ann Arbor, MI 48109 United States\\
}
\email{udaybag@umich.edu}

\thanks{Wang and Chakraborty acknowledge support from NSF DMS 2153915, DARPA HR0011-25-3-E012 and Korb Early Career Professorship. Shanbhag acknowledge support from AFOSR Grant FA9550-24-1-0259.}
\thanks{Our preliminary research on exponentially-shifted Gaussian smoothing was published in the 2024 Winter Simulation Conference~\cite{wang2024improving}.}
\date{\today}

\begin{document}
	\maketitle
\begin{abstract}
In this paper, we consider two distinct challenges in the resolution of nonsmooth stochastic optimization. Of these, the first pertains to the pronounced dependence of dimension in Gaussian smoothing-enabled zeroth-order schemes, impeding applications to large-scale settings. Second, no unified analysis {exists} for smoothing-enabled stochastic zeroth-order methods, allowing for capturing standard and decision-dependent stochastic optimization. To contend with the first challenge, we introduce a  new exponentially-shifted Gaussian smoothing {\bf esGs} estimator whose moment bounds enjoy a linear dependence on dimension (rather than a quadratic dependence as in standard Gaussian smoothing estimators).   Second, we show that such an estimator can be extended in two distinct ways to address decision-dependent regimes where the underlying densities are either available in closed form or not. Notably, the resulting gradient estimators continue to display a linear dependence in dimension. We then develop an {\bf esGs}-enabled stochastic zeroth-order method applicable to nonsmooth strongly convex, convex, and {non-}convex regimes. Importantly, our guarantees provide improved dimension-dependence in iteration complexities (by a factor of problem dimension $n$) while maintaing similar oracle complexities. In addition,  we also provide novel high probability guarantees   and almost sure sublinear rate guarantees in convex settings, while a new subsequential almost sure convergence guarantee is provided in nonconvex regimes. Preliminary numerics support our theoretical findings  show that the proposed schemes display improved computational times and more refined empirical accuracies. 
\end{abstract}

      \section{Introduction}\label{sec:intro}
		Consider the following optimization problem. 
		\begin{equation}\label{prob}
			\min_{x \in X} \ f(x), \qquad \text{ where }  f(x)\, \triangleq \,  \mathbb{E}_{\bxi}\left[\, F(x,\bxi)\, \right] ,
		\end{equation}  
		where $X \, \subseteq \, \mathbb{R}^n$ is a convex set, $\bxi: \Omega
		\to \mathbb{R}^d$ is a random variable taking realizations denoted by $\xi$,
		$\Xi \, \triangleq \, \left\{ \bxi(\omega) \, \mid \, \omega \in \Omega
		\,\right\}$,  $F: \mathbb{R}^n \times \mathbb{R}^d \to \mathbb{R}$ is a
		real-valued function, $F(\bullet,\xi)$ is convex and $L_0$-Lipschitz continuous
		for almost every $\xi \in \Xi$. Problems of the form \eqref{prob} routinely emerge in data science and machine learning and stochastic gradient descent (SGD) has emerged as the de-facto algorithm of choice in such settings~\cite{RobbinsMonro1951,Bottou2018}. 
		SGD proceeds by inexpensive iterative updates that use noisy but unbiased (or nearly unbiased) gradient samples, as captured by the following update rule, given an initial point $x_0$.
		$$x_{k+1} =  \Pi_X \left[\, x_k - \gamma_k \nabla F(x_k,\xi_k) \, \right], \qquad k \, \ge \, 0.$$
		The appeal of SGD lies in its low per-iteration computational cost, its ability to scale to massive datasets, and its empirical effectiveness in navigating high-dimensional, nonconvex loss landscapes. However, a fundamental assumption underlying SGD 
		lies in the availability of  {an unbiased stochastic gradient estimator} $\nabla F(x, \bxi)$ {of the true gradient $\nabla f(x)$}, 
		that can be evaluated tractably. 
		In many practically important problems, arising in 
		settings involving computer simulations, physical experiments, or when objectives are either non-differentiable or involve black-box models,
		such stochastic gradient estimators are either unavailable or too expensive to compute. 
		This brings forth the need for 
		zeroth-order (ZO) or derivative-free optimization techniques. In this paper, our focus is on  two specific  {directions in the context of stochastic} zeroth-order methods: (i) Avenues for {mitigating dimension-dependence} in gradient estimation ({with implications on complexity guarantees in stochastic optimization}) via convolution-smoothing with Gaussian kernels, and (ii) decision-dependent stochastic optimization {via zeroth-order methods}, each of which is discussed in some depth and we conclude with an articulation of the key challenges. 
		
		
		\subsection{Gradient estimation techniques}
		ZO algorithms typically build gradient surrogates through the use of function evaluations. The simplest techniques are finite-difference schemes that approximate directional derivatives; however, these can scale poorly with dimension~\cite{berahas2022theoretical}. A more scalable alternative is Simultaneous Perturbation Stochastic Approximation (SPSA), which estimates gradients by randomly perturbing all coordinates simultaneously, often requiring only two function evaluations per iteration regardless of the ambient dimension \cite{spall1992multivariate}. The statistical properties of these estimators, such as bias, variance, and higher moments, directly determine optimization performance. Recent work has provided extensive theoretical and empirical comparisons among different gradient approximation techniques, highlighting trade-offs between sample efficiency and noise robustness \cite{berahas2022theoretical}. A theoretically principled approach for ZO optimization relies on smoothing via convolution. Originating in classical analysis (Steklov smoothing) \cite{steklov1}, convolution-based smoothing replaces a potentially nonsmooth or non-Lipschitz function $f$ with a smooth surrogate $f_\eta$ obtained by integrating $f$ against a mollifier or kernel (e.g., a Gaussian). This idea has been effectively used to design algorithms for nonsmooth convex problems \cite{DeFarias08,Farzad1,nemirovski83problem} and for nonconvex stochastic optimization \cite{ghadimi2013stochastic,nesterov2017random}. Our framework builds on a smoothing framework originating in~\cite{steklov1}.   Consider $f_{\eta}$ for $\eta > 0$, the smoothed counterpart of $f$, where 
		\begin{equation}
			f_{\eta}(x) \, \triangleq \, \mathbb{E}_{Z}\left[ \, f(x+{\eta}Z) \, \right], 
		\end{equation}
		where $Z$ is a mean-zero unit-variance random variable taking realizations $z \, \in \,
		\mathbb{R}^n$. Under suitable assumptions, $f_{\eta}$ is
		$\mathcal{O}(\tfrac{1}{\eta})$-smooth; {i.e. $g_{\eta}$ is $\mathcal{O}(\tfrac{1}{\eta})$-Lipschitz}, where $g_{\eta}(x) = \nabla
		f_{\eta}(x) $ and is expectation-valued. When $Z$ is normally distributed 
		with correlation matrix $B^{-1}$,  $\tilde{g}_{\eta}$  denotes a
		gradient estimator of $f_{\eta}$, defined as   
		\begin{equation}\label{eq:twopointgrad}
			\tilde{g}_\eta(x,z) \, \triangleq \, \left(\tfrac{f(x+\eta z)-f(x)}{\eta}\right) B z.
		\end{equation} 
		In practice, one often estimates the gradient of the smoothed function using two-point estimators. Unfortunately, for common Gaussian-based schemes the second moment (and hence variance) of these estimators typically scales as $O(n^2)$ in the problem dimension $n$. This leads to worst-case iteration and sample complexity bounds of $O(n^2 \varepsilon^{-2})$ to reach an $\varepsilon$-accurate solution in nonsmooth convex settings \cite{nesterov2017random}, which is prohibitive for high-dimensional problems. This quadratic dependence has motivated alternative smoothing distributions and structural techniques such as spherical smoothing and coordinate-aware perturbations that aim to reduce the dimensional dependence for nonsmooth, nonconvex, and hierarchical problems \cite{shanbhag2021zeroth,cui2023complexity,qiu2023zeroth,marrinan2026zeroth} (See ~\cite{marrinan2026zeroth} for prior work in nonsmooth nonconvex stochastic optimization). 
		Table~\ref{tab:complexity} compares iterate and oracle
		complexity guarantees of related ZO methods in convex stochastic
		optimization for {some relevant Gaussian smoothing counterparts and is by no means expansive. Yet, most schemes leveraging Gaussian smoothing do so via the vanilla framework while the proposed {\bf esGS} framework is distinct} and leads to the best known iteration complexity for ZO methods in this setting.
		\begin{table}[htb]
			\scriptsize
			\caption{}{Comparison of complexity of zeroth-order methods}
			\centering
			\label{tab:complexity}
			\begin{tabular}{ p{4cm} C{2.3cm} c c  }
				\hline
				&  Iterate    & Oracle     &   \\
				Literature &  complexity & complexity &  nonsmooth\\
				\hline
				\citet{nesterov2017random}  &  $n^2\varepsilon^{-2}$   & $n^2\varepsilon^{-2}$    & \cmark  \\
				\hline
				\citet{ghadimi2013stochastic}   &  $\max\{n\varepsilon^{-1},n\sigma^2\varepsilon^{-2}\} $   &  $\max\{n\varepsilon^{-1},n\sigma^2\varepsilon^{-2}\} $ &  \xmark  \\
				\hline
				\rowcolor{pink}
				This work  &  $n\varepsilon^{-2}$   &  $n^2\varepsilon^{-2}$   & \cmark\\
				\hline
			\end{tabular}
		\end{table}
		\begin{tcolorbox}
			\noindent {\bf Challenge I.} Current Gaussian smoothing estimators (cf.~\cite{nesterov2017random}) are characterized by the bound $\mathbb{E}\left[ \, \left \| \, G(x,\bxi) \, \right\|^2 \, \right] \le \mathcal{O}(n^2 L_0^2)$,  where $f$ is $L_0$-Lipschitz. This leads to an iteration complexity of computing an $\epsilon$-solution of $\mathcal{O}(n^2 \epsilon^{-2})$, significantly limiting such avenues for contending with large-scale constrained regimes where every step may require projecting onto a complicated constraint set. {\em Are there related estimators where moment bounds display an improved dependence on $n$ with matching oracle complexities?}  
		\end{tcolorbox}

		\subsection{Decision-dependent stochastic optimization}
		Recent work has extended the scope of {stochastic optimization models} to more challenging environments in which the data-generating process depends on the decision variable, often referred to as \emph{decision-dependent} stochastic optimization or \emph{performative prediction}~\cite{perdomo2020performative, mendler2020stochastic, drusvyatskiy2023stochastic}. In this setting, the deployed decision $x$ influences the distribution $D(x)$ of future observations, and the resulting optimization problem takes the form
		\begin{align}
			\min_{x \, \in \, X} \, f(x) \, \triangleq \, \mathbb{E}_{\bxi\sim D(x)}\left[\, F(x,\bxi)\, \right]. \label{prob-dd}
		\end{align}
		This formulation captures many modern applications including strategic classification where agents modify features in response to a classifier, recommendation systems that alter user behavior, and economic settings where policy choices change market responses~ \cite{perdomo2020performative}. Decision-dependence leads to a violation of classical SA assumptions because naively estimated gradients are biased by the induced distributional shifts; moreover, the mapping $x \mapsto D(x)$ may be complex or implicitly defined. Consequently, algorithmic analysis requires more delicate stability and equilibrium concepts and new analytical tools to guarantee convergence to meaningful solutions~ \cite{drusvyatskiy2023stochastic}.  It is crucial to note the distinction between \emph{performative stability} and \emph{performative optimality}. The former concerns fixed points of a retrain-deploy loop, while the latter concerns the minimizer of the performative risk $f$ in \eqref{prob-dd}. These can differ substantially and consequently, convergence to a performatively stable point need not imply optimality for \eqref{prob-dd} \cite{perdomo2020performative, miller2021outside}. Early work studied how standard stochastic methods behave when the data distribution depends on deployment. \cite{mendler2020stochastic} compare greedy and lazy deploy SGD, showing the trade-off between tracking a shifting distribution and reducing deployment cost. Complementing this, \cite{drusvyatskiy2023stochastic} model decision-dependence as a structured bias in stochastic updates and show it fades near equilibrium, yielding strong convergence guarantees for a range of first-order methods and enabling algorithms with reduced number of deployments. \cite{cutler2024stochastic} establish asymptotic normality and optimality properties for SA procedures with decision-dependent distributions. More recently, the scope has expanded beyond strongly convex objectives: \cite{li2024stochastic} analyze stochastic optimization schemes for performative prediction under nonconvex losses and provide finite-time convergence guarantees to appropriate stationary-type solution notions.
		Since stable solutions may be suboptimal, there have been efforts to develop methods that explicitly aim at \emph{performative optimality}. \cite{miller2021echo} study structural conditions under which the performative risk becomes convex and propose algorithms exploiting such structure for improved sample complexity while \cite{izzo2021perfgd} propose a performative gradient descent algorithm targeting minimizer of $f$ by combining gradient information with estimate of how induced distribution changes with decision. {More recent work addressing decision-dependent problems via first-order methods may be found in~\cite{he2025decision,hikima2025stochastic}.} {There have been some limited developments using ZO methods, including addressing Markovian performative prediction~\cite{liutwo} as well as some recent work on ZO and Gaussian smoothing in {these} contexts~\cite{ray2022decision,chen2023performative,hikima2025zero}.}
		
		
		\begin{tcolorbox} \noindent{\bf Challenge II.} Few existing stochastic zeroth-order schemes can contend with  decision-dependent stochastic optimization and clean rate and complexity guarantees {for computing perfomatively optimal/stationary solutions} appear to be largely unavailable in convex or nonconvex settings even under smoothness assumptions, severely limiting the reach of such techniques in practical regimes.  
		\end{tcolorbox}

		\subsection{Contributions.} We make the following contributions as part of this work.
		
\begin{enumerate}[label =\textbf{\Roman*.}, wide, labelwidth=!, labelindent=0pt]
	
		\item {\bf  \ul{Exponentially-shifted Gaussian smoothing  {\bf esGs} estimator.}} First, we propose a novel \\ \ul{exponentially-shifted Gaussian smoothing} (\textbf{esGs}) gradient estimator. The estimator is based on a simple change-of-variable argument that perturbs coordinates by scaled exponential random variables composed with a Gaussian kernel. We show that this construction preserves unbiasedness for the gradient of the smoothed surrogate, while dramatically improving {some moment and approximation bounds, vis-a-vis Gaussian smoothing}. {Specifically, it can be shown that  the \textbf{esGs} estimator has a second moment bounded by $O(n)$, an improvement of $\mathcal{O}(n)$ over Gaussian smoothing.} 
		\medskip
		\item {\bf  \ul{Unified zeroth-order framework for standard and decision-dependent settings.}} Second, we present a unified {zeroth-order smoothed gradient} framework that covers both classical decision-independent stochastic problems and the more challenging decision-dependent  counterpart. {In the latter setting, we consider regimes where the decision-dependent distribution is either available in an explicit form or is only available via an oracle.} Within this framework, we show that the favorable moment and bias properties of the \textbf{esGs} estimator persist even when the data distribution depends on the decision variable. This allows us to extend ZO guarantees to a wider class of practical problems where distributional feedback cannot be ignored \cite{perdomo2020performative,drusvyatskiy2023stochastic}, representing a key advance in the regime of ZO methods.
		\medskip
		\item {\bf \ul{Complexity bounds for esGS-enabled stochastic ZO methods for nonsmooth optimization.}} Third, we {employ} the \textbf{esGs} estimator {within our smoothing framework in both nonsmooth convex and nonconvex settings}. {In} convex problems, we provide iteration  complexity bounds with linear dependence in $n$ for computing an $\varepsilon$-solution, improving upon the quadratic dimension dependence of prior Gaussian-smoothing methods. When incorporated within a ZO optimization scheme, employing this estimator directly reduces the iteration complexity for nonsmooth convex optimization from $O(n^2 \varepsilon^{-2})$ to $O(n \varepsilon^{-2})$, improving the dependence on dimension by a factor of $n$ relative to prior bounds obtained via Gaussian smoothing bounds \cite{nesterov2017random}.
		{In more general nonconvex regimes}, we establish guarantees for finding an $\varepsilon$-approximate stationary point, {providing similar improvements in dimension-dependence in iteration complexity and} matching the best-known dependence on accuracy while improving dimension scaling. 
		\medskip
		\item {\bf \ul{High probability guarantees and a.s. rates.}} In addition to the iteration and sample-complexity bounds, we derive a high probability guarantee in convex settings. This provides a pathway for showing that the suboptimality associated with an averaged iterate diminishes at the rate of $\tilde{\mathcal{O}}(\tfrac{n}{\sqrt{K}})$ in an almost sure sense.
\end{enumerate}		
		\medskip
		We complement our theoretical results with numerical experiments that demonstrate substantial empirical improvements over existing ZO methods, particularly in high-dimensional regimes, {largely a consequence of our lower variance gradient estimator}. {Numerical examples in decision-dependent settings validate the convergence claims.} 
		
		\medskip
		
		The remainder of the paper is organized as follows. Section~\ref{sec:Estimator} provides a formal background on convolution-based smoothing and derives the \textbf{esGs} gradient estimator, together with its key statistical properties. Section~\ref{sec:Decision_Dep} develops a unified framework for decision-independent and decision-dependent optimization problems {based on a suitably defined set of} assumptions. In Sections~\ref{sec:conv_theorems} and \ref{sec:Nonconv} we integrate the \textbf{esGs} estimator within a stochastic approximation algorithm and analyze its convergence and complexity for nonsmooth convex and nonconvex problems, respectively. Finally, Section~\ref{sec:numerics} presents numerical experiments validating the proposed method and illustrating its empirical advantages.\\
		
		\noindent {\em Notation.} We establish the notation used throughout the paper. Let $U \subseteq \mathbb{R}^n$ be an open domain.
		\textbf{$L^1(U)$} {represents the} space of Lebesgue integrable functions on the domain $U$ . For a function $g \in L^1(U)$, its $L^1$ norm is defined as
		$$ {\|g\|}_{L^1} \, \triangleq\,  \int_U |g(z)| dz.$$
		\textbf{$C^1(U)$} represents the space of continuously differentiable functions on $U$. For $g \in C^1(U)$, we denote the partial derivative with respect to the $i$-th coordinate as
		$ \pl_i g(x) \triangleq \tfrac{\pl g (x)}{\pl x_i}. $
		The gradient map is defined as
		$ \nabla g(x)\, \triangleq \, \left( \pl_i g(x) \right)_{i=1}^n,$ while  the \textbf{Sobolev} space $W^{1,1}(U)$  {is defined as}
		$$ W^{1,1}(U) \triangleq \left\{ g \in L^1(U) \cap C^1(U) \ \middle| \ \pl_i g \in L^1(U) \text{ for all } i \right\}. $$
		For a measurable function $G$ of a random vector $\bxi$, its $L^2$ norm is defined as 
		$$ \|G\|_{L^2(\xi)} \, \triangleq \, \sqrt{\E_{\bxi} [ \|\, G(\bxi)\, \|^2 ]} .$$ 
		When no ambiguity arises, we abbreviate this as $\|G\|_2$.

		\section{Exponentially-Shifted Gaussian Smoothing}\label{sec:Estimator}
		
		In this section, we initiate our discussion with a recap of the basics of convolution-based smoothing in
		Section~\ref{sec:Molli}. We then proceed to derive an exponentially-shifted
		Gaussian smoothing ({\bf esGs}) estimator in Section~\ref{sec:smoothedgrad}.
		Lipschitzian properties and moment bounds for this estimator are derived in
		Section~\ref{sec:smoothproperties}, allowing us to demonstrate the benefits with
		the standard Gaussian smoothing estimator. The section concludes with
		Section~\ref{sec:esGsestimator}, where we introduce a our proposed exponentially-shifted
		Gaussian smoothing ({\bf esGs}) estimator  and presenting an SA scheme with such an
		estimator.  
		
		\subsection{Convolution and mollification}\label{sec:Molli}
		The convolution of $f$ and $g$, where both $f$ and $g$ are real-valued measurable functions on $\mathbb{R}^n$,  is denoted by 
		$f*g$ and defined as 
		\begin{equation*}
			f*g(x) \,  \triangleq  \, \int_{\mathbb{R}^n} \ f(x-y)g(y)dy,
		\end{equation*}
		for all $x$ such that the integral exists. 
	The following regularity assumption on $f$ will be made in this paper. 
	\begin{assumption}\label{ass:f-Lip}
		$f: \R^n \mapsto \R$ is $L_0$-Lipschitz, i.e., there exists a constant $L_0>0$ such that
		$$
		\lln \, f(x) - f(y)\, \rrn \, \leq \,  L_0 \lnn \, x-y \, \rnn \text{ for all }x,y \in \R^n. 
		$$ 
	\end{assumption}
	The next result is standard, but is included for completeness.
	\begin{lemma}
		Suppose $f$ satisfies Assumption~\ref{ass:f-Lip} and $g \in L^1(\R^n)$. Then $f \ast g$ is Lipschitz continuous with $L_0(f \ast g) = L_0 {\|g\|}_{L^1}$.  
	\end{lemma}
	\begin{proof}
	Using the Lipschitzian property of $f$, the result follows through the following inequalities.
	\begin{align*}
		\lln \, (f \ast g)(x) - (f \ast g)(y)\, \rrn \, & {\, \le \, } \, \int \lln \, f(x-z) - f(y-z ) \, \rrn\, |\, g(z)\, | dz \, \\
		& \leq \, L_0 \left(\int |g(z)| dz \right)  \|\, x-y\, \|  \, = \,   L_0 {\|g\|}_{L^1} \| \, x - y \, \|. 
		\hspace{1in} \hfill \Box 
	\end{align*}
	\end{proof}
	\vspace{0.1in}
	
	The following is a special case of a classical result on convolution that relates the
	partial derivative of the convolution to the convolution of one of
	the functions and the partial derivative of the second. Note that we do not assume $f$ is $L^1$ or that $g$ is compactly supported.
		\begin{proposition}\label{prop:ConvSmooth}
			Suppose $f$ satisfies Assumption~\ref{ass:f-Lip} and $g \in W^{1,1}(\R^n)$. Then $f \ast g \in C^1(\R^n)$ and $\nabla(f \ast g) = f \ast (\nabla g)$.	
		\end{proposition}
		\begin{proof}
		With $e_i$ denoting the $i$-th standard basis vector and $t \in \R$, we have by the mean value theorem
		\begin{align*}
			\frac{1}{t} \lp f \ast g (x +t e_i) - f \ast g (x) \rp & = \int f(y) \frac{\lp g(x-y+t e_i) - g(x-y) \rp}{t} dy\\
			& = \int f(y) \frac{\pl g}{\pl x_i} (x -y +s_t(x-y) e_i) dy,
		\end{align*}
		where $0 \leq s_t(x-y) \leq t$. Through a change of variable, we have
		\beq\label{eq:f-st}
		\frac{1}{t} \lp f \ast g (x +t e_i) - f \ast g (x) \rp = \int f(z - s_t(x-y) e_i) \frac{\pl g}{\pl x_i} (x - z) dz. 
		\eeq
		Note by Lipschitzian property of $f$, $f(z)-|t| L_0 \leq f(z - s_t(x-y) e_i) \leq f(z)+|t| L_0$. Using this relation in \eqref{eq:f-st}, we obtain 
		\begin{align*}
			\frac{1}{t} \lp f \ast g (x +t e_i) - f \ast g (x) \rp   & = \int f(z - s_t(x-y) e_i) \frac{\pl g}{\pl x_i} (x - z) dz \\
			& = \int \left(f(z - s_t(x-y) e_i) - f(z) \right) \frac{\pl g}{\pl x_i} (x - z) dz \\
			&~~~+ \int f(z)  \frac{\pl g}{\pl x_i} (x - z) dz \\
			& \le L_0 \int  | t| \frac{\pl g}{\pl x_i} (x - z) dz +  \left(\, f *  \frac{\pl g}{\pl x_i}\,\right) (x) \\
			& = L_0 |t| {\lnn \frac{\pl g}{\pl x_i} \rnn}_{L^1} +  \left(\, f *  \frac{\pl g}{\pl x_i}\,\right) (x).
		\end{align*}
		The reverse inequality follows similarly, leading to the following claim. 
		\begin{align*}
			\lp f \ast \frac{\pl g}{\pl x_i}\rp (x) - |t| L_0 {\lnn \frac{\pl g}{\pl x_i} \rnn}_{L^1} \leq \frac{1}{t} \lp f \ast g (x +t e_i) - f \ast g (x) \rp \leq \lp f \ast \frac{\pl g}{\pl x_i} \rp (x) +|t| L_0  {\lnn \frac{\pl g}{\pl x_i} \rnn}_{L^1}.
		\end{align*}
		Taking limits as $t \to 0$, we now obtain the required result, completing the proof.
		$$
		\frac{\pl }{\pl x_i} \lp f \ast g \rp (x) = \lp f \ast \frac{\pl g}{\pl x_i} \rp (x).
		$$
		\end{proof}
		\medskip
		
		Now consider a nonnegative function $\phi$ on $\mathbb{R}^n$. This function $\phi$ forms the basis for our mollification\footnote{Note mollifiers $\phi$ are usually defined as being compactly supported~\cite{ermoliev1995minimization}. 
			However, here we expand the definition and consider functions like the Gaussian kernel which do not have compact support.}
		and {we} provide its explicit form in {Section~\ref{sec:smoothedgrad}}. Let $\eta>0$ and let $\phi_{\eta}$ be defined as
		\begin{equation}\label{eq:phieta}
			\phi_{\eta}(z) \, \triangleq \, \eta^{-n} \phi\left(\eta^{-1} z\right). 
		\end{equation}
		If $\phi \in L^1$  then $ \int_{\mathbb{R}^n} \phi_\eta(z) dz $ is independent of $ \eta$, which immediately follows by change of variables, i.e., 
		\begin{equation}\label{eq:intphi}
			\int_{\mathbb{R}^n} \phi_\eta(z) dz=\int_{\mathbb{R}^n} \phi\left(\eta^{-1} z\right) \eta^{-n} d z=\int_{\mathbb{R}^n} \phi(v) d v .
		\end{equation}
		We now introduce the smoothing or mollification of $f$ via $\phi$, as referred to as the $\phi_\eta$-mollified $f$, defined as 
		\begin{equation}\label{eq:feta}
			f_\eta(x) \, \triangleq \, f*\phi_\eta(x).
		\end{equation} 
		Mollification preserves strong and weak convexity, which are special cases of $\mu$-convexity. We adopt the definition employed in \cite{mu_convexity} whereby a function is $\mu$-convex if for $u_1, u_2 \in \mathbb{R}^n$,
		$f(\lambda u_1+(1-\lambda) u_2) \leq \lambda f(u_1)+(1-\lambda) f(u_2)-\frac{\lambda(1-\lambda) \mu}{2}\|u_1-u_2\|^{2}$, for any $\lambda \in[0,1].$ When $\mu$ is nonnegative (resp. nonpositive) $f$ is referred to as $\mu$-strongly convex (resp. $\mu$-weakly convex). 
		\begin{lemma}\label{prop:KeepConvex}
			Suppose $\phi$ is smooth satisfying $\int_{\mathbb{R}^n} \phi(u) d u =1$. Let $f$ satisfying Assumption~\ref{ass:f-Lip} be $\mu$-convex, where $\mu \, \in \, \mathbb{R}$. Then $f_\eta=f * \phi_\eta$ is $\mu$-convex  for all $\eta>0$.
		\end{lemma}
		\begin{proof} 
		Notice that by definition of $f_{\eta}$, 
		\begin{align*}
			f_\eta\left(\lambda u_1+(1-\lambda) u_2\right)
			& =\int_{\mathbb{R}^n} f(\lambda u_1+(1-\lambda) u_2 - z )\phi_\eta(z)dz \\
			&	=\int_{\mathbb{R}^n} f(\lambda (u_1-z)+(1-\lambda) (u_2-z) )\phi_\eta(z)dz.
		\end{align*}	
		Invoking $\mu$-convexity of $f$
		\begin{align*}		
			&f_\eta\left(\lambda u_1+(1-\lambda) u_2\right)  \leq \int_{\mathbb{R}^n} \left[ \lambda f(u_1-z)+ (1-\lambda)f(u_2-z) \right] \phi_\eta(z)dz\\
			&-\frac{\lambda(1-\lambda) \mu}{2}\|u_1-u_2\|^{2}\int_{\mathbb{R}^n} \phi_\eta(z)dz
			=\lambda f_\eta(u_1)+ (1-\lambda) f_\eta(u_2)- \frac{\lambda(1-\lambda) \mu}{2}\|u_1-u_2\|^{2},
		\end{align*}
		where we use \eqref{eq:intphi} and invoke $\int_{\mathbb{R}^n}\phi_\eta (z)dz=1$ by assumption. We have thus shown $f_\eta$ is $\mu$-convex.\\
		
		\end{proof}
		
		\medskip
		
		Next, we introduce our proposed exponentially-shifted Gaussian smoothing framework.
		
		\subsection{Exponentially-shifted Gaussian Smoothing}\label{sec:smoothedgrad}
		
		Suppose $\phi$, introduced in Section~\ref{sec:Molli}, is chosen to be the standard Gaussian density on $\mathbb{R}^n$, i.e. 
		\begin{equation}\label{eq:phiu}
			\phi(z)  \, \triangleq \,  \tfrac{1}{\sqrt{(2\pi)^n}} e^{-\frac{ \sum_{i=1}^n z_i^2}{2}}.
		\end{equation}
		From \eqref{eq:phieta}, \eqref{eq:phiu}, and the coordinate-wise decomposability of the standard Gaussian, $\phi_{\eta}$ may be defined as
		\begin{equation}\label{eq:phietau}
			\phi_{\eta}(z) \, = \,  \prod_{i=1}^n \rho_\eta(z_i), \mbox{ where } \rho_\eta(z_i) \, \triangleq \, \tfrac{1}{\eta\sqrt{(2\pi)}} e^{-\tfrac{z_i^2}{2\eta^2}} \mbox{ for } i = 1, \cdots, n.
		\end{equation}
		In addition, for any $i \in \{1, \cdots, n\}$, we define $\phi_{\eta}^{(-i)}(z^{-i}) \, \triangleq \,  \displaystyle \prod_{j \ne i}\ \rho_\eta(z_j).$ 
		\begin{remark}\label{rk:phidrvbd}
			Note that $\phi_\eta$ is infinitely smooth and  $\phi_{\eta} \in W^{1,1}(\R^n)$. 
		\end{remark}
		
		\medskip 
		Let  $f_\eta$ denote the mollified function introduced in \eqref{eq:feta}. 
		The following proposition derives the gradient of $f_{\eta}$ via convolution and this  
		representation is crucial for our gradient estimates. Observe that this representation of the gradient leads to the usage of exponential random variables in addition to Gaussian random variables, distinct from the more conventional gradient representation~\cite{nesterov2017random}.
		\begin{proposition}\label{prop:gradfeta}
			Assume $f$ satisfies Assumption~\ref{ass:f-Lip} and $\phi_{\eta}$, $f_{\eta}$ are defined as in \eqref{eq:phieta}, \eqref{eq:feta}, respectively.  
			Then for $1 \leq i \leq n$, the $i$-th partial derivative of $f_\eta$ is given by the following,  
			where $V \sim \mathcal{E}xp(1)$ and $Z \sim \mathcal{N}_n(0, \eta^2 I)$.
			\begin{equation}\label{eq:GradientfetaD}
				\frac{\partial_i f_{\eta}(x)}{\partial x_i} \,  = \, \frac{1}{\eta \sqrt{2\pi}}\mathbb{E}_{V,Z^{-i}}\left[f\left(x_i+\eta\sqrt{2V},x^{-i}-Z^{-i}\right)-f\left(x_i-\eta\sqrt{2V},x^{-i}-Z^{-i}\right)\right].
			\end{equation}
		\end{proposition}
		\begin{proof}
		Using Remark~\ref{rk:phidrvbd} to check the hypothesis of Proposition~\ref{prop:ConvSmooth}, we have that $f_{\eta} \triangleq f * \phi_{\eta}$. Consequently, if $\tfrac{\partial h(x)}{\partial x_i}$ is denoted by $h^{\prime}_i$, then by Proposition~\ref{prop:ConvSmooth}, for $i = 1, \cdots, n$, we have 
		\begin{equation}\label{eq:partialfeta}
			f_{\eta,i}^{\prime} \, = \, \left( f * \phi_{\eta} \right)_i^{\prime}  \, = \, f *  \phi_{\eta,i}^{\prime}
			\, \implies \, \tfrac{\partial f_{\eta}(x)}{\partial x_i} \, = \, \int_{\mathbb{R}^n} f(x - z) \tfrac{\partial\phi_{\eta}(z)}{\partial z_i}  dz.
		\end{equation}
		It is readily checked that 
		\begin{equation}\label{eq:gradphi1}
			\tfrac{\partial \phi_\eta(z)}{\partial z_i} \, = \, \left(\phi_\eta^{(-i)}\left(z^{-i}\right)\right)\tfrac{\partial \rho_\eta(z_i)}{\partial z_i},   \text{ where } \tfrac{\partial\rho_\eta(z_i)}{\partial z_i}= \tfrac{-z_i}{\eta^3\sqrt{2 \pi}}e^{-\frac{z_i^2}{2\eta^2}}.
		\end{equation}
		This allows us to claim the following from~\eqref{eq:partialfeta}. 
		\begin{equation}\label{eq:inner}
			\tfrac{\partial f_{\eta}(x)}{\partial x_i} \, = \,\int_{\mathbb{R}^{n-1}}\left(\int_{\mathbb{R}} f(x-z) \tfrac{\partial{\rho_\eta}(z_i)}{\partial z_i} dz_i\right) \phi_{\eta}^{(-i)}\left(z^{-i}\right)dz^{-i}.
		\end{equation}
		The inner integral in \eqref{eq:inner} can be split as follows:
		\begin{equation}\label{eq:split}
			\int_{\mathbb{R}} f(x-z) \tfrac{\partial{\rho_\eta}(z_i)}{\partial z_i}dz_i=\int_{0}^{+\infty} f(x-z) \tfrac{\partial{\rho_\eta}(z_i)}{\partial z_i} dz_i + \int_{0}^{+\infty} f(x+z) \tfrac{\partial{\rho_\eta}(-z_i)}{\partial z_i} dz_i.
		\end{equation}
		By change of variables, where $z_i \rightarrow v=\frac{z_i^2}{2\eta^2}$, from \eqref{eq:gradphi1}
		and by defining $h$ as $h(v) \, \triangleq \, e^{-v}$, we obtain 
		\begin{align}\label{eq:innerE}
			\int_{\mathbb{R}} f(x-z) {\tfrac{\partial\rho_\eta(z_i)}{\partial z_i}}dz_i=& \tfrac{1}{\eta\sqrt{2\pi}}\left[\int_{0}^{+\infty} f\left(x_i+\eta\sqrt{2v},x^{-i}-z^{-i}\right)h(v) dv \right. \nonumber\\
			&	\left. -\int_{0}^{+\infty} f\left(x_i-\eta\sqrt{2v},x^{-i}-z^{-i}\right)h(v) dv \right], 
		\end{align}
		where $h$ can be observed as the density function of the $\mathcal{E}xp(1)$ distribution. 
		Plugging \eqref{eq:innerE} into \eqref{eq:inner}, we obtain our desired result \eqref{eq:GradientfetaD}. 
		\end{proof}
		
		\begin{remark}
			{It bears reminding that $\nabla f_{\eta}(x)$ can be expressed as $$\nabla f_{\eta}(x) = \tfrac{1}{\eta} \mathbb{E}_{Z} \left[\, f(x+\eta Z) Z \, \right]= \tfrac{1}{\eta} \mathbb{E}_{Z} \left[\, (f(x+\eta Z) - f(x))Z \right],$$
				where $Z$ is a standard normal~\cite{nesterov2017random}. In Prop.~\ref{prop:gradfeta}, we show that  for any $i \in \{1, \cdots, n \}$, 
				$$[\nabla f_{\eta}(x)]_i = \frac{1}{\eta \sqrt{2\pi}}\mathbb{E}_{V,Z^{-i}}\left[f\left(x_i+\eta\sqrt{2V},x^{-i}-Z^{-i}\right)-f\left(x_i-\eta\sqrt{2V},x^{-i}-Z^{-i}\right)\right],$$ suggesting a new estimator (called an exponentially shifted Gaussian smoothing ({\bf esGs}) estimator) for the gradient that relies on generating both exponential and Gaussian random variables. We will then show that this estimator has significantly improved moment bounds, which has pronounced implications on complexity guarantees and empirical behavior.  }
		\end{remark}
		\subsection{Properties of the smoothed function.}\label{sec:smoothproperties}
		Some of the properties derived below for the smoothed function have been proven in \cite{nesterov2017random}, but are still included below for completeness. Crucial to note is the improved dependence on dimension $n$ arising from this modified estimator. 
		\begin{lemma}\label{Lemma:SmoothProperties} 
			Assume $f$ satisfies Assumption~\ref{ass:f-Lip}, and $ \phi_{\eta}$, $f_{\eta}$ are defined as in \eqref{eq:phieta}, \eqref{eq:feta}, respectively. Then the following hold for any $\eta, \eta^{\prime} > 0$ and any $x,y \in \R^n$. 
			\begin{enumerate}[label=\alph*.]
				\item \label{Lemma:b}$|f_{\eta}(x)-f_{\eta}(y)| \leq L_0\|x-y\|$ 
				\item \label{Lemma:c}$|f_{\eta}(x)-f(x)| \leq L_0\sqrt{n+1}\eta$ 
				\item \label{Lemma:d} If in addition $f$ is convex, then $f_\eta$ is convex and 
				\begin{equation}\label{eq:Lemmad}
					f(x) \leq f_\eta(x) \leq f(x)+L_0\sqrt{n+1}\eta.
				\end{equation}
				\item \label{Lemma:e} $\|f_\eta(x)-f_{\eta^{\prime}}(x)\| \leq L_0\sqrt{n+1}|\eta-\eta^{\prime}|$.
				\item \label{Lemma:f} $\|\nabla f_{\eta}(x) - \nabla f_{\eta}(y)\| \leq \frac{2L_0\sqrt{n}}{\eta \sqrt{2 \pi}}\|x-y\|$. 
				\item \label{Lemma:g} If $f$ is differentiable and $L_1$-smooth, then $\|\nabla f_{\eta}(x) - \nabla f(x)\| \leq L_1\eta \sqrt{n}$. 
			\end{enumerate}
		\end{lemma}
		\begin{table}[htb]
			\centering
			\caption{}{Comparison of {gradient estimators enabled with Gaussian smoothing vs {\bf es-Gs} }}
			\scriptsize
			\label{tab:estimators}
			\begin{tabular}{ C{2.7cm} c c c c }
				\hline
				&  &  &  &  \\
				Estimator &  $|f_{\eta}(x)-f(x)| $ & $\|\nabla f_{\eta}(x) - \nabla f_{\eta}(y)\|$ & $\|\nabla f_{\eta}(x) - \nabla f(x)\|$ &$\mathbb{E}[\|g_\eta\|^2] $\\
				\hline
				{Gaussian smoothing}~\cite{nesterov2017random}  &  $L_0\sqrt{n}\eta$   & $\frac{L_0\sqrt{n}}{\eta}\|x-y\|$   &$\frac{L_1\eta}{2}(n+3)^{3/2} $ & $L_0^2(n+4)^2$  \\
				\hline
				\rowcolor{pink}	{exp-shifted Gaussian smoothing} ({\bf esGs})  &  $L_0\sqrt{n+1}\eta$   &  $\frac{2L_0\sqrt{n}}{\eta \sqrt{2 \pi}}\|x-y\|$  & $  L_1\eta \sqrt{n}$& $\frac{4}{\pi}L_0^2n$\\
				\hline
			\end{tabular}
		\end{table}
		\begin{proof}
		\begin{enumerate}[label =\alph*., wide, labelwidth=!, labelindent=0pt]
			\item By definition of $f_\eta$ in \eqref{eq:feta}, triangle inequality for integrals, and Lipschitz continuity of $f$, respectively
			\begin{align*}
				|f_{\eta}(x)-f_{\eta}(y)|&= \left|\int_{\mathbb{R}^n}\left(f(x-z)-f(y-z)\right)\phi_{\eta}(z) d z\right| \leq \int_{\mathbb{R}^n}\left|f(x-z)-f(y-z)\right|\phi_{\eta}(z) d z\\
				&\leq L_0\|x-y\|\int_{\mathbb{R}^n}\phi_{\eta}(z) d z = L_0\|x-y\|.
			\end{align*}
			\item By definition of $f_\eta$ in \eqref{eq:feta}, triangle inequality for integrals and Lipschitz continuity of $f$, respectively
			\begin{align}\label{eq:SmoothProperty3}
				|f_{\eta}(x)-&f(x)|=\left|\int_{\mathbb{R}^n}f(x-z)\phi_{\eta}(z)dz-f(x)\right| =\left|\int_{\mathbb{R}^n}\left(f(x-z)-f(x)\right)\phi_{\eta}(z)dz\right|\nonumber\\
				& \leq \int_{\mathbb{R}^n}\left|f(x-z)-f(x)\right|\phi_{\eta}(z)dz \leq L_0 \int_{\mathbb{R}^n}\|z\|\phi_{\eta}(z)dz,
			\end{align}
			By setting $r=\|z\|=\sqrt{\sum_{i=1}^n z_i^2}$, and denoting $S_{n-1}=\frac{2\pi^{\frac{n}{2}}}{\Gamma(\frac{n}{2})}$ as the surface area of the $n-$dimensional unit sphere, by standard spherical coordinate transformation we obtain,
			\begin{align}\label{eq:Enormu}
				\int_{\mathbb{R}^n}\|z\|\phi_{\eta}(z)dz &=\int_0^{+\infty} r \tfrac{1}{\eta^n(2\pi)^{n/2}}e^{-\frac{r^2}{2\eta^2}}S_{n-1}r^{n-1}dr \stackrel{u=\tfrac{r}{\eta}}{=} \tfrac{\eta \sqrt{2 \pi}}{2^{\frac{n}{2}-1}\Gamma(\tfrac{n}{2})}\int_0^{+\infty}u^n \tfrac{1}{\sqrt{2\pi}}e^{-\frac{u^2}{2}}du\nonumber\\
				&=\tfrac{\eta \sqrt{2 \pi}}{2^{\frac{n}{2}-1}\Gamma(\frac{n}{2})} \tfrac{1}{2}\mathbb{E}_{U \sim \mathcal{N}(0,1)}\left[ \left|U\right|^n\right] = \tfrac{\sqrt{2}\eta\Gamma(\tfrac{n+1}{2})}{\Gamma(\tfrac{n}{2})},
			\end{align}
			where we utilize the fact that $\mathbb{E}_{U \sim \mathcal{N}(0,1)}[|U|^n]=\tfrac{2^{n/2}\Gamma(\frac{n+1}{2})}{\sqrt{\pi}}$ and $U \sim \mathcal{N}(0,1)$. Plugging \eqref{eq:Enormu} in \eqref{eq:SmoothProperty3} and using Gautschi's inequality~\cite{gautschi1959}, whereby $x^{1-s} < \frac{\Gamma(x+1)}{\Gamma(x+s)} < (x+1)^{1-s}$, with $x=\frac{n-1}{2}$ and $s=\frac{1}{2}$, it is readily checked that $|f_\eta(x)-f(x)| \leq L_0\eta\sqrt{n+1}$.
			\item Convexity of $f_\eta$ follows by setting $\mu=0$ in Lemma~\ref{prop:KeepConvex}. By Jensen's inequality and $Z \sim \mathcal{N}_n(0, \eta^2I)$,
			\begin{equation}\label{eq:fetageqf}
				f_\eta(x)=\mathbb{E}[f\left(x-Z\right)]=\mathbb{E}[f(x+Z)] \geq f(x+\mathbb{E}[Z])=f(x).
			\end{equation}
			Now our desired inequality \eqref{eq:Lemmad} is obtained by combining \eqref{eq:fetageqf} and Lemma~\ref{Lemma:SmoothProperties} \eqref{Lemma:c} .
			\item By Lipschitz continuity of $f$ and \eqref{eq:Enormu} we have
			\begin{align*}
				\left|f_{\eta}(x)-f_{\eta^{\prime}}(x)\right|&=\left|\mathbb{E}\left(f(x+\eta Z)-f(x+\eta^{\prime}Z)\right)\right|\\
				&\leq L_0\left|\eta-\eta^\prime\right|\mathbb{E}\|Z\|=\tfrac{\sqrt{2}L_0\Gamma(\tfrac{n+1}{2})}{\Gamma(\tfrac{n}{2})}|\eta-\eta^{\prime}|\leq L_0\sqrt{n+1}|\eta-\eta^{\prime}|,
			\end{align*}
			where $Z$ follows standard normal distribution.
			\item From Proposition~\ref{prop:ConvSmooth} and by Lipschitz continuity of $f$,
			\begin{equation}\label{eq:Lipparfeta}
				\left|\tfrac{\partial f_\eta(x)}{\partial x_i}-\tfrac{\partial f_\eta(y)}{\partial x_i}\right|=\left|\int_\mathbb{R}^n\left(f(x-z)-f(y-z)\right)\tfrac{\partial \phi_\eta(z)}{\partial z_i}dz\right|\le L_0 \|x-y\|\int_{\mathbb{R}^n}\left|\tfrac{\partial \phi_\eta(z)}{\partial z_i}\right|dz
			\end{equation}
			From \eqref{eq:gradphi1},
			\begin{align}\label{eq:intabsphieta}
				\quad \int_{\mathbb{R}^n}\left|\tfrac{\partial \phi_\eta(z)}{\partial z_i}\right|dz & = \int_{\mathbb{R}^n} \phi_{\eta}^{(-i)}(z^{-i}) \left|\tfrac{\partial \rho_\eta(z)}{\partial z_i}\right|  dz 
				= \int_{\mathbb{R}}  \left|\tfrac{\partial \rho_\eta(z)}{\partial z_i}\right|  dz_i \nonumber\\
				& = 
				\tfrac{1}{\eta^2}\mathbb{E}\left[|Z|\right]=\tfrac{1}{\eta}\sqrt{\tfrac{2}{\pi}}
				\implies 
				\left|\tfrac{\partial f_\eta(x)}{\partial x_i}-\tfrac{\partial f_\eta(y)}{\partial x_i}\right|
				\leq \tfrac{1}{\eta}\sqrt{\tfrac{2}{\pi}}{L_0 \|x-y\|}.
			\end{align}
			From \eqref{eq:intabsphieta}, 
			\begin{equation*}
				\|\nabla f_{\eta}(x) - \nabla f_{\eta}(y)\|=\sqrt{\sum_{i=1}^n\left|\tfrac{\partial f_\eta(x)}{\partial x_i}-\tfrac{\partial f_\eta(y)}{\partial x_i}\right|^2}\leq \sqrt{\sum_{i=1}^n\left(\tfrac{2L_0}{\eta\sqrt{2\pi}}\|x-y\|\right)^2}= \tfrac{2L_0\sqrt{n}}{\eta\sqrt{2 \pi}}\|x-y\|.
			\end{equation*}
			\item Since $f \in C^1, \phi_\eta \in L^1$, by Prop.~\ref{prop:ConvSmooth}, $\int_{\mathbb{R}^n}\phi_\eta(z)dz=1$, and Jensen's inequality, 
			\begin{equation*}
				\left|\tfrac{\partial f_\eta(x)}{\partial x_i}-\tfrac{\partial f_\eta(y)}{\partial x_i}\right|
				=\left|\int_{\mathbb{R}^n}\left(\tfrac{\partial f(x-z)}{\partial x_i}-\tfrac{\partial f(x)}{\partial x_i}\right)\phi_\eta(z)dz\right|
				\leq \left(\int_{\mathbb{R}^n}\left|\tfrac{\partial f(x-z)}{\partial x_i}-\tfrac{\partial f(x)}{\partial x_i}\right|^2\phi_\eta(z)dz\right)^{1/2}
			\end{equation*}
			By the above inequality  and by using $\|\nabla f(x-z)-\nabla f(x)\|^2 \leq L_1^2\|z\|^2$
			\begin{align}\label{eq:gradfetaf}
				\left\|\nabla f_{\eta}(x)-\nabla f(x)\right\|^2= {\sum_{i=1}^n	\left(\tfrac{\partial f_\eta(x)}{\partial x_i}-\tfrac{\partial f(x)}{\partial x_i}\right)^2}
				&\leq \int_{\mathbb{R}^n}{\sum_{i=1}^n \left(\tfrac{\partial f(x-z)}{\partial x_i}-\tfrac{\partial f(x)}{\partial x_i}\right)^2}\phi_\eta(z)dz\nonumber\\
				&\leq L_1^2\int_{\mathbb{R}^n}\|z\|^2\phi_{\eta}(z)dz.
			\end{align}
			Now by a spherical transformation,
			\begin{align}\label{eq:Enorm2}
				\int_{\mathbb{R}^n}\|z\|^2\phi_{\eta}(z)dz &=\int_0^{+\infty} r^2 \tfrac{1}{\eta^n(2\pi)^{n/2}}e^{-\tfrac{r^2}{2\eta^2}}S_{n-1}r^{n-1}dr\stackrel{u=\frac{r}{\eta}}{=} \tfrac{\eta^2 \sqrt{2 \pi}}{2^{\frac{n}{2}-1}\Gamma(\tfrac{n}{2})}\int_0^{+\infty}u^{n+1} \tfrac{1}{\sqrt{2\pi}}e^{-\frac{u^2}{2}}du\nonumber\\
				&=\tfrac{\eta \sqrt{2 \pi}}{2^{\frac{n}{2}-1}\Gamma(\tfrac{n}{2})}\cdot \frac{1}{2}\mathbb{E}_{U \sim \mathcal{N}(0,1)}[|U|^{n+1}]=\tfrac{2\eta^2\Gamma(\tfrac{n}{2}+1)}{\Gamma(\tfrac{n}{2})}=\tfrac{2\eta^2\tfrac{n}{2}\Gamma(\tfrac{n}{2})}{\Gamma(\tfrac{n}{2})}=\eta^2 n,
			\end{align}
			where we recall  that $\Gamma(a+1)=a\Gamma(a)$. Substituting \eqref{eq:Enorm2} in \eqref{eq:gradfetaf}, we obtain our result.
		\end{enumerate}
		\end{proof}

		\subsection{Moment bounds for {\bf esGs} gradient estimators}\label{sec:esGsestimator}
		In Proposition~\ref{prop:gradfeta}, we define an  ({\bf esGs}) gradient estimator for $f_{\eta}$, given by {$g_{\eta}{(x,V,Z)} \, \triangleq \, (g_{\eta}^i{(x,V,Z)})_{i=1}^n$, 
			\begin{equation}\label{eq:g_eta}
				g_{\eta}^i{(x,V,Z)} \, \triangleq \, \tfrac{1}{\eta\sqrt{2\pi}}\left[f\left(x_i+\eta\sqrt{2V},x^{-i}-Z^{-i}\right)-f{\left(x_i-\eta\sqrt{2V},x^{-i}-Z^{-i}\right)}\right],
			\end{equation}
			where $u^{-i} \, \triangleq \, (u_j)_{j \ne i}$, 
			$V$ and $Z$ are random variables following $\mathcal{E}{ xp}(1)$ and
			$\mathcal{N}(0,\eta^2I)$ taking realizations $v$ and $z$, respectively. 
			From eq.~\eqref{eq:GradientfetaD} $g_{\eta}$ is an unbiased estimator for $\nabla f_{\eta}$:
			\begin{equation}\label{eq:unbiased-0}\mathbb{E}_{V,Z}\left[{g}_\eta(x,V,Z)\right]=\nabla f_\eta (x).
			\end{equation}
			\begin{remark}
				The ({\bf esGs}) gradient estimator is obtained as a coordinate-wise difference of function values evaluated at points shifted by an {${\eta}$-}scaling of the square-root of an exponential random variable with parameter unity, while maintaining Gaussian smoothing at all other coordinates. It may also be recalled that if $V \sim {\ce}xp(\lambda)$, then $V \sim \mbox{Rayleigh}(1/\sqrt{2\lambda}).$ In addition, the simulation of an exponential random variable is computationally (relatively)  cheap.  
				Though our result is obtained by a direct analysis, we realized that this framework 
				is reminiscent of weak derivatives and their estimators as in
				\citep{FU2006575,jie2022using} or likelihood ratio derivative estimation techniques as in \citep{glynn1989likelihood, glynn1987likelilood}, where 
				derivatives are considered with respect to some parameter of choice. However, 
				we consider derivatives with respect to $z_i$ on the support of
				$\rho_\eta$.  In addition, in
				~\eqref{eq:partialfeta}, the differentiability requirement is transferred from $f$
				to $\phi_\eta$, weakening smoothness requirements on $f$. This is 
				crucial in obtaining \eqref{eq:GradientfetaD}. We now show that second moment bounds of ({\bf esGs}) have distinctly better dimensional dependence.  
			\end{remark}
			\begin{proposition}\label{prop:secmombd-1}
				Assume $f$ satisfies Assumption~\ref{ass:f-Lip}. Then for any $x$ and $\eta > 0$
				\begin{equation}
					\mathbb{E}_{V,Z}[\|{g}_\eta(x,V,Z)\|^2] \leq \tfrac{4}{\pi}L_0^2n,
				\end{equation}
				where $g_{\eta}$ is defined in \eqref{eq:g_eta}
			\end{proposition}
			\begin{proof}
			Since $f$ is $L_0$-Lipschitz continuous, 
			\begin{align*}
				\mathbb{E}\left[\left|\, g_{\eta}^i\left(x,V,Z \right)\right|^2\right]& = \tfrac{1}{\eta^2 2\pi}\mathbb{E}\left[\left | f \left(x_i+\eta\sqrt{2V},x^{-i}-Z^{-i} \right)-f\left(x_i-\eta\sqrt{2V},x^{-i}-Z^{-i}\right)\right|^2\right]\\
				\leq &\tfrac{1}{\eta^2 2\pi}\mathbb{E}\left[\left|L_0\left\|(2\eta\sqrt{2V},0,\cdots,0)\right\|\right|^2 \right]\leq  \tfrac{4L_0^2}{\pi}\mathbb{E}[V]=\tfrac{4L_0^2}{\pi}.
			\end{align*}
			Then the result follows by noting that
			$
			\mathbb{E}_{V,Z}[\|{g}_\eta(x,V,Z)\|^2]  = \sum_{i=1}^n \mathbb{E}\left[\left|\, g_{\eta}^i\left(x,V,Z \right)\right|^2\right] \leq \tfrac{4}{\pi}L_0^2n
			$.
			\end{proof}
			
			It bears reminding that the {\bf esGs} estimator has two distinctive improvements over the standard {\bf GS} estimator, as seen from Table~\ref{tab:estimators}. 
			
			\noindent (i) First, $\|\nabla f_{\eta}(x) - \nabla f(x)\| \le \mathcal{O}\left(\eta \sqrt{n}\right)$ for any $x$ (in contrast with $\mathcal{O}\left(\eta n^{3/2}\right)$ for standard {\bf GS} estimators). 
			
			\noindent (ii) Second,  $\mathbb{E}_{V,Z}[\|{g}_\eta(x,V,Z)\|^2] \le \mathcal{O}(L_0^2n)$ (as opposed to $\mathcal{O}(L_0^2 n^2)$ for {\bf GS} estimators). \\
			
			In this and the subsequent sections, we adapt our {\bf (esGs)} gradient estimator defined in \eqref{eq:g_eta} to stochastic optimization problems, where the objective is given in the form of an expectation {and the smoothing of $F(\bullet,\xi)$ is carried out}. Consider the constrained (unconstrained when $X=\R^n$) stochastic optimization problem
			\begin{equation}\label{eq:f=EF}
				\min_{x \, \in \, X} f(x) \text{ where }	f(x) \, \triangleq \, \mathbb{E}_{\bxi}\left[\, F(x,\bxi)\, \right]. 
			\end{equation}
			It bears reminding that in this section, we consider the setting where the expectation is taken with respect to a probability distribution that is independent of the decision variable $x$; we refer to this as the ``non-decision-dependent'' setting and this assumption is relaxed in the next section. Recall the definition of $L^2$ norm in the notation defined in Sec.~\ref{sec:intro}. We employ the following assumption on $F$ and $X$ in our future discourse.
			\begin{assumption}\label{asm:ascvgstochastic}
				\label{asm:covstoc1} 
					$F(\bullet, \bxi)$ is a Lipschitz continuous function with respect to the norm $\| \bullet \|_{L^2(\xi)}$ with Lipschitz constant $L_0$, i.e. for any $x, y\, \in \, \R^n$, 
					$
					\left\|\, F(x, \bullet) - F(y, \bullet)\, \right\|_{L^2(\xi)} \, \leq \, L_0 \|\, x-y \, \|.
					$ 
				\end{assumption}
				\begin{remark}\label{rem:EF-Lip}
					By invoking Jensen's inequality and Assumption~\ref{asm:ascvgstochastic}, observe that $f$ defined in \eqref{eq:f=EF} is $L_0$-Lipschitz, and the consequences in Lemma~\ref{Lemma:SmoothProperties} hold true.
				\end{remark}
				
				\smallskip
				
				We thus obtain the following proposition.
				\begin{proposition}\label{cor:gradfeta} 
					Suppose Assumption~\ref{asm:ascvgstochastic} holds, $f$ is defined as in \eqref{eq:f=EF}, and $\phi_{\eta}$, $f_{\eta}$ are defined as in \eqref{eq:phieta}, \eqref{eq:feta} respectively. Then for $1 \leq i \leq n$, the $i$-th partial derivative of $f_\eta$ is given by the following, where $V \sim \mathcal{E}xp(1)$ and $Z \sim \mathcal{N}_n(0, \eta^2 I)$.
					\begin{equation}\label{eq:GradientfetaD_2}
						\frac{\partial_i f_{\eta}(x)}{\partial x_i} \,  = \, \frac{1}{\eta \sqrt{2\pi}}\mathbb{E}_{V,Z, \bxi}\left[F\left(x_i+\eta\sqrt{2V},x^{-i}-Z^{-i}, \bxi\right)-F\left(x_i-\eta\sqrt{2V},x^{-i}-Z^{-i}, \bxi\right)\right]. 
					\end{equation}
				\end{proposition}
				\begin{proof}
				Since $f(x_i,x_{-i}) = \mathbb{E}_{\bxi}\left[ F(x_i,x_{-i},\bxi) \right]$, by Proposition~\ref{prop:gradfeta} and Fubini's theorem, 
				we observe that \eqref{eq:GradientfetaD_2} holds.
				\end{proof}
				
				\medskip
				{It is worth emphasizing that the oracle requirements for computing this gradient estimator  are no different but computing  a {\bf esGs}-enabled gradient estimator requires first generating an exponential random variable $V$ and a mean-zero Gaussian random variable $Z$ and then evaluating the $n$ components of the estimator via the stochastic function oracle via  \eqref{df:grad_est_i} for $i=1, \hdots, n.$ }
				Thanks to Proposition~\ref{cor:gradfeta} we may now define $\tilde{g}_{\eta}(x,v, z,\xi)$ to denote a gradient estimate of $f_{\eta}(x)$, where the $i$-th component is defined as
				\begin{equation} \label{df:grad_est_i}
					\tilde{g}_{\eta}^i(x,v,z,\xi)\, = \, \tfrac{1}{\eta\sqrt{2\pi}}\left[F\left(x_i+\eta\sqrt{2v},x^{-i}-z^{-i},\xi\right)-F\left(x_i-\eta\sqrt{2v},x^{-i}-z^{-i},\xi\right)\right]
				\end{equation}
				and {$\tilde{g}_{\eta}(x,v,z,\xi)=(\tilde{g}_{\eta}^1,\tilde{g}_{\eta}^2,\cdots,\tilde{g}_{\eta}^n)^T$}. From eq.~\eqref{eq:GradientfetaD_2}, ${g}_\eta$ is an unbiased estimator for $\nabla f_\eta$, i.e.,  
				\begin{equation}\label{eq:tilde-g-unbiased}
					\mathbb{E}_{V,Z,\bxi}\left[\tilde{g}_\eta(x,V,Z,\bxi)\right]=\nabla f_\eta (x).
				\end{equation}
		
		For any given scalar $\eta > 0$, we consider the smoothing of the integrand $F(\bullet,\xi)$, given by $F_{\eta}(\bullet,\xi)$ and defined as
		\begin{equation}
			F_{\eta}(x,\xi)= \int_{\mathbb{R}^n} F(x-z,\xi) \phi_{\eta}(z) dz,
		\end{equation}
		where $\phi_\eta$ is defined in \eqref{eq:phietau}. Note that $\E_{\bxi}[F_{\eta}(x, \bxi)] = f_{\eta}(x)$. Thanks to Proposition~\ref{prop:gradfeta}, $g_\eta(x, V,Z, \bxi)$ is an unbiased estimator of $\nabla F_{\eta}(x,\xi)$.
		As a consequence of Proposition~\ref{prop:secmombd-1}, we obtain that the second moment of $\tilde{g}_\eta$ is bounded and scales with $n$, rather than $n^2$. This represents a key differentiator with standard gradient estimators employed in Gaussian smoothing~\cite{nesterov2017random}. 
		\begin{proposition}\label{prop:secmombd}
			Let Assumption~\ref{asm:ascvgstochastic} hold. Then for any $x \in X$ and $\eta > 0$, we have that
			\begin{equation}
				\mathbb{E}_{V,Z,\bxi} \lc \|\tilde{g}_\eta(x,V,Z,\bxi)\|^2 \rc \leq \tfrac{4}{\pi}L_0^2n.
			\end{equation}
		\end{proposition}
		\begin{proof}
			Since $F(\bullet,\xi)$ is $L_0$-Lipschitz continuous, 
			\begin{align*}
				\mathbb{E}_{{V,Z,\bxi}}\left[\left|\, \tilde g_{\eta}^i\left(x,V,Z,\bxi\right)\right|^2\right]& = \tfrac{\mathbb{E}_{{V,Z,\bxi}}\left[\left|F\left(x_i+\eta\sqrt{2V},x^{-i}-Z^{-i},\bxi\right)-F\left(x_i-\eta\sqrt{2V},x^{-i}-Z^{-i},\bxi\right)\right|^2\right]}{\eta^2 2\pi}\\
				& \leq \tfrac{1}{\eta^2 2\pi}\mathbb{E}_{{V,Z,\bxi}}\left[\left|L_0^2 \left\|(2\eta\sqrt{2V},0,\cdots,0)\right\|\right|^2 \right] =  \tfrac{4L_0^2}{\pi}\mathbb{E}[V]=\tfrac{4L_0^2}{\pi}.
			\end{align*}
			Then the result follows by noting that $\mathbb{E}\left[\left\|\tilde{g}_{\eta}(x,V,Z,\bxi)\right\|^2\right]=\sum_{i=1}^n \mathbb{E}\left[\left|\tilde{g}_{\eta}^i(x,V,Z,\bxi)\right|^2\right]\leq \tfrac{4nL_0^2}{\pi} .$
			\end{proof}

			\section{Decision-dependent stochastic optimization}\label{sec:Decision_Dep}
			In this section, we consider the decision-dependent stochastic optimization problem, which requires solving the following problem:
			\begin{equation}\label{eq:DDproblem}
				\min_{x \, \in \, X} f(x), \quad \mbox{ where } f(x) \, \triangleq  \, \mathbb{E}_{\bxi \sim \mathcal{D}(x)} \left[\, \hat{F}(x, \bxi)\, \right] \, = \, \int_{\Xi(x)} \hat{F}(x, \xi)p(\xi \mid  x) d\xi,
			\end{equation}
			where $\bxi \in \Xi(x)$ follows a distribution $\mathcal{D}(x)$, dependent on $x$, and $p(\xi \mid x)$ represents the conditional (on $x$) density function of $\bxi$. We consider two distinct avenues for contending with such a problem. First, we examine the setting where the conditional density $p(\bullet \mid x)$ is known to the decision-maker in Section~\ref{sub:p_known} and subsequently examine the regime (in Section~\ref{sub:p_unknown}) where $p(\bullet\mid x)$ is unknown with evaluations of $F(x,\xi)$ obtained through a black-box oracle. Table~\ref{tab:all_estimators} {compares} the moment bounds for the {\bf esGs} estimator in the standard regime {and the two distinct} decision-dependent regimes {with standard Gaussian smoothing.}

			\begin{table}[htb]
				\tiny
				\centering
				\caption{Comparison of moment bounds of {\bf esGs} estimators}
				\label{tab:all_estimators}
				\renewcommand{\arraystretch}{1.5}
				\begin{tabular}{ C{1cm} C{0.2cm} c C{0.8cm} C{0.8cm} c  }
					\hline
					Estimator& {dd} & $g_{\eta}$ & Lip-constant of $f$& $\mathbb{E}[\|\tilde{g}_\eta\|^2] $   &Assumptions on $\hat{F}/F$\\
					\hline
					GS~\cite{nesterov2017random} &   \xmark  & $ \left(\tfrac{f(x+\eta z)-f(x)}{\eta}\right) B z$ & $L_0$   & $L_0^2(n+4)^2$ &$F$ is $L_0$ Lips in $L^2(\bxi)$ norm \\
					\hline
					\rowcolor{gray!10}    esGs  & \xmark   & $\tfrac{1}{\eta\sqrt{2\pi}}\left[f\left(x_i+\eta\sqrt{2v},x^{-i}-z^{-i}\right)-f\left(x_i-\eta\sqrt{2v},x^{-i}-z^{-i}\right)\right]$ &  $L_0$   & $\frac{4}{\pi}L_0^2n$ &$F$ is $L_0$ Lips in $L^2(\bxi)$ norm\\
					\hline
					\rowcolor{gray!30}
					&&&&&$\hat{F}$ is $\hat{L}_f$ Lips in $L^2(\bxi\sim \tilde{p})$ norm, \\
					\rowcolor{gray!30}
					esGs$^{\rm dd}_{\rm a}$  & \cmark & $\tfrac{1}{\eta\sqrt{2\pi}}\left(\frac{\hat{F}(x_{i+}^\eta(V,Z),\xi)p(\xi \mid x_{i+}^\eta(V,Z))}{\tilde{p}(\xi)}-\frac{\hat{F}(x_{i-}^\eta(V,Z),\xi)p(\xi \mid x_{i-}^\eta(V,Z))}{\tilde{p}(\xi)}\right)$ & $\widehat{L}_0$ &  $\frac{4}{\pi}\widehat{L}_0^2n$&$D(p(\xi \mid x) , p(\xi \mid y)) \leq L_\xi^2\|x-y\|^2$,\\
					\rowcolor{gray!30}
					&&&& &$\frac{p(\xi \mid x)}{\tilde{p}(\xi)}\leq M$, $\hat{F}(\bullet,\xi)\leq M_f$ \\
					\hline
					\rowcolor{gray!50}
					esGs$^{\rm dd}_{\rm u}$ &  \cmark   & $\tfrac{1}{\eta\sqrt{2\pi}}\left(\hat{F}\left(x_{i+}^\eta(v,z), \xi_1\right)-\hat{F}\left(x_{i-}^\eta(v,z), \xi_2\right)\right)$&  $L_0$  &  $\frac{4}{\pi}L_0^2n$&$\hat{F}$ is $L_0$ Lips in $L^2(\bxi\sim p)$ norm\\
					\hline
				\end{tabular}
				esGs$^{\rm dd}_{\rm {a}}$: {available} $p(\xi|x)$;      esGs$^{\rm dd}_{\rm {u}}$: Unknown $p(\xi|x)$; $\widehat{L}_0 =\sqrt{2 \lp M^2 \hat{L}_f^2 +M M_f^2L_\xi^2 \rp}$    
			\end{table}
			
			\subsection{Known structure of $p(\xi \mid x)$.}\label{sub:p_known}
			If we do know the {functional form of the} density $p(\xi \mid x)$, then we may choose an appropriate positive density $\tilde{p}(\xi)$  that allows for rewriting the function $f$ as
			\begin{equation}\label{eq:dectonondec}
				f(x)=\int_{\Xi(x)} \hat{F}(x, \xi)p(\xi \mid  x) d\xi=\int_{\Xi(x)} \hat{F}(x, \xi)\frac{p(\xi \mid  x)}{\tilde{p}(\xi)}\tilde{p}(\xi) d\xi=\mathbb{E}_{\bxi \sim \tilde{p}}\left[\tilde{F}(x,\bxi)\right],
			\end{equation}
			where
			\beq\label{eq:tilde-F} 
			\tilde{F}(x,\xi) \, \triangleq \, \hat{F}(x,
			\xi)\frac{p(\xi \mid  x)}{\tilde{p}(\xi)}.
			\eeq
			
			{Throughout} this section, {we impose a ground assumption requiring} existence of a stochastic first-order oracle {\bf SFO$^{\rm dd}_a$}, corresponding to a decision-dependent regime with  an available density $p(\xi\mid x)$.\\
			
			\begin{definition}[{\bf SFO$^{\rm dd}_a$}]  \label{sfo-dda} There exists a stochastic first-order oracle that given an $x$, produces $\hat{F}(x,\xi)$  with known density density $p(\xi\mid x)$.  
			\end{definition}
			
			\medskip
			
			Observe that in \eqref{eq:dectonondec} the problem is transformed to the non-decision-dependent setting.  But this
			may lead to the emergence of nonconvexity in
			$\tilde{F}(\bullet,\xi)$ while $\hat{F}(\bullet,\xi)$ may have been
			convex. To apply Proposition~\ref{prop:gradfeta} and define an \textbf{esGs} gradient estimator, we need to {provide assumptions under which $f$ is Lipschitz continuous. We begin by reminding the reader of the definition of the Kullback-Leibler (KL) divergence between two distributions $\mathbb{P}$ and $\mathbb{Q}$ with $\mathbb{P} \ll \mathbb{Q}$ and densities $p$ and $q$, respectively as 
				\begin{equation}
					D_{{\textnormal{KL}}}(\mathbb{P}  \|  \mathbb{Q}) \, 
					{ = D_{\textnormal{KL}}(p(\bullet), q(\bullet))}
					\triangleq \, \int p(u) \log \left(\frac{p(u)}{q(u)}\right) du,
				\end{equation}
				while the symmetric KL divergence is defined as 
				{\begin{equation}
						D_{{\textnormal{KL}}}^{{\rm sym}}(\mathbb{P},\mathbb{Q}) =	D_{{\textnormal{KL}}}(\mathbb{P} \| \mathbb{Q})+	D_{{\textnormal{KL}}}(\mathbb{Q} \| \mathbb{P}),
				\end{equation}}
				if $\mathbb{P} \ll \mathbb{Q}$ and $\mathbb{Q} \ll \mathbb{P}$. Recall that $\mathbb{P} \ll \mathbb{Q}$ indicates that $\mathbb{P}$ is absolutely continuous with respect to $\mathbb{Q}$.
				We now outline some assumptions on the density, the divergence, and $\hat{F}(\bullet,\xi)$. 
				\begin{assumption}\label{asm:decdep}
					For any $x,y\in\R^n$, the following hold 
					\begin{enumerate}
						\item $p(\xi \mid x)>0$ for any $\xi\in\Xi(x)$.
						\item\label{ass:p-x-y-dist} There exists an $L_\xi$ such that the {symmetric KL divergence} satisfies $D_{{\textnormal{KL}}}^{{\rm sym}}(p(\bullet \mid x) , p(\bullet \mid y)) \leq L_\xi^2\|x-y\|^2$. 
						\item\label{ass:p-ratio-bnd} For any $\xi\in\Xi(x)$, $\frac{p(\xi \mid x)}{\tilde{p}(\xi)}$ is uniformly bounded by $M$.
						\item\label{ass:f-hat-Lip} $\hat{F}(\bullet,\xi)$ is $\hat{L}_f$-Lipschitz continuous in the $L^2(\bxi \sim \tilde{p})$ metric, as specified next, i.e., for any $x, y \, \in \, \R^n$, the following holds.
						$$
						{\lnn \hat{F}(x, \bullet) - \hat{F}(y, \bullet) \rnn}_{L^2(\bxi \sim \tilde{p})} \, \leq \, \hat{L}_f \|x-y\|.
						$$
						Furthermore, $\hat{F}(\bullet,\xi)$ is bounded by $M_f$ for almost every $\xi\in\Xi(\bullet)$. 
					\end{enumerate}
				\end{assumption}
				
				We recall the following Lemma from \cite[Lemma 8.4]{boucheron2013concentration}, which provides a bound on {a symmetrized quadratic divergence} in terms of the KL {divergence}.
				\begin{lemma}\label{lem:d-2-ineq}
					Let $ \mathbb{P} $ and $ \mathbb{Q} $ be distributions on a common measurable space $ (\Omega, \mathcal{A}) $. If $ \mathbb{Q} \ll \mathbb{P} $,
					\begin{equation*}
						d_{2}^{2}(\mathbb{Q}, \mathbb{P})+d_{2}^{2}(\mathbb{P}, \mathbb{Q}) \leq 2 D_{{\textnormal{KL}}}(\mathbb{Q} \| \mathbb{P}),
					\end{equation*}
					where $d_{2}^{2}(\mathbb{Q}, \mathbb{P})$ is defined as
					\begin{equation*}
						d_{2}^{2}(\mathbb{Q}, \mathbb{P}) \, \triangleq \, \int \frac{(p(u)-q(u))_{+}^{2}}{p(u)} d u
					\end{equation*}
					and $p$ and $q$ are the density functions of $\mathbb{P}$ and $\Q$, respectively. 
				\end{lemma}
				
				Following similar arguments, we obtain that $f$, defined in \eqref{eq:DDproblem}, can be shown to be Lipschitz, a result of proving an intermediate Lipschitzian property as shown next.
				\begin{lemma}\label{lem:Epx-py}
					Under Assumption~\ref{asm:decdep}, the following holds for any $x, y \in \R^n$.
					\begin{align}\label{inter_lip}
						\E_{\bxi \sim \tilde{p}} \left[ \, {\lln \dfrac{p(\bxi | x)}{\tilde{p}(\bxi)} - \dfrac{p(\bxi | y)}{\tilde{p}(\bxi)} \rrn}^2 \, \right] \leq M L_\xi^2 {{\|x-y\|}}^2.
					\end{align}
				\end{lemma}
				
				\begin{proof}
				Observe that
				\begin{align}\label{eq:seclast}
					\notag     \E_{\bxi \sim \tilde{p}} &\left[{\lln \dfrac{p(\bxi | x)}{\tilde{p}(\bxi)} - \dfrac{p(\bxi | y)}{\tilde{p}(\bxi)} \rrn}^2 \right] 
					= 
					\int \left|\frac{p(\xi \mid x)}{\tilde{p}(\xi)}-\frac{p(\xi \mid y)}{\tilde{p}(\xi)}\right|^2 \tilde{p}(\xi)d\xi=\int \frac{\left|p(\xi \mid x)-p(\xi \mid y)\right|^2}{\tilde{p}(\xi)} d\xi\nonumber\\
					&=\int \frac{\left(p(\xi \mid x)-p(\xi \mid y)\right)_+^2}{p(\xi \mid x)}\frac{p(\xi \mid x)}{\tilde{p}(\xi)} d\xi+\int \frac{\left(p(\xi \mid x)-p(\xi \mid y)\right)_-^2}{p(\xi \mid y)}\frac{p(\xi \mid y)}{\tilde{p}(\xi)} d\xi\nonumber\\
					& \leq Md^2_2\left(p(\xi \mid x),p(\xi \mid y)\right)+Md^2_2\left(p(\xi \mid y),p(\xi \mid x)\right),
				\end{align}
				where we have used the fact that $\frac{p(\xi | x)}{\tilde p (\xi)}$ is uniformly bounded by $M$, which follows from Assumption~\ref{asm:decdep}.\ref{ass:p-ratio-bnd}. 
				By Lemma~\ref{lem:d-2-ineq} and Assumption~\ref{asm:decdep}.\ref{ass:p-x-y-dist}, we have that
				\begin{equation}\label{eq:KBbound}
					d^2_2\left(p(\xi \mid x),p(\xi \mid y)\right)+d^2_2\left(p(\xi \mid y),p(\xi \mid x)\right) \leq D_{{\textnormal{KL}}}^{{\rm sym}}\left(p(\xi \mid x),p(\xi \mid y)\right) \leq L_\xi^2{\|x-y\|}^2
				\end{equation}
				Combining \eqref{eq:seclast} and \eqref{eq:KBbound}, we obtain the required bound, given by \eqref{inter_lip}, for any $x, y \in \R^n$.
				\end{proof}
				
				\medskip
				
				We are now ready to prove that $f$ is indeed Lipschitz continuous.
				
				\begin{proposition}\label{lem:F-decdep1-Lip}
					Suppose Assumption~\ref{asm:decdep} holds. Then the following hold. 
					
					\noindent (a) $\tilde{F}(\bullet, \xi)$, defined in \eqref{eq:tilde-F}, is a $L_0$-Lipschitz continuous function with respect to the norm ${\| \bullet \|}_{L^2(\bxi \sim \tilde{p})}$, 
					
					\noindent (b) $f$, defined in \eqref{eq:DDproblem},  is $L_0$-Lipschitz, 
					where $L_0 = \sqrt{2 \lp M^2 \hat{L}_f^2 +M M_f^2L_\xi^2 \rp}$.  
				\end{proposition}
				\begin{proof}
				\noindent (a) We observe that by Assumption~\ref{asm:decdep} ${\lln \tilde{F}(x,\xi) - \tilde{F}(y,\xi)\rrn}^2$ can be bounded  as follows for any $x, y \, \in \, \R^n$ 
				\begin{align*}
					{\lln \tilde{F}(x,\xi) - \tilde{F}(y,\xi)\rrn}^2  &=  {\lln \dfrac{\hat{F}(x, \xi)p(\xi | x)}{\tilde{p}(\xi)} - \dfrac{\hat{F}(y, \xi)p(\xi | y)}{\tilde{p}(\xi)} \rrn}^2\\
					&\leq 2 \,  \lp {\lln \dfrac{p(\xi|x)}{\tilde{p}(\xi)} \rrn}^2 {\lln \hat{F}(x, \xi) - \hat{F}(y, \xi) \rrn}^2 + {\lln \hat{F}(y, \xi) \rrn}^2 {\lln \dfrac{p(\xi|x)}{\tilde{p}(\xi)} - \dfrac{p(\xi | y)}{\tilde{p}(\xi)} \rrn}^2 \rp\\
					&\leq 2 \lp M^2  {\lln \hat{F}(x, \xi) - \hat{F}(y, \xi) \rrn}^2 + M_f^2  {\lln \dfrac{p(\xi|x)}{\tilde{p}(\xi)} - \dfrac{p(\xi | y)}{\tilde{p}(\xi)} \rrn}^2 \rp.
				\end{align*}
				Then by taking expectations we have by invoking Assumption~\ref{asm:decdep}.\ref{ass:f-hat-Lip} and Lemma~\ref{lem:Epx-py} that 
				\begin{align*}
					\mathbb{E}_{\bxi \sim \tilde{p}}\left[ {\lln \tilde{F}(x,\xi) - \tilde{F}(y,\xi)\rrn}^2 \right]
					&\leq 2 \lp M^2\mathbb{E}_{{\bxi \sim \tilde{p}}}{\lln \hat{F}(x, \xi) - \hat{F}(y, \xi) \rrn}^2 + M_f^2\mathbb{E}_{{\bxi \sim \tilde{p}}}{\lln \dfrac{p(\xi|x)}{\tilde{p}(\xi)} - \dfrac{p(\xi | y)}{\tilde{p}(\xi)} \rrn}^2 \rp\\
					& \leq 2 \lp M^2\hat{L}_f^2 + M M_f^2 L_\xi^2 \rp {\|x-y\|}^2.
				\end{align*}
					Thus it follows that $\tilde{F}(\bullet,\xi)$ is Lipschitz in the $\|\bullet\|_{L^2(\bxi\sim \tilde{p})}$ metric. 
					
					\noindent (b) 
						Since $f$ is defined as \eqref{eq:DDproblem}, it follows from (a) $f$ is also $L_0$-Lipschitz continuous, where $L_0^2 = 2 {\lp M^2\hat{L}_f^2 + M M_f^2L_\xi^2 \rp}.$
						\end{proof}
						
						\medskip
						
						Note that the above proof for {Lipschitzian requirement on} $f$ is obtained {via} Jensen's inequality from $L^2(\bxi \sim \tilde{p})$ norms, and hence the Lipschitz constant is sub-optimal. Below we provide another proof utilizing only $L^1$ bounds thus obtaining an improved Lipschitz constant.

						\begin{lemma}\label{lem:f-decdep1-Lip}
							Under Assumption~\ref{asm:decdep}, $f$ defined in \eqref{eq:DDproblem} is $\hat{L}_0$-Lipschitz continuous, where $\hat{L}_0=M \hat{L}_f+M_fL_\xi$. Consequently, Lemma~\ref{Lemma:SmoothProperties} holds true with Lipschitz constant $\hat{L}_0$. 
						\end{lemma}
						\begin{proof}
						Recall from \eqref{eq:DDproblem} that $f(x)=\int_{\Xi} \hat{F}(x, \xi)p(\xi \mid  x) d\xi$. Consequently
						\begin{align}\label{eq:fdecLip1}
							|\, f(x)-f(y)\, |&=\left|\int_{\Xi} \hat{F}(x, \xi)p(\xi \mid  x) d\xi-\int_{\Xi} \hat{F}(y, \xi)p(\xi \mid  y) d\xi\right|\nonumber\\
							& \leq \int_\Xi\left|\hat{F}(x,\xi)-\hat{F}(y,\xi)\right| p(\xi\mid x) d\xi+\int_\Xi\left|\hat{F}(y,\xi)\right|\left|p(\xi \mid x)-p(\xi \mid y)\right|d\xi.
						\end{align}
						By Lipschitz continuity of $\hat{F}(\bullet,\xi)$ and the boundedness of $\frac{p(\xi\mid x)}{\tilde{p}(\xi)}$ from Assumption~\ref{asm:decdep}, we may bound the first term in \eqref{eq:fdecLip1} as follows.
						\beq\label{eq:fdecLip-a}
						\int_\Xi\left|\hat{F}(x,\xi)-\hat{F}(y,\xi)\right| p(\xi\mid x) d\xi=\int_\Xi\left|\hat{F}(x,\xi)-\hat{F}(y,\xi)\right|\tilde{p}(\xi) \frac{p(\xi\mid x)}{\tilde{p}(\xi)} d\xi \, \leq \,  \hat{L}_f M  {\|x-y\|}.
						\eeq
						By Pinsker's Theorem, the total variation distance $D_{\textnormal{TV}}(p(\bullet|x), p(\bullet|y)) =\frac{1}{2} \int |p(\xi|x) - p(\xi|y)| d\xi \leq \sqrt{\frac{1}{2} D_{\textnormal{KL}}(p(\bullet|x), p(\bullet|y))}$. Consequently by the elementary inequality $\sqrt{a/2}+\sqrt{b/2}<\sqrt{a+b}$ and Assumption~\ref{asm:decdep}.\ref{ass:p-x-y-dist}, we obtain
						\beq\label{eq:fdecLip-c}
						\int \lln p(\xi | x) - p(\xi|y) \rrn d\xi \leq \sqrt{D_{{\textnormal{KL}}}^{{\rm sym}}(p(\bullet |x), p(\bullet | y))} \leq L_\xi \|x-y\|. 
						\eeq
						By boundedness of $\hat{F}{(\bullet,\xi)}$ from Assumption~\ref{asm:decdep}.\ref{ass:f-hat-Lip} and \eqref{eq:fdecLip-c}, we obtain that
						\beq\label{eq:fdecLip-b}
						\int_\Xi\left|\hat{F}(y,\xi)\right|\left|p(\xi \mid x)-p(\xi \mid y)\right|d\xi \leq M_f \int \lln p(\xi | x) - p(\xi|y) \rrn d\xi \leq M_f L_\xi {\|x-y\|}.
						\eeq
						Plugging the inequalities \eqref{eq:fdecLip-a} and \eqref{eq:fdecLip-b} in the relation \eqref{eq:fdecLip1} we obtain our desired result.
						\end{proof}
						
						\medskip
						
						The prior Lipschitzian requirements then allow for us to define the {\bf esGs} gradient estimator in  decision-dependent settings. For ease of presentation, let $x_{i+}^\eta{(V,Z)}$ and $x_{i-}^\eta{(V,Z)}$ be defined as 
						\begin{equation}\label{eq:x_i+-}
							x_{i+}^\eta{(V,Z)}:=\left(x_i+\eta\sqrt{2V},x^{-i}-Z^{-i}\right) \mbox{ and }
							x_{i-}^\eta{(V,Z)}:=\left(x_i-\eta\sqrt{2V},x^{-i}-Z^{-i}\right),
						\end{equation}
						respectively. {The next} proposition {provides a pathway for constructing the associated gradient estimator}.
						\begin{proposition}
							Suppose Assumption~\ref{asm:decdep} holds. Let $\phi$ and $\phi_\eta$ be as in \eqref{eq:phiu} and \eqref{eq:phietau}. Then for $1 \leq i \leq n$, the $i$-th partial derivative of $f_\eta$ is given by
							\begin{equation}\label{eq:eta-grad-2}
								\frac{\partial_i f_{\eta}(x)}{\partial x_i} \,  = \, \frac{1}{\eta\sqrt{2\pi}}\mathbb{E}_{V,Z,\bxi {\sim \tilde{p}}}\left[\tilde{F}(x_{i+}^\eta(V,Z),\bxi)-\tilde{F}(x_{i-}^\eta(V,Z),\bxi)\right],
							\end{equation}
							where $V \sim \mathcal{E}xp(1),Z \sim \mathcal{N}_n(0,\eta^2 I)$ and $\bxi$ follows distribution with density $\tilde{p}(\xi)$.
						\end{proposition}
						\begin{proof}
						By Proposition~\ref{lem:F-decdep1-Lip} and Proposition~\ref{cor:gradfeta} we get our desired result.
						\end{proof}}
					
					\medskip
					
					We may now define the the gradient estimator $\tilde{g}_\eta$ with the $i$-th component to be 
					\begin{align}\label{estidec1}
						\tilde{g}^i_\eta&\left(x,V,Z,\xi\right)=\tfrac{1}{\eta\sqrt{2\pi}}\left(\tilde{F}(x_{i+}^\eta(V,Z),\xi)-\tilde{F}(x_{i-}^\eta(V,Z),\xi)\right)\nonumber\\
						&\hspace{2em}\overset{\eqref{eq:tilde-F}}{=}\tfrac{1}{\eta\sqrt{2\pi}}\left(\frac{\hat{F}(x_{i+}^\eta(V,Z),\xi)p(\xi \mid x_{i+}^\eta(V,Z))}{\tilde{p}(\xi)}-\frac{\hat{F}(x_{i-}^\eta(V,Z),\xi)p(\xi \mid x_{i-}^\eta(V,Z))}{\tilde{p}(\xi)}\right).
					\end{align}
					In our proposed scheme, at each iteration, we first generate a {realization} $\xi$ from the user-specified density $\tilde{p}(\xi)$ and then generate  $x_{+}^{\eta}, x_{-}^{\eta}$ specified in \eqref{eq:x_i+-} where $V,Z$ are exponentially and normally distributed, respectively. These objects are then input to the {the stochastic first-order oracle  {\bf SFO$^{\rm dd}_a$} (defined in Definition~\ref{sfo-dda})},  which then gives us $\hat{F}(x_{i+}^\eta(V,Z),\xi)p(\xi \mid x_{i+}^\eta(V,Z))$ and $\hat{F}(x_{i-}^\eta(V,Z),\xi)p(\xi \mid x_{i-}^\eta(V,Z))$. Therefore, we may construct the {\bf esGs} gradient estimator $\tilde{g}^i_\eta\left(x,V,Z,\xi\right)$ at \eqref{estidec1}.\\
					
					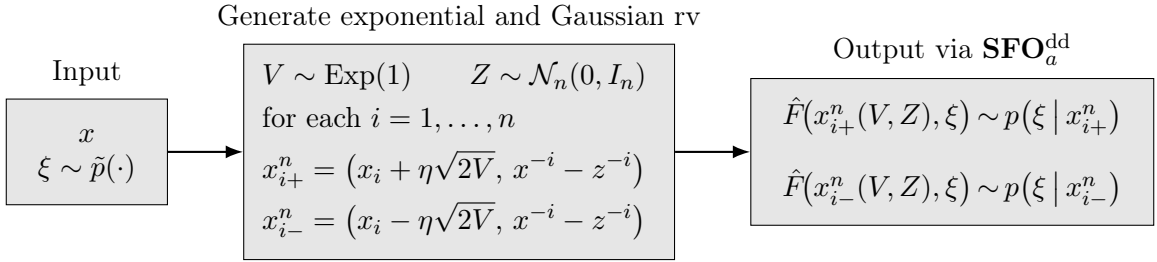
\begin{figure}[h!]
						\centering
						\begin{tikzpicture}[scale = 0.5,
							box/.style={draw, fill=gray!20, rectangle, align=center, inner sep=7pt},
							arrow/.style={-Latex, thick},
							lab/.style={font=\normalsize}
							]
							\node[box] (inp) {$
								\begin{array}{c}
									x\\ 
									\xi \sim \tilde p(\cdot)
								\end{array}
								$};
							\node[lab, above=1pt of inp] {Input};
							\node[box, right=1cm of inp] (orc) {%
								$
								\begin{aligned}
									&V \sim \mathrm{Exp}(1)
									\qquad
									Z \sim \mathcal{N}_n(0,I_n)\\
									&\text{for each } i=1,\ldots,n\\
									&x^{n}_{i+}=\bigl(x_i+\eta\sqrt{2V},\,x^{-i}-z^{-i}\bigr)\\
									&x^{n}_{i-}=\bigl(x_i-\eta\sqrt{2V},\,x^{-i}-z^{-i}\bigr)
								\end{aligned}
								$
							};
							\node[lab, above=1pt of orc] {{Generate exponential and Gaussian rv}};
							\node[box, right=1cm of orc] (out) {$
								\begin{array}{c}
									\hat F\!\bigl(x^{n}_{i+}(V,Z),\xi\bigr){\, \sim \,}
									p\!\left(\xi\,\middle|\,x^{n}_{i+}\right)\\
									\\
									\hat F\!\bigl(x^{n}_{i-}(V,Z),\xi\bigr){\, \sim \,}
									p\!\left(\xi\,\middle|\,x^{n}_{i-}\right)
								\end{array}
								$};
							\node[lab, above=1pt of out] {Output {via {\bf SFO$^{\rm dd}_a$}}};
							\draw[arrow] (inp.east) -- (orc.west);
							\draw[arrow] (orc.east) -- (out.west);
						\end{tikzpicture}
						\caption{Oracle schematic for decision dependent stochastic optimization with known struction of $p(\xi \mid x)$}
						\label{fig:oracle-1}
					\end{figure}

					From \eqref{eq:eta-grad-2}, we get 
					\beq\label{eq:unbiased-dd-1}
					\E_{V,Z, \bxi} [\tilde{g}_{\eta}(x,V,Z, \bxi)] = \nabla f_{\eta}(x).
					\eeq
					We can also derive an analogous second moment bound on the estimator, {which can also be seen to grow linearly with the dimension $n$, akin to the non-DD regime}.
					\begin{proposition}\label{prop:2mombd-DD1}
						Let Assumption~\ref{asm:decdep} holds.  Then for any $x$ and $\eta > 0$,
						\begin{equation*}
							\mathbb{E}_{V,Z,\bxi {\sim \tilde{p}}}\left[\|\tilde{g}_\eta(x,V,Z,\bxi)\|^2\right]\leq {\tfrac{4L_0^2 n}{\pi}}, \mbox{ where } {L_0 \triangleq {\sqrt{2 \lp M^2\hat{L}_f^2 +M M_f^2L_\xi^2 \rp}}},
						\end{equation*}
						{ and } $\tilde g_{\eta}$ is defined as in \eqref{estidec1}.
					\end{proposition}
					\begin{proof}
					Observe the non-decision-dependent representation of $f$ in \eqref{eq:dectonondec} and the Lipschitz coefficient $L_0 = {\sqrt{2 \lp M^2\hat{L}_f^2 +M M_f^2L_\xi^2 \rp}}$ of $\tilde{F}$ in Proposition~\ref{lem:F-decdep1-Lip}-(a). Consequently, from Proposition~\ref{prop:secmombd}, we obtain that
					\begin{equation*}
						\mathbb{E}_{V,Z,\bxi {\sim \tilde{p}}}\left[\|\tilde{g}_\eta(x,V,Z,\bxi)\|^2\right]\leq \tfrac{4L_0^2 n}{\pi},
					\end{equation*}
					which proves our result.
					\end{proof}

					\subsection{Unknown structure of $p(\xi \mid x)$.}\label{sub:p_unknown} 
					In this subsection, we consider the more common setting where the functional form of $p(\bullet \mid x)$ is unavailable. Yet, we proceed to show that similar guarantees can be obtained under some assumptions.  We use the same notation $x_{i+}^\eta{(V,Z)}$ and $x_{i-}^\eta{(V,Z)}$, as defined in \eqref{eq:x_i+-}. Note that in this setting, we do not introduce a density $\tilde{p}(\bullet)$.
					Our first result derives the gradient of $f_{\eta}$ when $p(\xi\mid x)$ is unavailable. 
					Contrary to our exposition in the prequel, we are going to assume existence of a random field ${\bxi} = \{\bxi_{x}, x \in \R^n\}$ which exhibit the marginals $\cd(x)$ or $p(\xi \mid x)$ (when density exists) for every $x \in \R^n$, in addition to a correlation structure we define below. To talk about expectations of functionals with respect to this random field, we will also employ the following notation for expectation of any function $u(x,y,\bxi_x,\bxi_y)$ with respect to $\bxi$, defined as 
					\begin{equation}\label{eq:Expxi12}
						\mathbb{E}_{\bxi_x, \bxi_y} \left[ \, u(x,y,\bxi_x,\bxi_y)\, \right] \, \triangleq \, \int_{\Xi(y)} \ \int_{\Xi(x)} \ u(x,y, \xi_1, \xi_2) p_{\bxi_x, \bxi_y}(\xi_1, \xi_2) d\xi_1 d\xi_2,
					\end{equation}
					where $p_{\bxi_x, \bxi_y}$ is the joint density of $(\bxi_x,\bxi_y)$, with marginals $\bxi_x \sim \mathcal{D}(x),\bxi_y \sim \mathcal{D}(y)$.
					Unless mentioned otherwise, in this subsection, the expectation operator is as defined above. We now outline some assumptions on $\bxi$ and $\hat{F}$, {based on which we may conclude}  the Lipschitzian property of $f$. 
					\begin{assumption}\label{asm:randomfield}
						$\hat{F}(\bullet,\xi_\bullet)$ is $L_0$-Lipschitz continuous in the $L^2(\bxi_\bullet)$ metric, i.e., there exists an $L_0$ such that the following holds for any $x, y \in \R^n$.
						$$
						{\lnn \hat{F}(x, \bxi_x) - \hat{F}(y, \bxi_y) \rnn}_{L^2(\bxi_\bullet)} \, \leq \, L_0 \|x-y\|.
						$$  
					\end{assumption}
					
					In this section, we assume existence of a stochastic first-order oracle {\bf SFO$^{\rm dd}_u$} corresponding to a decision-dependent regime with an unknown density $p(\xi\mid x)$.\\
					
					\begin{definition} [{\bf SFO$^{\rm dd}_u$}] \label{sfo-ddu} There exists a stochastic first-order oracle that given {a collection $x\in {\laa}$, produces $\{ \hat{F}(x,\xi_x); x \in {\laa} \}$},  where {$\{\xi_x ; x \in {\laa} \}$} is generated from an identical but independent random field that satisfies the Assumption~\ref{asm:randomfield} with the marginal distribution $\bxi_x \sim \mathcal{D}(x)$.  
					\end{definition}
					

				\begin{figure}[h]
					\centering
					\begin{tikzpicture}[scale = 0.5,
						box/.style={draw, rectangle, fill=gray!40, align=center, inner sep=7pt},
						arrow/.style={-Latex, thick},
						lab/.style={font=\normalsize}
						]
						\node[box] (inp) {$
							\begin{array}{c}
								x
							\end{array}
							$};
						\node[lab, above=1pt of inp] {Input};
						\node[box, right=1cm of inp] (orc) {%
							$
							\begin{aligned}
								&V \sim \mathrm{Exp}(1)
								\qquad
								Z \sim \mathcal{N}_n(0,I_n)\\
								&\\
								&\text{for each } i=1,\ldots,n\\
								&x^{n}_{i+}=\bigl(x_i+\eta\sqrt{2V},\,x^{-i}-z^{-i}\bigr)\\
								&x^{n}_{i-}=\bigl(x_i-\eta\sqrt{2V},\,x^{-i}-z^{-i}\bigr)\\
								&\\
								&{\xi_{x_{i+}^\eta} \, \sim \,  p_{\xi_{x_{i+}^\eta}}} \mbox{ and } {\xi_{x_{i-}^\eta}  \, \sim \,  p_{\xi_{x_{i-}^\eta}}}
							\end{aligned}
							$
						};
						\node[lab, above=1pt of orc] {Oracle};
						\node[box, right=1cm of orc] (out) {$
							\begin{array}{c}
								\hat{F} \left(x_{i+}^{\eta}, \xi_{x_{i+}^{\eta}} \right)\\
								\\
								\hat{F} \left(x_{i-}^{\eta}, \xi_{x_{i-}^{\eta}} \right)
							\end{array}
							$};
						\node[lab, above=1pt of out] {Output};
						\draw[arrow] (inp.east) -- (orc.west);
						\draw[arrow] (orc.east) -- (out.west);
					\end{tikzpicture}
					\caption{Oracle schematic for decision dependent stochastic optimization with unknown struction of $p(\xi \mid x)$}
					\label{fig:oracle-2}
				\end{figure}
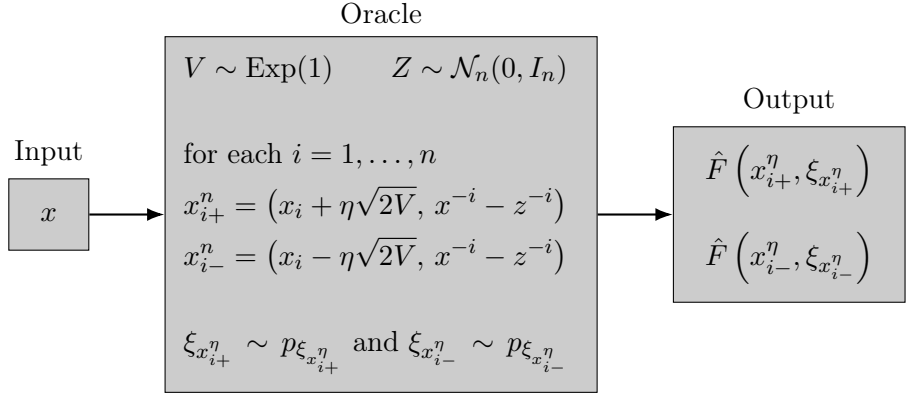

				\begin{remark}\label{ass:f-hat-Lip-DD2}
					Naturally, one might inquire if the prior Assumption ~\ref{asm:randomfield} is vacuous. We observe that under the following correlation and Lipschitzian requirements, this assumption indeed holds. 
					\begin{enumerate}
						\item\label{ass:DD2-xi} For any $x,y\in \R^n$, there exists a constant $c_{\xi}$ such that $\mathbb{E}[\|\bxi_x-\bxi_y\|^2] \, \leq \, c_{\xi} \|x-y\|^2$ for $\bxi_x \sim \mathcal{D}(x),\bxi_y \sim \mathcal{D}(y)$, i.e.,   $\bxi$ is $L_2$-Lipschitz. 
						\item $\hat{F}(x,\xi)$ is $L_\xi$-Lipschitz in $x$ and $\xi$.
					\end{enumerate}
					It is readily checked that if the above two requirements are satisfied then Assumption~\ref{asm:randomfield} will be satisfied with $L_0=L_\xi\sqrt{2+2c_\xi}$. 
					 
				\end{remark}
				\medskip
				Examples of $L_2$-Lipschitz random fields are aplenty. We provide two simple Gaussian examples. For further details please refer to \cite{gaussian-processes-ml}.\\
				\begin{enumerate}[wide, labelindent = 0pt, label = (\roman*)]
					\item {Let $\{\psi_k\}_{k=1}^n$ be deterministic continuous functions and $\{Z_k\}_{k=1}^n$ be independent standard normals}. {If $\{a_k\}_{k=1}^n$ denotes a set of scalars, then }
					$$\bxi_x = \sum_{k=1}^n a_k Z_k \psi_k(x)$$
					is a centered Gaussian random field satisfying the $L_2$-Lipschitz property.
					\item Let $W$ be Gaussian white noise on $\R^d$. Then for a kernel $\phi \in L^2$ with bounded gradient, the Gaussian convolutional random field
					$$\bxi_x = \int \phi(x-u) W(du)$$
					satisfies the $L_2$-Lipschitz property.
				\end{enumerate}
				
				As one may notice, $\bxi_x,\bxi_y$ in Assumption~\ref{asm:randomfield} or Remark~\ref{ass:f-hat-Lip-DD2} cannot be independent. Assumption~\ref{asm:randomfield}  suffices to guarantee the Lipschitz continuity of the function $f$. 
				
				\begin{lemma}\label{lem:f-decdep2-Lip}
					Suppose Assumption~\ref{asm:randomfield} holds. Then $f$, defined in \eqref{eq:DDproblem}, is $L_0$-Lipschitz continuous. 
				\end{lemma}
				\begin{proof} 
				Recall from \eqref{eq:DDproblem} that $f(x)=\int_{\Xi(x)} \hat{F}(x, \xi)p(\xi \mid  x) d\xi$. Consequently, by Jensen's inequality, for any $x, y$, we have  
				\begin{align*}
					|\, f(x)-&f(y)\, |\, = \, \left|\, \mathbb{E}_{\bxi_x, \bxi_y}\left[\, \hat{F}(x,\bxi_x)-\hat{F}(y,\bxi_y)\, \right]\, \right|\\
					&\, \leq \, \sqrt{\E_{\bxi_x,\bxi_y}\left[\, \left| \,\hat{F}(x,\bxi_x)-\hat{F}(y,\bxi_y)\, \right|^2 \, \right]}\, =\, {\lnn \hat{F}(x, \bxi_x) - \hat{F}(y, \bxi_y) \rnn}_{L^2(\bxi_\bullet)} \leq L_0 \|x-y\|.
				\end{align*}
				\end{proof}
				This result allows us to claim that Proposition ~\ref{prop:gradfeta} and Lemma~\ref{Lemma:SmoothProperties} hold true with Lipschitz constant $L_0$. 
				\begin{proposition}\label{prop:DDgradfeta}
					Suppose Assumption~\ref{asm:randomfield} holds. Let $\phi$ and $\phi_\eta$ be as in \eqref{eq:phiu} and \eqref{eq:phietau}. 
					Then for $1 \leq i \leq n$, the $i$-th partial derivative of $f_\eta$ is given by 
					\begin{equation}\label{eq:DDGradientfetaD}
						\tfrac{\partial_i f_{\eta}(x)}{\partial x_i} \,  = \, \tfrac{1}{\eta\sqrt{2\pi}}\mathbb{E}_{{V,Z,\bxi_{x_{i+}^\eta},\bxi_{x_{i-}^\eta}}}\left[\hat{F}\left(x_{i+}^\eta(V,Z), \bxi_{x_{i+}^\eta}\right)-\hat{F}\left(x_{i-}^\eta(V,Z), \bxi_{x_{i-}^\eta}\right)\right],
					\end{equation}
					where the expectation taken over $\bxi_{x_{i+}^\eta},\bxi_{x_{i-}^\eta}$ is in the sense of \eqref{eq:Expxi12} with marginals $ \mathcal{D}(x_{i+}^\eta(V,Z))$ and $\mathcal{D}(x_{i-}^\eta(V,Z))$, respectively. Additionally, $V \sim \mathcal{E}xp(1)$ and $Z\sim \mathcal{N}_n(0,\eta^2 I)$.
				\end{proposition}
				\begin{proof}
				By invoking Lemma~\ref{lem:f-decdep2-Lip}, we have that $f$ is $L_0$-Lipschitz continuous. Therefore. we may apply Proposition~\ref{prop:gradfeta}. By \eqref{eq:GradientfetaD} and using the representation of $f$ in \eqref{eq:DDproblem} it is immediate that
				\begin{align*}
					\tfrac{\partial_i f_{\eta}(x)}{\partial x_i}&=
					\, \tfrac{1}{\eta \sqrt{2\pi}}\mathbb{E}_{V,Z}\left[f\left(x_{i+}^{\eta}(V,Z)\right)-f\left(x_{i-}^{\eta}(V,Z) \right)\right] \\
					&=\tfrac{1}{\eta\sqrt{2\pi}}\mathbb{E}_{V,Z,\bxi_{x_{i+}^\eta},\bxi_{x_{i-}^\eta}}\left[\hat{F}\left(x_{i+}^\eta(V,Z), \bxi_{x_{i+}^\eta}\right)-\hat{F}\left(x_{i-}^\eta(V,Z), \bxi_{x_{i-}^\eta}\right)\right].
				\end{align*}	
				\end{proof}
				
				Denote $\bxi_{x_+^\eta}=(\bxi_{x_{i+}^\eta})_{1\leq i\leq n}$ and $\bxi_{x_-^\eta}=(\bxi_{x_{i-}^\eta})_{1\leq i\leq n}$. Let us consider a new gradient estimator $\tilde{g}_{\eta}$, where the $i$-th component $\tilde{g}_{\eta}^i$ is defined as 
				\begin{equation*}
					\tilde g_{\eta}^i \left(x,v,z,\xi_1,\xi_2\right) \, = \,\tfrac{1}{\eta\sqrt{2\pi}}\left(\hat{F}\left(x_{i+}^\eta(v,z), \xi_1\right)-\hat{F}\left(x_{i-}^\eta(v,z), \xi_2\right)\right)
				\end{equation*}
				and $\xi_1$ and $\xi_2$ denote realizations of $\bxi_{x_+^\eta}$ and $\bxi_{x_-^\eta}$, respectively. As a result of Proposition~\ref{prop:DDgradfeta}, $\tilde{g}_{\eta}$ is an unbiased estimate of $\nabla f_{\eta}$, as qualified in: 
				\beq\label{eq:unbiasedd-dd-2}
				\E_{V,Z,\bxi_{x_+^\eta},\bxi_{x_-^\eta}}  \left[\, \tilde{g}_{\eta}(x,V,Z,\bxi_{x_{+}}, \bxi_{x_{-}}) \, \right]  \, = \, \nabla f_{\eta}(x).
				\eeq
				This then allows for deriving an analogous second moment bound on the estimator {with a similar dependence (linear) on $n$ as in the prior two estimators.}
				\begin{proposition}\label{prop:DDsecmombd}
					Suppose Assumption~\ref{asm:randomfield} holds. Then for any $x\in \R^n$ and $\eta > 0$,
					\begin{equation}
						\mathbb{E}\left[\, \|\tilde{g}_\eta(x,V,Z,\bxi_{x_{+}^{\eta}},\bxi_{x_{-}^{\eta}})\|^2\, \right] \, \leq \, \tfrac{4L_0^2n}{\pi}.
					\end{equation}
				\end{proposition}
				\begin{proof}
				Consider an $L_\xi$-Lipschitz continuous $\hat{F}$. Then 
				\begin{align*}
					& \quad \mathbb{E}_{V,Z,\bxi_{x_{i+}^\eta},\bxi_{x_{i-}^\eta}}\left[\left|\, \tilde g_{\eta}^i\left(x,V,Z,\bxi_{x_{i+}^\eta},\bxi_{x_{i-}^\eta}\right)\right|^2\right] \\
					& = \tfrac{1}{\eta^2 2\pi}\mathbb{E}_{V,Z,\bxi_{x_{i+}^\eta},\bxi_{x_{i-}^\eta}}\left[\left|\hat{F}\left(x_{i+}^\eta(V,Z), \bxi_{x_{i+}^\eta}\right)-\hat{F}\left(x_{i-}^\eta(V,Z), \bxi_{x_{i-}^\eta}\right)\right|^2\right]\\
					&= \tfrac{1}{\eta^2 2\pi} \mathbb{E}_{V,Z} \left\|\hat{F}\left(x_{i+}^\eta(V,Z), \bxi_{x_{i+}^\eta}\right)-\hat{F}\left(x_{i-}^\eta(V,Z), \bxi_{x_{i-}^\eta}\right)\right\|_{L^2(\bxi_\bullet)}^2.
				\end{align*}
				Then by Assumption~\ref{asm:randomfield} we have
				$$
				\mathbb{E}_{V,Z,\bxi_{x_{i+}^\eta},\bxi_{x_{i-}^\eta}}\left[\left|\, \tilde g_{\eta}^i\left(x,V,Z,\bxi_{x_{i+}^\eta},\bxi_{x_{i-}^\eta}\right)\right|^2\right]
				\leq \tfrac{1}{\eta^2 2\pi}L_0^2 \mathbb{E}_{V,Z} \left[\, \left|x_{i+}^\eta(V,Z)-x_{i-}^\eta(V,Z)\right|^2 \, \right] =\frac{4L_0^2}{\pi}.
				$$
				Then the result follows by noting that $\mathbb{E}[\|\tilde{g}_{\eta}(x,V,Z,\bxi_{x_{i+}^\eta},\bxi_{x_{i-}^\eta})\|^2]=\sum_{i=1}^n \mathbb{E}[|\tilde{g}_{\eta}^i(x,V,Z,\bxi_{x_{i+}^\eta},\bxi_{x_{i-}^\eta})|^2]$.
				
				\end{proof}
				
				\begin{remark}
					The need for Assumption~\ref{asm:randomfield}.\ref{ass:DD2-xi} is clear from the proof of Proposition~\ref{prop:DDsecmombd}. Simply speaking, one requires that $\hat{F}(x_{+}^{\eta}(V,Z),\bxi_{x_{+}})$ and $\hat{F}(x_{-}^{\eta}(V,Z),\bxi_{x_{-}})$ be correlated in order to bound the second moment of $\tilde g_{\eta}$. One way to achieve that is by imposing Assumption~\ref{asm:randomfield}.\ref{ass:DD2-xi} on the random field $\{\bxi_x; x \in \R^n\}$. Note that in Section~\ref{sub:p_known}, such dependence is also needed but created artificially between $\tilde F(x_{+}^{\eta}(V,Z), \bxi)$ and $\tilde F(x_{-}^{\eta}(V,Z), \bxi)$ as evident in the definition of $\tilde g_{\eta}$ in \eqref{estidec1}. 
				\end{remark}
				
				\subsection{Optima, stationary points, and performatively stable points}
				Early work in decision-dependent stochastic optimization problems, as outlined in ~\cite{drusvyatskiy2023stochastic, perdomo2020performative}, focused on an examination of a  ``performative stable point'' or an ``equilibrium'' point.  In general, such a point is distinct from either an optimal solution or a stationary point.  We recap the three distinct  definitions.
				
				\medskip
				
				\begin{definition}\label{def:eq}
					Consider the decision-dependent problem \eqref{eq:DDproblem}. 
					
					\noindent (i) $x^{\rm opt}$ is an optimal solution of \eqref{eq:DDproblem}, if 
					\begin{align}
						f(x^{\rm opt}) \, \le \, f(x) = \mathbb{E}_{\bxi \sim {\mathcal{D}}(x)} \left[ F(x,\bxi) \right], \quad \forall x \, \in \, X.
					\end{align}
					\noindent (ii) If $f$ is C$^1$ on an open neighborhood containing $X$, then $x^{\rm stat}$ is a stationary point of \eqref{eq:DDproblem} if 
					\begin{align}
						\nabla_x f(x)^\top (y - x) \, \ge \, 0, \quad \forall y \, \in \, X.
					\end{align}
					\noindent (iii)  $x^{\rm ps} \in X$ is called a performatively stable point  of  \eqref{eq:DDproblem} if
					\begin{equation}
						x^{\rm ps} \, \in \, {\displaystyle \text{arg}\hspace{-0.01in}\min_{x \, \in \, X}} \, \mathbb{E}_{\bxi \sim \mathcal{D}({x^{\rm ps}})} \left[\hat{F}(x, \bxi)\right]. 
					\end{equation} 
				\end{definition}
				{It may be noted that a performatively stable point, denoted by}  $x^{\rm ps}$, satisfies a fixed-point property by definition. While the relationship between optima and stationary points is well studied in the optimization literature, we point out two other relationships, one of which is less obvious, showing that if the distribution function satisfies a suitably defined ${\beta}$-sensitivity with respect to $x$, then under suitable strong convexity and Lipschitzian requirements, then $\|x^{\rm opt}-x^{\rm ps}\| = \mathcal{O}({\beta})$. {However, such a parameter $\beta$ is idiosyncratic to the problem setting, suggesting that $x^{\rm opt}$ and $x^{\rm ps}$ need not be close to zero.  } 
				
				\begin{lemma} Consider the decision-dependent problem \eqref{prob-dd}. Then the following hold.
					
					\noindent (a) $f(x^{\rm opt}) \, \le \, f(x^{\rm ps}).$
					
					\noindent (b) Suppose $F(x,\xi)$ is $\gamma$-strongly convex in $x$ for any $\xi$ and $L_{\xi}$-Lipschitz in $\xi$ for any $x$. Further, suppose ${\mathcal{D}}(\cdot)$ is ${\beta}$-sensitive, i.e., if $W_1({\mathcal{D}}(x),{\mathcal{D}}(x^\prime)) \le {\beta} \| x-x^\prime\|$ for any $x,x^{\prime} \in X$. Then for any $x^{\rm opt}$ and any $x^{\rm ps}$, 
					$$\| x^{\rm opt}-x^{\rm ps} \| \, \le \, \frac{2L_{\xi} {\beta}}{\gamma}.$$  
				\end{lemma}
				\begin{proof} (i) Omitted; (ii) See \cite[Theorem 4.3]{perdomo2020performative}.
				\end{proof}
				
				\medskip
				
				More generally, absent strong assumptions on the distribution and the stringent convexity requirements, a performatively stable solution $x^{\rm ps}$ may be significantly poorer than the optimal solution $x^{\rm opt}$ and may 
				have no corresponding stationarity guarantee. Consequently, the definition of an optimal point requires careful consideration.
				Observe that
				\begin{align}
					f_\eta(x) & =\int f(y) \phi_\eta(x-y) \, d y  =\int \mathbb{E}\left[\hat{F}\left(y, \bxi_y\right)\right] \phi_\eta(x-y) \, d y \nonumber\\
					& \stackrel{\text{Fubini}}{=}\underbrace{\mathbb{E}\left[\int \hat{F}\left(y, \bxi_y\right) \phi_\eta(x-y) \, d y\right]}_{\, =\mathbb{E}_{\bxi {\sim {\mathcal{D}}(y)}} \left[\bar{F}_\eta(x, \bxi)\right]} \label{eq:DDeqf1}\\
					& =\mathbb{E}\left[\int \hat{F}\left(x-y, \bxi_{x-y}\right) \phi_\eta(y) \, d y\right] \nonumber
					=\mathbb{E}_{\bxi}\left[\mathbb{E}_{Z_\eta}\left[\hat{F}\left(x-Z_\eta, \bxi_{x-Z_\eta}\right)\right]\right]\notag \\
					& =\mathbb{E}_{{\bxi \sim {\mathcal{D}}(x-Z_{\eta}), Z_{\eta}}}\left[\hat{F}\left(x-Z_\eta, \bxi\right)\right], \label{eq:DDeqf2}
				\end{align}
				where $Z_\eta$ denotes a random variable with probability density function $\phi_\eta$. Observe that this representation reveals that the decision dependence is no longer local (i.e. at $x$) but is instead  spread over the feasible region through $x-Z_{\eta}$.
				This necessitates a new definition of equilibrium for the smoothed problem, provided next.  \\
				
				\begin{definition}\label{def:etasmeq}
					$\bar{x}_\eta$ is a $\eta$-smoothed {performatively stable point} of the decision-dependent problem \eqref{eq:DDproblem} if
					\begin{equation*}
						\bar{x}_\eta  \, {\in} \, \arg\min_{x\in X}\E_{\bxi}\left[\E_{Z_{\eta}}\left[\hat{F}\left(x-Z_\eta,\bxi_{\bar{x}_\eta-Z_\eta}\right)\right]\right]. 
					\end{equation*} 
					
				\end{definition}
				
				\medskip
				
				A natural question is whether the $\eta$-smoothed {perfomatively stable point} is an approximate {performatively stable point} for the original problem. We show below that this is indeed the case with the following enhanced version of Assumption~\ref{asm:randomfield}.
				\begin{proposition}
					Suppose $\varepsilon > 0$ and there exists an $\bar{L}$ such that the following holds for any $x, y \in \R^n$.
					\begin{equation}\label{eq:asmenhance}
						{\lnn \hat{F}(x_1, \bxi_{x_2}) - \hat{F}(y_1, \bxi_{y_2}) \rnn}_{L^2(\bxi_\bullet)} \, \leq \, \frac{\bar{L}}{2} \left(\|x_1-y_1\|+\|x_2-y_2\|\right).
					\end{equation} Let $\bar{x}_\eta$ be an $\eta$-smoothed {performatively stable point} as in Definition~\ref{def:etasmeq} with $\eta \leq \tfrac{\varepsilon}{2\bar{L}\sqrt{n+1}}$. Then $\bar{x}_\eta$ is an $\varepsilon$-{performatively stable point} of the original problem $\min f(x)= \mathbb{E}_{\bxi \sim \mathcal{D}(x)} [\hat{F}(x, \bxi)]$.
				\end{proposition}
				\begin{proof}
				Observe that by Lemma~\ref{Lemma:SmoothProperties} \eqref{Lemma:c} for any $x \in X$ and by definition of the smoothing, we have
				\begin{align}\label{eq:fxbareta}
					\notag	f(\bar{x}_\eta) & = \mathbb{E}_{\bxi} \left[\hat{F}(\bar{x}_\eta, \bxi_{\bar{x}_\eta})\right] \,\leq \, f_\eta(\bar{x}_\eta)+\bar{L}\eta\sqrt{n+1}\\
					& =\E_{\bxi}\left[\E_{Z_\eta}\left[\hat{F}\left(\bar{x}_\eta-Z_\eta, \bxi_{\bar{x}_\eta-Z_\eta}\right)\right] \right]+\bar{L}\eta \sqrt{n+1}.
				\end{align}
				By definition of $\bar x_{\eta}$ and Fubini's theorem, the following holds for any $x \in X$,	
				\begin{equation}\label{eq:xbareta}
					\E_{\bxi}\left[\E_{Z_\eta}\left[\hat{F}\left(\bar{x}_\eta-Z_\eta, \bxi_{\bar{x}_\eta-Z_\eta}\right)\right] \right]\leq \E_{Z_\eta}\left[\E_{\bxi}\left[\hat{F}\left(x-Z_\eta, \bxi_{\bar{x}_\eta-Z_\eta}\right)\right] \right].
				\end{equation}	
				Furthermore, by Jensen's inequality we have for any $x \in X$
				\begin{align*}
					& \left|\E_{Z_\eta}\left[\E_{\bxi}\left[\hat{F}\left(x-Z_\eta, \bxi_{\bar{x}_\eta-Z_\eta}\right)\right]\right]-\E_{\bxi}\left[\hat{F}(x,\bxi_{\bar{x}_\eta})\right]\right|  \\
					&~~~~ \leq \E_{Z_\eta}\left[\, \left\|\hat{F}\left(x-Z_\eta, \bxi_{\bar{x}_\eta-Z_\eta}\right)-\hat{F}(x,\bxi_{\bar{x}_\eta})\right\|_{L^2(\bxi_{\bullet})} \, \right] \\
					&~~~~ \overset{\eqref{eq:asmenhance}}{\leq} \bar{L}\E_{Z_\eta}\left\|Z_\eta\right\|.
				\end{align*}
				Using the above inequality and following same argument as in the proof of Lemma~\ref{Lemma:SmoothProperties} \eqref{Lemma:c} we have 
				\begin{equation}\label{eq:xbaretaJen}
					\E_{Z_\eta}\left[\E_{\bxi}\left[\hat{F}\left(x-Z_\eta, \bxi_{\bar{x}_\eta-Z_\eta}\right)\right]\right]\leq\E_{\bxi}\left[\hat{F}(x,\bxi_{\bar{x}_\eta})\right]+\bar{L}\eta \sqrt{n+1}.
				\end{equation}
				Combining \eqref{eq:fxbareta}, \eqref{eq:xbareta} and \eqref{eq:xbaretaJen}, the following holds for any $x \in X$,
			$$
			f(\bar{x}_\eta)\leq\E_{\bxi}\left[F(x,\bxi_{\bar{x}_\eta})\right]+2\bar{L}\eta \sqrt{n+1}=\E_{\bxi\sim{\cd(\bar{x}_\eta})}\left[\hat{F}(x,\bxi)\right]+2\bar{L}\eta \sqrt{n+1},
			$$
			proving the result.
			\end{proof}
			
			{In contrast with some of the prior work in decision-dependent stochastic optimization~\cite{drusvyatskiy2023stochastic},  our focus lies on computing either stationary points or optima of the original problem, as observed in the subsequent sections. We conclude this section with an example where we observe that performatively stable points can indeed depart significantly from optima.}

			\subsection{Example}\label{subsec:marketproblem}
			In this section, we present an example that demonstrates how the assumptions outlined in Sections~\ref{sub:p_known} and~\ref{sub:p_unknown} can be met. This example also serves as an illustration of the distinction between an optimal point and an equilibrium point.  For the first case with known structure of $p(\xi | x)$ (Sec.~\ref{sub:p_known}), the KL-divergence needs to satisfy Assumption~\ref{asm:decdep} (2).
			For the second case with unknown structure of $p(\xi | x)$ (Sec.~\ref{sub:p_unknown}), the assumption was imposed on the oracle itself. One way to satisfy Assumption~\ref{asm:randomfield} is by assuming the oracle can give us two correlated random variables $\bxi_x$ and $\bxi_y$, satisfying the correlation
			\begin{equation}\label{eq:rho}
				\rho({x,y})  = \frac{\mathbb{E}[\bxi_x\bxi_y]-\mathbb{E}[\bxi_x]\mathbb{E}[\bxi_y]}{\sqrt{\textnormal{Var}[\bxi_x]}\sqrt{\textnormal{Var}[\bxi_y]}}
				\geq \frac{\frac{1}{2}\left(\mathbb{E}[\bxi_x^2]+\mathbb{E}[\bxi_y^2]-c_\xi{\|x-y\|}^2\right)-\mathbb{E}[\bxi_x]\mathbb{E}[\bxi_y]}{\sqrt{\textnormal{Var}[\bxi_x]}\sqrt{\textnormal{Var}[\bxi_y]}},
			\end{equation} 
			{where the inequality is a result of employing $\mathbb{E} [ \|\bxi_x - \bxi_y\|^2 ] \le  c_\xi\|x-y\|^2.$} In this case Remark~\ref{ass:f-hat-Lip-DD2}~\eqref{ass:DD2-xi} is satisfied. 
			We now consider a specific example from \cite{wang2023constrained}.\\
			
			\noindent \underline{Market Problem with Normal Noise.} 
			Consider the market problem in \cite{wang2023constrained}, given by the following optimization problem.
			\begin{equation}\label{eq:mkproblem}
				\min_{x_1, x_2} \quad -x_1\left(\mathbb{E}\left[\bzeta_1\left(x_1\right)\right]-a_1 x_1\right)-x_2\left(\mathbb{E}\left[\bzeta_2\right]-a_2 x_2\right).
			\end{equation}
			where ${\bzeta_1(x_1)} \sim \mathcal{N}(a+{\beta} x_1, \sigma^2)$ and ${\bzeta_2} \sim \mathcal{U}[L_2, R_2]$. Plugging in the expectation, we have the explicit form of the problem
			\begin{align}\label{eq:mkproblem2}
				\min_{x_1,x_2}  \quad &-x_1\left(a+{\beta} x_1-a_1 x_1\right)-x_2\left((L_2+R_2)/2-a_2 x_2\right)\nonumber\\
				{\, \equiv \, }  \min_{x_1, x_2} \quad & \ {\begin{pmatrix} x_1 \\ x_2 
					\end{pmatrix}^\top} \begin{pmatrix}
					a_1-{\beta} & 0\\
					0 & a_2
				\end{pmatrix}\begin{pmatrix}
					x_1\\ x_2
				\end{pmatrix}-\begin{pmatrix}
					a \\ (L_2+R_2)/2
				\end{pmatrix}^\top \begin{pmatrix}
					x_1\\ x_2
				\end{pmatrix},
			\end{align}
			which is a strongly convex optimization problem {if $a_1 - \beta > 0$ and $a_2 >0$}. Thus it has an unique optimal point. 
			
			\noindent {\em Case 1: Known density $p(\xi \mid x)$.} Here, it is easy to check that the density of $\zeta_1$ is Lipschitz in $x_1$ since the first derivative is bounded. Moreover, {by noting that $\sigma_x = \sigma_y$,} the KL-divergence of the normal distribution is 
			\begin{equation*}
				\frac{1}{2}\left(\left(\frac{\sigma_y}{\sigma_x}\right)^2+\frac{\left(\mu_x-\mu_y\right)^2}{\sigma_x^2}-1+\ln \left(\frac{\sigma_x^2}{\sigma_y^2}\right)\right)=\frac{{\beta}^2 (x-y)^2}{2\sigma^2},
			\end{equation*}
			which satisfies {Assumption~\ref{asm:decdep}~(\ref{ass:p-x-y-dist}).}  We {may} choose $\tilde{p}(\xi)$ to be the zero mean normal distribution with variance $\sigma^2$, same as the variance of $\zeta_1$. In this case the ratio $\frac{p(\xi\mid x)}{\tilde{p}(\xi)}$ will stay in a reasonable range.
			
			\noindent {\em Case 2: Unknown density $p(\xi \mid x)$.} In the second case, from \eqref{eq:rho} we require 
			\begin{align*}
				\rho({x,y}) &\geq \frac{\frac{1}{2}\left(\mathbb{E}[{\bxi_x}^2]+\mathbb{E}[{\bxi_y^2}]-c_\xi{\|x-y\|}^2\right)-\mathbb{E}[\bxi_x]\mathbb{E}[\bxi_y]}{\sqrt{\mathbb{E}[\bxi_x^2]-\left(\mathbb{E}[\bxi_x]\right)^2}\sqrt{\mathbb{E}[\bxi_y^2]-\left(\mathbb{E}[\bxi_y]\right)^2}}\\
				&= \frac{1}{\sigma^2}\left[\frac{1}{2} \left((a+{\beta} x)^2+\sigma^2+(a+{\beta} y)^2+\sigma^2\right)-\frac{c_\xi}{2}{\|x-y\|}^2-(a+{\beta} x)(a+{\beta} y)\right]\\
				&=1-\frac{c_\xi-{\beta}^2}{2\sigma^2}(x-y)^2.
			\end{align*}
			With carefully chosen constants, the right hand side of the above inequality is always less than $1$. When applying optimization algorithm in this case, one can ask the oracle to simulate correlated random variables with correlation $\rho=\max(1-\frac{c_\xi-{\beta}^2}{2\sigma^2}(x^+-x^-)^2,-1)$. 
			
			
			\begin{remark}
				In the {quadratic} problem \eqref{eq:mkproblem2},  the unique optimal point is given by  
				$$x^*={\begin{pmatrix} \tfrac{a}{2(a_1-{\beta})} \\ \tfrac{L_2+R_2}{4} \end{pmatrix}}.$$ 
				However, there exist an unique performatively stable solution  of \eqref{eq:mkproblem} with our setting. Indeed,
				\begin{align*}
					(\bar{x}_1,\bar{x}_2) &= \text{arg} \min_{x_1,x_2} \left(-x_1\left(\mathbb{E}\left[\zeta_1\left(\bar{x}_1\right)\right]-a_1 x_1-x_2\left(\mathbb{E}\left[\bzeta_2\right]-a_2 x_2\right)\right) \right)\\
					\implies 	(\bar{x}_1,\bar{x}_2)	& = \begin{pmatrix}
						\tfrac{a}{2a_1-{\beta}} \\ \tfrac{L_2+R_2}{4}
					\end{pmatrix}.
				\end{align*}
				Observe that $x^*$ and $\bar{x}$ may differ significantly unless $\beta = 0$. Our schemes are guaranteed to converge to the optimal point $x^*_1$ under some conditions. 
			\end{remark}
			\section{{Nonsmooth convex optimization via stochastic zeroth-order methods} }\label{sec:conv_theorems} 
			In this section, we present a unified  stochastic zeroth-order framework enabled by {\bf esGs} smoothing that can contend with {the standard non-decision-dependent setting as well as the decision-dependent regime}. Recall the function $f$ and the optimization problem given by either \eqref{eq:f=EF} (in the standard setting) or \eqref{eq:DDproblem} (in the decision-dependent setting). In this section, we require that $f$ be convex.
			\begin{assumption}\label{asm:covstoc2}
				\begin{enumerate}[label=\arabic*., leftmargin=*]
					\item[]
					\item\label{asm:cvx1}
					\begin{enumerate}[label=(\alph*), leftmargin=*, nosep]
						\item\label{asm:cvx1a} In the non-decision dependent setting of Section~\ref{sec:esGsestimator}, in addition to Assumption~\ref{asm:ascvgstochastic}, we assume that $F(\bullet,\xi)$ is a convex function for every $\xi \in \Xi$.
						\item\label{asm:cvx1b} In the decision-dependent setting of Section~\ref{sec:Decision_Dep}, in addition to Assumption~\ref{asm:decdep} or Assumption~\ref{asm:randomfield}, we assume $f$ to be convex.
					\end{enumerate}
					\item\label{asm:covstoc3}  $X \subseteq \mathbb{R}^n$ is a convex set. 
				\end{enumerate}
			\end{assumption}
			Note that under Assumption~\ref{asm:covstoc2}.\ref{asm:cvx1}\ref{asm:cvx1a} (Assumption~\ref{asm:covstoc2}.\ref{asm:cvx1}\ref{asm:cvx1b}, resp.), by Remark~\ref{rem:EF-Lip} (Lemmas~\ref{lem:F-decdep1-Lip} or~\ref{lem:f-decdep2-Lip}, resp.), the function $f$ is $L_0$-Lipschitz continuous. 
			Recall the gradient estimator $\tilde{g}_{\eta}$ which is unbiased due to the previously derived relations \eqref{eq:tilde-g-unbiased}, \eqref{eq:unbiased-dd-1} and \eqref{eq:unbiasedd-dd-2}.
			Moreover, for the non-decision-dependent problem and decision-dependent problem with known structure of $p(\xi | x)$, Proposition~\ref{prop:secmombd} (Propositions~\ref{prop:2mombd-DD1}, resp.) implies that
			\begin{equation*}
				\mathbb{E}_{V, Z, \bxi} \lc \tilde{g}_{\eta}(x, V, Z, \bxi) \rc \leq \tfrac{4L_0^2 n}{\pi}.
			\end{equation*}
			While in the decision-dependent problem with unknown structure of $p(\xi | x)$, Proposition~\ref{prop:DDsecmombd} yields that
			\begin{equation*}
				\mathbb{E}_{V, Z, \bxi_{x_{+}},\bxi_{x_{-}}} \lc \|\tilde{g}_{\eta}(x,V,Z,\bxi_{x_{+}},\bxi_{x_{-}})\|^2 \rc \leq \tfrac{4L_0^2n}{\pi}.
			\end{equation*}
			For the purpose of writing, in the following of this paper we will abuse the notation of $\tilde{g}_{\eta}(x, V, Z, \bxi)$ for all cases even when $\tilde{g}_\eta(x,V,Z,\bxi_{x_{+}},\bxi_{x_{-}})$ would be more precise. We have provided a summary of our guarantees obtained for nonsmooth convex optimization. 
			
			\begin{table}[htb]
				\scriptsize
				\centering
				\caption{}{{Complexity guarantees} in convex {settings}}
				\renewcommand{\arraystretch}{2}
				\begin{tabular}{ C{3cm} c C{2cm}  C{1.8cm} C{1.8cm} C{2.7cm} }
					\hline
					Parameters&Convergence metric&  Convergence rate &Iteration complexity&Improvement {in iteration complexity} &Assumptions addition to Asm~\ref{asm:covstoc2} \\
					\hline
					$\gamma_k=\eta_k=\tfrac{1}{\sqrt{n(k+1)}}$,  $\bar{x}_{K} \triangleq \tfrac{\sum_{k=0}^{{K-1}} \gamma_{k} x_{k}}{\sum_{k=0}^{{K-1}}\gamma_{k}}$  & $\mathbb{E}\left[f\left(\bar{x}_{K}\right)-f^{*}\right]$  & $\mathcal{O}\left(\tfrac{\sqrt{n} \ln(K)}{\sqrt{K}}\right)$&$\mathcal{O}(n\varepsilon^{-2})$&${\mathcal{O}(n)}$&-\\
					\hline
					$\gamma_k=\eta_k=1/(k+1)^{1/2}$ &  $ f(\bar{x}_K)-f(x^*) \quad a.s.$& $\mathcal{O}\left(\frac{n\ln (K)}{\sqrt{K}}\right) $&$\mathcal{O}(n^2\varepsilon^{-2})$&-&$X$ is bounded \\
					\hline
					$ \gamma_{k}=\tfrac{\theta}{k} $, where $ \theta>1 /  \mu $   & $\mathbb{E}\left[\left\|x_{k}-x^{*}\right\|^{2}\right] $ &$ \mathcal{O}\left(\tfrac{n}{k}\right)$&$\mathcal{O}(n\varepsilon^{-1})$&${\mathcal{O}(n)}$&Asm~\ref{asm:covstoc4}, $f$ is $\mu$-strongly convex\\
					\hline
				\end{tabular}
				\vspace{0.1em}
				
				{\textit{Remark:} Sample complexity scales as $n$ times the iteration complexity.}
			\end{table}
			
			\subsection{Zeroth-order stochasic gradient descent.}\label{sec:zo-method}
			If $\Pi_X[u]$ denotes the Euclidean projection of  $u$ onto the set $X$, 
			consider the following update rule for generating $x_{k+1}$, given $x_0 \in X$, for any $k \, \ge \, 0$.
			\begin{equation}\label{algstoc}
				x_{k+1}\, =\, \Pi_{X}\left[\, x_{k}-\gamma_{k}\,
				\tilde{g}_{\eta_k}(x_k, V_k, Z_k, \bxi_k)\, \right]\, 
			\end{equation}
			where $\eta_k > 0$, $\tilde{g}_{\eta_k}$ denotes the gradient estimator in either the standard or the decision-dependent setting, and $f_{\eta}$ is defined in \eqref{eq:feta} Here $\{V_k\}_{k=0}^N$ are i.i.d $\mathcal{E}xp(1)$ and $\{Z_k\}_{k=0}^N$ are i.i.d. $\mathcal{N}_n(0,\eta^2I)$ with realizations $\{v_k\}$ and $\{z_k\}$, respectively.
			Denote by $\cf_k$ the filtration generated by the initial iterate $x_0$ and the random variates $\{V_i, Z_i, \bxi_i\}_{i=1}^{k-1}$. Note that with this definition of $\cf_k$, $x_k \in \cf_k$. For brevity, in the following we will simply use $\tilde{g}_{\eta_k}(x_k)$ to stand for the random variable $\tilde{g}_{\eta_k}(x_k,V_k,Z_k,\bxi_k)$.
			From \eqref{eq:tilde-g-unbiased}, \eqref{eq:unbiased-dd-1} and \eqref{eq:unbiasedd-dd-2} note that $\tilde{g}_{\eta_k}(x_k)$ is an unbiased  estimator for $\nabla f_{\eta_k}(x_k)$. In addition, since $x_k \in \cf_k$, the conditional expectation  
			$\mathbb{E}[ \, w_k \, \mid \, \mathcal{F}_k] = 0$ almost surely, where $w_k = \tilde{g}_{\eta_k}(x_k) - \nabla f_{\eta_k}(x_k)$. 
			Let us now introduce an assumption on the terms of our stochastic approximation algorithm \eqref{algstoc}
			
			\begin{assumption} \label{asm:covstoc4}
				The sequences $\{\gamma_{k},\eta_k\}$ are positive, satisfying \( \sum_{k=0}^{\infty} \gamma_{k}=\infty \), \( \sum_{k=0}^{\infty} \gamma_{k}^{2}<\infty \) and $ \sum_{k=0}^{\infty} \gamma_k\eta_k < \infty $.
		\end{assumption}
		\medskip
		Akin to~\cite{Farzad1}, we derive consistency claims for Algo.~\eqref{algstoc}; this is a unified framework that captures the non-decision-dependent problem \eqref{eq:f=EF} and the decision-dependent problem \eqref{eq:DDproblem}. 
		
		\medskip
		
		\begin{proposition}[{\bf a.s. convergence}]\label{prop:ascvgstochastic}
			If Assumptions~\ref{asm:covstoc2} and \ref{asm:covstoc4} hold, and the optimal set $X^*$ is nonempty, the sequence $\{x_k\}$ generated by \eqref{algstoc} converges almost surely to some point $x^* \in X^*$.
		\end{proposition}
		\begin{proof}
		By the non-expansive property of projection, for any $x^* \in X^*$ and $k \geq0$,
		\begin{align*}
			\|x_{k+1}-x^*\|^2 &\leq \|x_k-x^*-\gamma_k \left(\nabla f_{\eta_k}(x_k)+w_k\right)\|^2\\
			&=\|x_k-x^*\|^2-2\gamma_k (\nabla f_{\eta_k}(x_k)+w_k)^{\top}(x_k-x^*)+ \gamma_k^2\|\nabla f_{\eta_k}(x_k)+w_k\|^2.
		\end{align*}
		By Lemma~\ref{prop:KeepConvex}, $f_{\eta_k}$ is convex for any $\eta_k >0$. Consequently, for any $k \ge 0$,
		\begin{align*}
			\|x_{k+1}-x^*\|^2 &\leq  \|x_k-x^*\|^2-2\gamma_k \left(f_{\eta_k}(x_k)-f_{\eta_k}(x^*)\right)-2\gamma_k w_k^{\top}(x_k-x^*)\\
			&+ \gamma_k^2\left\| \tilde{g}_{\eta_k}(x_k,v_k,z_k, \xi_k)\right\|^2.
		\end{align*}
		Moreover, by Lemma~\ref{Lemma:SmoothProperties} \eqref{Lemma:d},  $f(x_k) \leq f_{\eta_k}(x_k)$ and $f_{\eta_k}(x^*) \leq f(x^*)+L_0\sqrt{n+1}\eta_k$, implying
		\begin{align}\label{eq:xk-x*}
			\|x_{k+1}-x^*\|^2 &\leq  \|x_k-x^*\|^2-2\gamma_k \left(f(x_k)-f(x^*)\right)+2L_0\sqrt{n+1}\gamma_k\eta_k-2\gamma_k w_k^{\top}(x_k-x^*)\nonumber\\
			& + \gamma_k^2\left\| {\tilde{g}_{\eta_k}}(x_k,v_i,z_i, \xi_k)\right\|^2.
		\end{align}
		Taking conditional expectations of both sides with respect to ${V_k}$, ${Z_k}$, and $\bxi_k$, given $\mathcal{F}_k$, we get
		\begin{align} \label{eq:Exkx*}
			\mathbb{E}[\|x_{k+1}-x^*\|^2 \mid \mathcal{F}_k] &\leq  \|x_k-x^*\|^2-2\gamma_k \left(f(x_k)-f(x^*)\right)+2L_0\sqrt{n+1}\gamma_k\eta_k\nonumber\\
			&+ \gamma_k^2\mathbb{E}[\| {\tilde{g}_{\eta_k}}(x_k,V_k,Z_k, \bxi_k)\|^2 \mid \mathcal{F}_k]\nonumber\\
			&\leq \|x_k-x^*\|^2-2\gamma_k \left(f(x_k)-f(x^*)\right)+2L_0\sqrt{n+1}\gamma_k\eta_k+ \tfrac{4L_0^2n}{ \pi}\gamma_k^2,
		\end{align}
		where \eqref{eq:Exkx*} follows from Prop.~\ref{prop:secmombd}. Applying the Robbins-Siegmund Lemma~\cite{robbins1971convergence}, for any $x^* \in X^*$,  the sequence $\{\|x_k-x^*\|\}$ converges and 
		$\sum_{k=0}^\infty \gamma_k \left(f(x_k)-f(x^*)\right) < \infty$ a.s.. The latter implies that $ \liminf _{k \rightarrow \infty} f\left(x_{k}\right)=f^{*} $ a.s. in view of the condition $ \sum_{k=0}^{\infty} \gamma_{k}=\infty$. Therefore, there exists a subsequence $\{x_{k_j}\}$ of $\{x_k\}$ such that $f(x_{k_j}) \rightarrow f^*$ a.s. By the continuity of $f$, the sublevel sets $\{x \in \mathbb{R}^{n}: f(x) \leq c\}$ are closed, implying that all cluster points of $\{x_{k_j}\}$ belong to the optimal set $X^*$. Therefore there exists a subsubsequence of $\{x_k\}$ that converges to some point in $X^*$ a.s. Combining with the a.s. convergence of $\{\|x_k-x^*\|\}$, the entire sequence $\{x_k\}$ converges to a point in $X^*$ a.s. .
		\end{proof}
		
		\smallskip
		
		Having derived a.s. convergence, we may then show that the function value of an averaged sequence admits non-asymptotic rate guarantees. Notably, this claim holds for both the non decision-dependent setting as well as the decision-dependent regime.
		
		\medskip
		
		\begin{proposition}[{\bf Rate of convergence under diminishing $\eta_k$}]\label{prop:convex_zogd_rate}
			Suppose Assumption~\ref{asm:covstoc2} 
			holds and the optimal set $X^*$ is nonempty. Suppose $\gamma_k=\eta_k=\tfrac{1}{\sqrt{n(k+1)}}$,  $\bar{x}_{K} \triangleq \tfrac{\sum_{k=0}^{{K-1}} \gamma_{k} x_{k}}{\sum_{k=0}^{{K-1}}\gamma_{k}}$ and $a_0=\|x_0-x^*\|$. Then, for all $K$,
			\begin{equation*}
				\mathbb{E}\left[f\left(\bar{x}_{K}\right)-f^{*}\right] \leq 
				\tfrac{a_{0}^{2}{\sqrt{n}}+\left(2L_0{\frac{\sqrt{n+1}}{\sqrt{n}}}+\frac{4L_0^2{\sqrt{n}}}{ \pi}\right)\left(1+\ln ({K})\right)}{{4} \sqrt{K+{1}}}
				\lesssim \mathcal{O}\left(\tfrac{\sqrt{n} \ln(K)}{\sqrt{K}}\right).
			\end{equation*}
		\end{proposition}
		\begin{proof}
		Let us rewrite equation~\eqref{eq:Exkx*} from the previous proof.
		\begin{equation*}
			\mathbb{E}[\|x_{k+1}-x^*\|^2 \mid \mathcal{F}_k] \leq  \|x_k-x^*\|^2-2\gamma_k \left(f(x_k)-f(x^*)\right)+2L_0\sqrt{n+1}\gamma_k\eta_k+\tfrac{4L_0^2n}{\pi}\gamma_k^2.
		\end{equation*}
		Taking unconditional expectations, we have
		\begin{equation*}
			2\mathbb{E}[\gamma_k \left(f(x_k)-f(x^*)\right)] \leq  \mathbb{E}[\|x_{k}-x^*\|^2 ]-\mathbb{E}[\|x_{k+1}-x^*\|^2 ]+2L_0\sqrt{n+1}\gamma_k\eta_k+ \tfrac{4L_0^2n}{\pi}\gamma_k^2.
		\end{equation*}
		Summing over $k=0,\cdots, {K-1}$,
		\begin{equation}\label{eq:sumfkf*}
			\sum_{k=0}^{{K-1}}2\mathbb{E}[\gamma_k \left(f(x_k)-f(x^*)\right)] \leq \|x_{0}-x^*\|^2 +\sum_{k=0}^{{K-1}}\left(2L_0\sqrt{n+1}\gamma_k\eta_k+\tfrac{4L_0^2n}{\pi}\gamma_k^2\right).
		\end{equation}
		Since $\bar{x}_{K} \triangleq \tfrac{\sum_{k=0}^{{K-1}} \gamma_{k} x_{k}}{\sum_{k=0}^{{K-1}}\gamma_{k}}$ and $f$ is convex, we can apply Jensen's inequality to \eqref{eq:sumfkf*} to obtain
		\begin{equation}\label{eq:ieqfkf*}
			2\mathbb{E}[\left(f(\bar{x}_K)-f(x^*)\right)] \leq 2 \tfrac{\sum_{k=0}^{{K-1}} \mathbb{E}\left[\gamma_{k}\left(f\left(x_{k}\right)-f^{*}\right)\right]}{\sum_{k=0}^{{K-1}} \gamma_{k}} \leq \tfrac{a_{0}^{2}+ \sum_{k=0}^{{K-1}}\left(2L_0\sqrt{n+1}\gamma_k\eta_k+\frac{4L_0^2n}{ \pi}\gamma_k^2\right)}{\sum_{k=0}^{{K-1}} \gamma_{k}} .
		\end{equation}
		Recall that $\gamma_k=\eta_k={\tfrac{1}{\sqrt{n(k+1)}}}$ for any $k$, implying that $\sum_{k=0}^{{K-1}} \gamma_{k} \geq \tfrac{1}{\sqrt{n}}(\int_{1}^{K+{1}} \tfrac{1}{\sqrt{x}} d x)$, while $\sum_{k=0}^{{K-1}} \gamma_k\eta_k=  \sum_{k=0}^{{K-1}} \gamma_{k}^2 \leq \tfrac{1}{n}(1+\int_{1}^{{K}} \tfrac{1}{x} d x)$. Using these inequalities in \eqref{eq:ieqfkf*}, our desired result follows as shown next.  
		\begin{equation*}
			2\mathbb{E}[\left(f(\bar{x}_K)-f(x^*)\right)] \leq \tfrac{a_{0}^{2}+{\frac{1}{n}}\left(2L_0\sqrt{n+1}+\frac{4L_0^2n}{ \pi}\right) \left(1+\int_{1}^{{K}} \frac{1}{x} d x\right)}{\frac{1}{\sqrt{n}}\left(\int_{1}^{K+{1}} \frac{1}{\sqrt{x}} d x\right)} \leq \tfrac{a_{0}^{2}{\sqrt{n}}+\left(2L_0{\frac{\sqrt{n+1}}{\sqrt{n}}}+\frac{4L_0^2{\sqrt{n}}}{ \pi}\right)\left(1+\ln ({K})\right)}{2 \sqrt{K+{1}}}.
		\end{equation*}
		\end{proof}
		\smallskip
		From the above analysis, one can easily derive the following when $\eta_k = \eta$ for all $k$.
		\smallskip
		
		\begin{corollary}[Rate of convergence under constant $\eta$]\label{Coro:2}
			Suppose the assumptions of Prop.~\ref{prop:convex_zogd_rate} hold and 
			let $\eta_k=\eta$ for any $k$. For any $K>0$, denote $S_K=\sum_{k=0}^{{K-1}} \gamma_k$.
			
			\noindent (i)  Then the following holds for any $K > 0$. 
			\begin{equation*}
				\mathbb{E}[f(\bar{x}_K)]-f(x^*) \leq \tfrac{1}{S_K}\sum_{k=0}^{K-1}\mathbb{E}[\gamma_k \left(f(x_k)-f(x^*)\right)] \leq L_0\sqrt{n+1}\eta +\tfrac{1}{S_K}\left[\tfrac{a_0^2}{2}+\tfrac{2L_0^2n}{\pi}\sum_{k=0}^{K-1}\gamma_k^2\right].
			\end{equation*}
			\noindent (ii) Suppose for a given $K > 0$, $\gamma_k=\tfrac{R}{L_0 \sqrt{nK}}$ and $\eta = \tfrac{1}{\sqrt{nK}}$ where $R>0$ is a constant. The required iteration and sample complexities (in function evaluations) for ensuring $\E\left[f(\bar{x}_K)\right]-f(x^\ast)\leq \varepsilon $ are $\cO(n\varepsilon^{-2})$ and $\cO(n^2 \varepsilon^{-2})$, respectively.
			
		\end{corollary}
		\begin{remark}
			{Comparing} Cor.~\ref{Coro:2} to \cite[Thm 6]{nesterov2017random}, it can be observed that our scheme yields $\varepsilon$-error in $\mathcal{O}(n\varepsilon^{-2})$ iterations, as opposed to $\mathcal{O}(n^2 \varepsilon^{-2})$ in~\cite{nesterov2017random} with no degradation in sample complexity. This has a pronounced impact when the projection operation onto a set $X$ is computationally expensive. In general, this operation requires solving a potentially large-scale strongly convex optimization problem with potentially nonlinear constraints. {In fact, these benefits are supported by the numerics, as we observe later in this paper. } 
		\end{remark}
		We provide an intermediate lemma for claiming convergence of a suitable recursion~\cite{polyak1987introduction}.
		\begin{lemma}\label{lemma:stronglycov} 
			Let \( \left\{v_{k}\right\} \) be a sequence of nonnegative random variables, where \( \mathbb{E}\left[v_{0}\right]<\infty \), and let \( \left\{u_{k}\right\} \) and \( \left\{\mu_{k}\right\} \) be deterministic scalar sequences such that the following hold.
			\begin{enumerate}[label=(\roman*)]
				\item \( \mathbb{E}\left[v_{k+1} \mid v_{0}, \ldots, v_{k}\right] \leq\left(1-u_{k}\right) v_{k}+\mu_{k} \) a.s. for all \( k \geq 0 \);
				\item \( 0 \leq u_{k} \leq 1,  \mu_{k} \geq 0 \), for all \( k \geq 0 \);
				\item \( \sum_{k=0}^{\infty} u_{k}=\infty,  \sum_{k=0}^{\infty} \mu_{k}<\infty,  \lim _{k \rightarrow \infty} \frac{\mu_{k}}{u_{k}}=0 \).
			\end{enumerate}
			Then, \( v_{k} \rightarrow 0 \) almost surely as \( k \rightarrow \infty \).
		\end{lemma}
		\begin{proposition}[{\bf a.s. convergence for strongly convex $f$}]\label{prop:Strongconvex}
			Suppose Assumptions~\ref{asm:covstoc2} and \ref{asm:covstoc4} hold, and the optimal set $X^*$ is nonempty. If, in addition, $f$ is $\mu$-strongly convex\footnote{note in the non-decision dependent setting of Section~\ref{sec:esGsestimator}, we can assume that $F(\bullet, \xi)$ is $\mu$-strongly convex to obtain $f = \E_{\bxi}[F(\bullet, \bxi)]$ is $\mu$-strongly convex.} on $X$ almost surely, the sequence $\{x_k\}$ generated by \eqref{algstoc} converges almost surely to a unique optimal solution $x^*$.
		\end{proposition}
		\begin{proof}
		By the non-expansive property of the Euclidean projection, for any $x^* \in X^*$ and $k>0$,
		\begin{align*}
			\|x_{k+1}-x^*\|^2 &\leq \|x_k-x^*-\gamma_k \left(\nabla f_{\eta_k}(x_k)+w_k\right)\|^2\\
			&=\|x_k-x^*\|^2-2\gamma_k (\nabla f_{\eta_k}(x_k)+w_k)^{\top}(x_k-x^*)+ \gamma_k^2\|\nabla f_{\eta_k}(x_k)+w_k\|^2.
		\end{align*}
		Since $f$ is $\mu$-strongly convex, we have 
		\begin{align*}
			\|x_{k+1}-x^*\|^2 &\leq  \|x_k-x^*\|^2-2\gamma_k \left(f_{\eta_k}(x_k)-f_{\eta_k}(x^*)\right)-\mu\gamma_k\|x_k-x^*\|^2\\
			&-2\gamma_k w_k^{\top}(x_k-x^*) + \gamma_k^2\|\tilde{g}_{\eta_k}(x_k,v_k,z_k,\xi_k)\|^2.
		\end{align*}
		Moreover, by Lemma~\ref{Lemma:SmoothProperties} \eqref{Lemma:d} we have $f(x_k) \leq f_{\eta_k}(x_k)$ and $f_{\eta_k}(x^*) \leq f(x^*)+L_0\sqrt{n+1}\eta_k$. Thus
		\begin{align}\label{eq:norm2xkx*}
			\|x_{k+1}-x^*\|^2 &\leq  \|x_k-x^*\|^2-2\gamma_k \left(f(x_k)-f(x^*)\right)+2L_0\sqrt{n+1}\gamma_k\eta_k-\mu\gamma_k\|x_k-x^*\|^2\nonumber\\
			&-2\gamma_k w_k^{\top}(x_k-x^*)+ \gamma_k^2\| \tilde{g}_{\eta_k}(x_k,v_k,z_k,\xi_k)\|^2.
		\end{align}
		Since $x^*$ is optimal, $f(x_k)-f(x^*) \geq 0$. Thus, from \eqref{eq:norm2xkx*}, {we obtain} 
		\begin{align*}
			\|x_{k+1}-x^*\|^2 &\leq  (1-\mu \gamma_k)\|x_k-x^*\|^2+2L_0\sqrt{n+1}\gamma_k\eta_k
			-2\gamma_k w_k^{\top}(x_k-x^*)\\
			&+ \gamma_k^2\| {\tilde{g}_{\eta_k}}(x_k,v_k,z_k,\xi_k)\|^2.
		\end{align*}
		Taking expectations of both sides with respect to {${V_k}$, ${Z_k}$ and $\bxi_k$}, given $\mathcal{F}_k$, we obtain
		\begin{align}\label{eq:norm2xkx*cond}
			\mathbb{E}[\|x_{k+1}-x^*\|^2 \mid \mathcal{F}_k] &\leq  (1-\mu\gamma_k)\|x_k-x^*\|^2+2L_0\sqrt{n+1}\gamma_k\eta_k\nonumber\\
			&\quad \quad \quad \quad \quad+ \gamma_k^2\mathbb{E}\left[\| {\tilde{g}_{\eta_k}}(x_k,V_k,Z_k,\bxi_k)\|^2 \mid \mathcal{F}_k\right]\nonumber\\
			&\leq (1-\mu\gamma_k)\|x_k-x^*\|^2+2L_0\sqrt{n+1}\gamma_k\eta_k+ \tfrac{4L_0^2n}{ \pi}\gamma_k^2.
		\end{align}
		From Assumption~\ref{asm:covstoc4}, $\sum_{k=0}^\infty \gamma_k=\infty$ and $\sum_{k=0}^\infty (2L_0\sqrt{n+1}\gamma_k\eta_k+ \frac{4L_0^2n}{ \pi}\gamma_k^2)<\infty$. In addition note,\\
		$\lim _{k \rightarrow \infty}(2L_0\sqrt{n+1}\gamma_k\eta_k+ \frac{4L_0^2n}{ \pi}\gamma_k^2)/\mu\gamma_k \rightarrow 0$. Thus by Lemma~\ref{lemma:stronglycov} the result holds.
		\end{proof}
		
		\begin{lemma}~\cite{shapiro2014lectures}\label{Lemma:ak}
			Consider the following recursion: \( {b}_{k+1} \leq(1-2 c \theta / k) {b}_{k}+\frac{1}{2} \theta^{2} M^{2} / k^{2} \), where \( \theta \) and \( M \) are positive constants, \( {b}_{k} \geq 0 \), and \( (1-2 c \theta)<0 \). Then for \( k \geq 1 \), we have that
			$2 {b}_{k} \leq \frac{\max \left(\frac{\theta^{2}}{2 c \theta-\lfloor 2 c \theta\rfloor} M^{2}, 2 {b}_{1}\right)}{k}$.
		\end{lemma}
		We now prove a ZO variant of an analogous result from~\cite{shapiro2014lectures}.
		\begin{proposition}[{\bf Convergence in mean and rate statement under strong convexity}]  
			Under the setting of Proposition~\ref{prop:Strongconvex}. In addition, suppose $a_0 = {\|x_0-x^{*}\|} $ and \( \gamma_{k}=\tfrac{\theta}{k} \), where \( \theta>1 /  \mu \). Then the sequence \( \left\{x_{k}\right\} \) converges to \( x^{*} \) in mean and
			\begin{equation*}
				\mathbb{E}\left[\left\|x_{k}-x^{*}\right\|^{2}\right] \leq \tfrac{\max \left\{\theta^{2} {(4L_0\sqrt{n+1}+ \frac{8L_0^2n}{ \pi})}(\mu \theta-1)^{-1}, {a_0^2} \right\}}{k}{\lesssim \mathcal{O}\left(\tfrac{n}{k}\right)}.
			\end{equation*}
		\end{proposition}
		\begin{proof}
		Recall the recursion obtained from \eqref{eq:norm2xkx*cond}:
		\begin{equation*}
			\mathbb{E}\left[\left\|x_{k+1}-x^{*}\right\|^{2} \mid \mathcal{F}_{k}\right] \leq (1-\mu\gamma_k)\|x_k-x^*\|^2+2L_0\sqrt{n+1}\gamma_k\eta_k+ \tfrac{4L_0^2n}{ \pi}\gamma_k^2.
		\end{equation*}
		By setting $\eta_k=\gamma_k=\frac{\theta}{k}$, it follows that by taking unconditional expectations we have that
		\begin{equation*}
			\mathbb{E}\left[\left\|x_{k+1}-x^{*}\right\|^{2}\right] \leq (1-\mu\gamma_k)\mathbb{E}\left[\|x_k-x^*\|^2\right]+(2L_0\sqrt{n+1}+ \tfrac{4L_0^2n}{ \pi})\gamma_k^2.
		\end{equation*}	
		Choosing \( \theta>1 /  \mu \) and \( M^{2} / 2=(2L_0\sqrt{n+1}+ \frac{4L_0^2n}{ \pi}) \) in Lemma~\ref{Lemma:ak}, {and noting that $a_0^2 = {\|x_0-x^{*}\|}^2 $}, we obtain the result for all \( k \).
		\end{proof}
		
		\subsection{{High probability guarantees}}
		{In this subsection, we provide a high probability guarantee for a general choice of steplength and smoothing parameter sequences. These statements are then refined for a specific choice of steplengths. We proceed to show that such guarantees pave the way for developing a rate of convergence provided in an almost-sure sense (rather than for the mean sub-optimality).}   
		\begin{proposition}[\bf{Concentration inequality for function value}]\label{prop:concentration}
			{Let Assumption~\ref{asm:covstoc2} hold and assume $X$ is bounded with a finite-valued diameter $D_X$, where$D_X \, \triangleq \, \sup_{x, x' \in X} \| x - x'\|$ denote the diameter of  $X$.}
			\begin{itemize}
				\item[(i)] Denote $I_{A_K}=\sqrt{8\left(1+\frac{4n}{\pi}\right)D_X^2L_0^2\sum_{k=0}^{K-1}\gamma_k^2}$ and $I_{B_K}=\frac{4nL_0^2}{\pi} \sum_{k=0}^{K-1}\gamma_k^2$. Then {for any $\lambda > 0$,}
				\begin{equation*}
					\mathbb{P}\left(f(\bar{x}_K)-f(x^*)\geq \frac{D_X^2+\sum_{k=0}^{K-1}2L_0\sqrt{n+1}\gamma_k\eta_k 
						+\lambda\left(I_{A_K}+I_{B_K}\right)}{2\sum_{k=0}^{K-1} \gamma_k}\right)\leq \frac{1}{\lambda}+\frac{1}{\lambda^2}.
				\end{equation*}
			\end{itemize}
			{Suppose that the sub-Gaussian property, defined as\footnote{By Jensen's inequality this implies the $L^2$ Lipschitz property of $F$.}
				\begin{equation}\label{eq:sub-Gaussian}
					\E_{\bxi_k}\left[\exp\left(\|F(x,\bxi_k)-F(y,\bxi_k)\|^2\right)\right]\leq \exp\left(L_0^2\|x-y\|^2\right),
				\end{equation}
				is satisfied for (ii), (iii), and (iv)}.
			\begin{itemize}
				\item[(ii)]  There exists a constant $\al \in (0,1)$ only depending on $n$ such that for any $\varepsilon> 0$, 
				\begin{equation*}
					\mathbb{P}\left(f(\bar{x}_K)-f(x^*)\geq \frac{c_{D,K,n}+\varepsilon+\sqrt{\frac{2nc_Lc_K}{\al}\varepsilon}}{2\sum_{k=0}^{K-1} \gamma_k}\right)\leq \exp\left(-\varepsilon/3\right)+\left(\frac{\varepsilon}{c_nc_K}\right)^{c_K}\exp\left(c_K-\varepsilon/c_n\right),
				\end{equation*}
				where $c_n=\frac{4nL_0^2}{\pi}$, $c_L=8D_X^2L_0^2,$ $c_K=\sum_{k=0}^{K-1}\gamma_k^2$ and $c_{D,K,n} =D_X^2+\sum_{k=0}^{K-1}2L_0\sqrt{n+1}\gamma_k\eta_k+\sqrt{\frac{2nc_Lc_K}{\al}}$. 
				\item[(iii)] Suppose $\gamma_k=\eta_k=1/(k+1)^{1/2}$ and $\varepsilon=\max\{4c_nc_K,4c_K\}$. {Then for any positive integer $K$},  we have
				\begin{equation}\label{eq:concenfinal}
					\mathbb{P}\left(f(\bar{x}_K)-f(x^*)\geq\frac{Cn\ln(K)}{\sqrt{K}}\right)\leq K^{-4/3}+K^{\ln(4)-3},
				\end{equation}
				where $C>0$ is a constant independent of $n,K$.	
				\item[(iv)]
				Suppose {$\gamma_k=\eta_k=1/(k+1)^{1/2}$} and $c_K = \frac{\varepsilon}{2c_n}.$  Then for any $\varepsilon> 0$, 
				\begin{align*}
					& \mathbb{P}\left(f(\bar{x}_K)-f(x^*)\geq \frac{\varepsilon}{4c_n}\left(D_X^2 +\sqrt{\tfrac{c_Ln\varepsilon}{\al c_n}}+ \varepsilon\left(\left(2 L_0\sqrt{(n+1)}\right)+1+ \left(\sqrt{\tfrac{2nc_L}{2\al c_n}}\right)\right)\right)\right) \\
					&~~~~\le \exp\left(-\varepsilon/3\right)+\left(\frac{e}{2}\right)^{-(\frac{\varepsilon}{2c_n})}. 
				\end{align*}
			\end{itemize}
		\end{proposition}
		
		\begin{proof}
		\emph{(i)} Recall from \eqref{eq:xk-x*}, we have
		\begin{align}
			2\gamma_k \left(f(x_k)-f(x^*)\right)\leq&   \|x_k-x^*\|^2-\|x_{k+1}-x^*\|^2+2L_0\sqrt{n+1}\gamma_k\eta_k-2\gamma_k w_k^{\top}(x_k-x^*)\nonumber\\
			& + \gamma_k^2\left\| {\tilde{g}_{\eta_k}}(x_k,V_k,Z_k, \bxi_k)\right\|^2.
		\end{align}
		Summing over $k=0,\cdots, {K-1}$,
		\begin{align}\label{eq:fk-f*}
			2\sum_{k=0}^{K-1}\gamma_k \left(f(x_k)-f(x^*)\right) \leq& \|x_0-x^*\|^2-\|x_{K}-x^*\|^2+\sum_{k=0}^{K-1}2L_0\sqrt{n+1}\gamma_k\eta_k\nonumber\\
			&-\sum_{k=0}^{K-1}2\gamma_k w_k^{\top}(x_k-x^*)+\sum_{k=0}^{K-1}\gamma_k^2\left\| {\tilde{g}_{\eta_k}}(x_k,V_k,Z_k, \bxi_k)\right\|^2.
		\end{align}
		Let $A_k$ and $B_k$ be defined as $A_k \triangleq -2\gamma_k w_k^{\top}(x_k-x^*)$ and $B_k \triangleq \gamma_k^2\left\| {\tilde{g}_{\eta_k}}(x_k,V_k,Z_k, \bxi_k)\right\|^2$, respectively. We first obtain {a large-deviation bound} for the partial sums of $A_k$  and $B_k$.  In order to obtain a large-deviation result for partial sums of $A_k$, note that from the $L_0$-Lipschitz property of $f$ and $f_{\eta_k}$ (by Lemma~\ref{Lemma:SmoothProperties}.\ref{Lemma:b}), $\|\nabla f_{\eta_k}(x_k)\| \leq L_0$. Thus, for any sequence of $\sigma_k$ we have using $\|x_k - x^{\ast}\| \leq D_X$ that
		\begin{equation*}
			\E \lc {\|A_k\|^2}  \rc \leq\E \left[ 4D_X^2\gamma_k^2\|w_k\|^2  \right]\leq\E \lc 8D_X^2\gamma_k^2\left(L_0^2+\|\tilde{g}_{\eta_k}(x_k,V_k,Z_k, \bxi_k) \|^2\right)\rc,
		\end{equation*} 
		where we have used the elementary inequality $(a+b)^2 \leq 2(a^2 + b^2)$. By  employing Proposition~\ref{prop:secmombd} ({decision-independent setting}) {or} \ref{prop:2mombd-DD1} ({decision-dependent setting}), we have
		\begin{equation*}
			\E \lc {\|A_k\|^2}  \rc  \leq 8\left(1+\frac{4n}{\pi}\right)D_X^2L_0^2\gamma_k^2.
		\end{equation*}
		By \cite[Lem 2.3]{ghadimi2013stochastic} (see also \cite[Thm 2.1]{juditsky2008large}) we obtain
		\begin{equation}\label{eq:pbAk1}
			\mathbb{P}\left(\left\|\sum_{k=0}^{K-1}A_k\right\|\geq \lambda\sqrt{8\left(1+\frac{4n}{\pi}\right)D_X^2L_0^2\sum_{k=0}^{K-1}\gamma_k^2}\right)\leq \frac{1}{\lambda^2}.
		\end{equation}{We may get an analogous moment bound for} the partial sums of $B_k$ by {leveraging} the second moment bound of the \textbf{esGs} estimator, {as shown next.} 
		\begin{equation*}
			\E \lc {\sum_{k=0}^{K-1}B_k}  \rc \leq \frac{4nL_0^2}{\pi} \sum_{k=0}^{K-1}\gamma_k^2.
		\end{equation*}
		By the Markov inequality, we obtain
		\begin{equation}\label{eq:pbBk1}
			\mathbb{P}\left(\sum_{k=0}^{K-1}B_k\geq \lambda\frac{4nL_0^2}{\pi} \sum_{k=0}^{K-1}\gamma_k^2\right)\leq \frac{1}{\lambda}.
		\end{equation}
		By Jensen's inequality and \eqref{eq:fk-f*} we have 
		\begin{align}\label{eq:fxK-fx*1}
			f(\bar{x}_K)-f(x^*)&\leq\frac{\sum_{k=0}^{K-1}\gamma_k \left(f(x_k)-f(x^*)\right)}{\sum_{k=0}^{K-1} \gamma_k} \nonumber\\
			&\leq \frac{1}{2\sum_{k=0}^{K-1} \gamma_k}\left[D_X^2+\sum_{k=0}^{K-1}2L_0\sqrt{n+1}\gamma_k\eta_k 
			+\sum_{k=0}^{K-1}A_k+\sum_{k=0}^{K-1}B_k\right].
		\end{align}
		Denote $I_{A_K}=\sqrt{8\left(1+\frac{4n}{\pi}\right)D_X^2L_0^2\sum_{k=0}^{K-1}\gamma_k^2}$ and $I_{B_K}=\frac{4nL_0^2}{\pi} \sum_{k=0}^{K-1}\gamma_k^2$. From \eqref{eq:pbAk1} - \eqref{eq:fxK-fx*1} we have
		\begin{align*}
			\mathbb{P}&\left(f(\bar{x}_K)-f(x^*)\geq \frac{D_X^2+\sum_{k=0}^{K-1}2L_0\sqrt{n+1}\gamma_k\eta_k 
				+\lambda\left(I_{A_K}+I_{B_K}\right)}{2\sum_{k=0}^{K-1} \gamma_k}\right)\\
			&\leq \mathbb{P}\left(\left\|\sum_{k=0}^{K-1}A_k\right\|\geq \lambda I_{A_K}\right)+	\mathbb{P}\left(\sum_{k=0}^{K-1}B_k\geq \lambda I_{B_K}\right)\leq \frac{1}{\lambda}+\frac{1}{\lambda^2}.
		\end{align*}

		\emph{(ii):} By Assumption~\ref{asm:covstoc2}.\ref{asm:cvx1} and \eqref{eq:sub-Gaussian}, we have
		\begin{align}\label{eq:gkbd} \notag
			\E_{\bxi_k}  \lc \exp\left(\left\| {\tilde{g}_{\eta_k}}(x_k,V_k,Z_k, \bxi_k)\right\|^2 \right) \mid \cf_k \rc 
			&= \sum_{i=1}^n \E_{\bxi_k} \left[\exp\left(\tfrac{1}{ 2\pi \eta_k^2} \left|F\left(x_k+\eta_k\sqrt{2V_k},x_k^{-i}-Z_k^{-i},\bxi_k\right) \right. \right. \right. \\
			& \left. \left. \left. -F\left(x_k-\eta_k\sqrt{2V_k},x_k^{-i}-Z_k^{-i},\bxi_k\right)\right|^2\right) \mid \cf_k \right] \nonumber \\
			&\leq \exp\left(\tfrac{4nL_0^2}{\pi}V_k\right).
		\end{align}
		In order to obtain a large-deviation result for partial sums of $A_k$, note that from the $L_0$-Lipschitz property of $f$ and $f_{\eta_k}$ (by Lemma~\ref{Lemma:SmoothProperties}.\ref{Lemma:b}), $\|\nabla f_{\eta_k}(x_k)\| \leq L_0$. Thus, for any sequence $\left\{\sigma_k\right\}$, by leveraging $\|x_k - x^{\ast}\| \leq D_X$, we have 
		\begin{align*}
			\E_{\bxi_k} \lc \exp\left({\|A_k\|^2/\sigma_k^2} \right) \mid \cf_k \rc &\leq\E_{\bxi_k} \left[ \exp\left({4D_X^2\gamma_k^2\|w_k\|^2/\sigma_k^2}\right) \mid \cf_k \right]\\
			&\leq\E_{\bxi_k} \lc \exp\left({8D_X^2\gamma_k^2\left(L_0^2+\|\tilde{g}_{\eta_k}(x_k,V_k,Z_k, \bxi_k) \|^2\right)/\sigma_k^2}\right)\mid \cf_k \rc,
		\end{align*} 
		where we have used the elementary inequality $(a+b)^2 \leq 2(a^2 + b^2)$.
		Applying \eqref{eq:gkbd} we have
		\begin{equation*}
			\E_{\bxi_k}\left[ \exp\left({\|A_k\|^2/\sigma_k^2}\right)\mid \cf_k \right]
			\leq\exp\left(\frac{c_L\gamma_k^2}{\sigma_k^2}\right)\exp\left(\frac{4nc_L\gamma_k^2 V_k}{\pi\sigma_k^2}\right),
		\end{equation*}
		where $c_L=8D_X^2L_0^2$. Since $\{V_k\}_{k \geq 1}$ are i.i.d. random variables following $\mathcal{E}xp(1)$ and noting that $\E [\exp(t V_k)] = 1/(1-t)$ for $t < 1$, we have for large enough $\sigma_k$, 
		\begin{equation*}
			\E \left[\exp\left({\|A_k\|^2/\sigma_k^2}\right) \mid \cf_k \right]\leq\exp\left(\frac{c_L\gamma_k^2}{\sigma_k^2}\right)\left(1-\frac{4nc_L\gamma_k^2}{\pi\sigma_k^2}\right)^{-1}.
		\end{equation*}
		It is readily checked that for small enough $\al > 0$, $e^{\al-1} \leq1-\frac{4n}{\pi}\al$. Therefore, by choosing $\sigma_k=\frac{\sqrt{nc_L}\gamma_k}{\sqrt{\al}},$ 
		\begin{equation}\label{eq:exp1}
			\E \left[ \exp\left({\|A_k\|^2/\sigma_k^2}\right) \mid \cf_k \right]\leq \exp(1).
		\end{equation}
		Since $\mathbb{E}\left[A_k \mid \mathcal{F}_{k}\right]=0$, together with \eqref{eq:exp1} we can apply \cite[Lem 2.3]{ghadimi2013stochastic} (see also \cite[Thm 2.1]{juditsky2008large}) to obtain
		\begin{equation}\label{eq:pbAk}
			\mathbb{P}\left(\left\|\sum_{k=0}^{K-1}A_k\right\| \geq \sqrt{2}(1+\lambda) \sqrt{\sum_{k=0}^{K-1} \sigma_k^2}\right)\leq \exp\left(-\lambda^2/3\right).
		\end{equation}
		
		We now proceed towards obtaining an analogous large-deviation result for the partial sums of $B_k$. Observe from \eqref{eq:gkbd} we have
		\begin{equation*}
			\E_{\bxi_k} \lc \exp(B_k)\mid \cf_k \rc \leq \exp(c_n\gamma_k^2 V_k),
		\end{equation*} 
		where $c_n=\frac{4nL_0^2}{\pi}$. From the moment generating function of an exponential random variable and Jensen's inequality,  for any $t < 1/c_n$ and $\gamma_k<1$, we have
		\begin{equation}\label{eq:mgfexp}
			\E \lc \exp(tB_k)\mid \cf_k \rc \leq \E \lc\exp(tc_n\gamma_k^2 V_k)\rc\leq\left(\E\exp\left(tc_nV_k\right)\right)^{\gamma_k^2}=(1-c_nt)^{-\gamma_k^2}.
		\end{equation}
		Consider the moment generating function of {the partial sum of $B_k$ as expressed next.}
		{\small
			\begin{equation}\label{eq:MGFbd}
				\E\left(\exp\left(t\sum_{k=0}^{K-1}B_k\right)\right)=\E \lc\E\left[\exp\left(t\sum_{k=0}^{K-1}B_k \right)\mid \cf_{K-1}\right] \rc=\E\left[\exp\left(t\sum_{k=0}^{K-2}B_k \right)\E\lc\exp\lp tB_{K-1}\rp\mid \cf_{K-1}\rc\right].
		\end{equation}}
		Applying \eqref{eq:mgfexp} to \eqref{eq:MGFbd} we have
		\begin{equation}\label{eq:iterBk}
			\E\left(\exp\left(t\sum_{k=0}^{K-1}B_k\right)\right)\leq (1-c_nt)^{-\gamma_{K-1}^2}\E\lc\exp\left(t\sum_{k=0}^{K-2}B_k \right)\rc.
		\end{equation}
		Iterating \eqref{eq:iterBk} from $k=K-1$ to $k=0$ we obtain
		\begin{equation}\label{eq:prodMGF}
			\E\left(\exp\left(t\sum_{k=0}^{K-1}B_k\right)\right)\leq\left(1-c_nt\right)^{-\sum_{k=0}^{K-1}\gamma_k^2}.
		\end{equation}
		Therefore, combining \eqref{eq:prodMGF} and a Chernoff bound, we obtain 
		\begin{equation}\label{eq:min_t}
			\mathbb{P}\left(\sum_{k=0}^{K-1}B_k \geq \varepsilon\right)\leq \E\left(\exp\left(t\sum_{k=0}^{K-1}B_k\right)\right)\exp(-t\varepsilon)\leq\left(1-c_nt\right)^{-\sum_{k=0}^{K-1}\gamma_k^2}\exp(-t\varepsilon).
		\end{equation}
		Minimizing the right hand side of \eqref{eq:min_t}  over $t$, the following holds by choosing $t=\frac{1}{c_n}-\frac{\sum_{k=0}^{K-1}\gamma_k^2}{\varepsilon}$s.
		\begin{equation}\label{eq:Bk}
			\mathbb{P}\left(\sum_{k=0}^{K-1}B_k \geq \varepsilon\right)\leq \left(\frac{\varepsilon}{c_n\sum_{k=0}^{K-1}\gamma_k^2}\right)^{\sum_{k=0}^{K-1}\gamma_k^2}\exp\left(\sum_{k=0}^{K-1}\gamma_k^2-\varepsilon/c_n\right).
		\end{equation}
		From \eqref{eq:pbAk}, \eqref{eq:Bk} and \eqref{eq:fxK-fx*1} we have
		\begin{align*}
			\mathbb{P}&\left(f(\bar{x}_K)-f(x^*)\geq \frac{D_X^2+\sum_{k=0}^{K-1}2L_0\sqrt{n+1}\gamma_k\eta_k 
				+\varepsilon+\sqrt{2}(1+\lambda) \sqrt{\frac{nc_L}{\al}\sum_{k=0}^{K-1} \gamma_i^2}}{2\sum_{k=0}^{K-1} \gamma_k}\right)\\
			&\leq \mathbb{P}\left(\left\|\sum_{k=0}^{K-1}A_k\right\| \geq \sqrt{2}(1+\lambda) \sqrt{\frac{nc_L}{\al}\sum_{k=0}^{K-1} \gamma_i^2}\right)+\mathbb{P}\left(\sum_{k=0}^{K-1}B_k \geq \varepsilon\right)\\
			&\leq \exp\left(-\lambda^2/3\right)+\left(\frac{\varepsilon}{c_n\sum_{k=0}^{K-1}\gamma_k^2}\right)^{\sum_{k=0}^{K-1}\gamma_k^2}\exp\left(\sum_{k=0}^{K-1}\gamma_k^2-\varepsilon/c_n\right).
		\end{align*}
		{\em (iii)} If $c_K = \sum_{k=0}^{K-1}\gamma_k^2$ and $c_{D,K,n} =D_X^2+\sum_{k=0}^{K-1}2L_0\sqrt{n+1}\gamma_k\eta_k+\sqrt{\frac{2nc_Lc_K}{\al}}$, by setting $\lambda=\sqrt{\varepsilon}$, we have
		\begin{equation}\label{eq:conceneps}
			\mathbb{P}\left(f(\bar{x}_K)-f(x^*)\geq \frac{c_{D,K,n}+\varepsilon+\sqrt{\frac{2nc_Lc_K}{\al}\varepsilon}}{2\sum_{k=0}^{K-1} \gamma_k}\right)\leq \exp\left(-\varepsilon/3\right)+\left(\frac{\varepsilon}{c_nc_K}\right)^{c_K}\exp\left(c_K-\varepsilon/c_n\right).
		\end{equation}

		Since $\gamma_k$ is not summable. we get a concentration inequality with exponential rate.
		Choosing $\gamma_k=\eta_k=1/(k+1)^{1/2}$, $\varepsilon=\max\{4c_nc_K,4c_K\}$, we know that $c_K \simeq \ln(K), \sum_{k=0}^{K-1}\gamma_k\eta_k\simeq \ln(K)$ and $\sum_{k=0}^{K-1}\gamma_k\simeq \sqrt{K}$. Therefore, applying \eqref{eq:conceneps} we have for large enough $K$ there exist a constant $C>0$ independent of $n,K$ such that
		\begin{align}\label{eq:BigOconcen}
			\notag\mathbb{P}\left(f(\bar{x}_K)-f(x^*)\geq\frac{Cn\ln(K)}{\sqrt{K}}\right) & \le 
			\mathbb{P}\left(f(\bar{x}_K)-f(x^*)\geq \frac{c_{D,K,n}+\varepsilon+\sqrt{\frac{2nc_Lc_K}{\al}\varepsilon}}{2\sum_{k=0}^{K-1} \gamma_k}\right)  \\
			& \leq \exp\left(-\varepsilon/3\right)+\left(\frac{\varepsilon}{c_nc_K}\right)^{c_K}\exp\left(c_K-\varepsilon/c_n\right).
		\end{align}
		Since $\varepsilon\geq4c_K$ and $c_K\geq \ln(K)$, the first term of the right hand side of \eqref{eq:BigOconcen} becomes
		\begin{equation}\label{eq:concen1}
			\exp\left(-\varepsilon/3\right)\leq\exp\left(\frac{-4c_K}{3}\right)\leq \exp\left(\frac{-4}{3}\ln (K)\right)=K^{-4/3}.
		\end{equation}
		Set $u=\max\{4,4/c_n\}$, then $\varepsilon=uc_nc_K$. Substituting in the last term of \eqref{eq:BigOconcen} we have
		\begin{equation}\label{eq:concen2}
			\left(\frac{\varepsilon}{c_nc_K}\right)^{c_K}\exp\left(c_K-\varepsilon/c_n\right)=u^{c_K}\exp\left(c_K-uc_K\right)=\exp\left(\left(\ln (u)+1-u\right)c_K\right).
		\end{equation}
		The function $h(u)=\ln(u)+1-u$ has negative derivative $h^{\prime}(u)=1/u-1$ for $u>1$. Since $u\geq4$, $h(u)\leq\ln(4)-3$. Therefore from \eqref{eq:concen2} we have 
		\begin{equation}\label{eq:concen3}
			\left(\frac{\varepsilon}{c_nc_K}\right)^{c_K}\exp\left(c_K-\varepsilon/c_n\right)\leq\exp\left((\ln(4)-3)c_K\right)\leq\exp\left((\ln(4)-3)\ln(K)\right)=K^{\ln(4)-3},
		\end{equation}
		where we use $\ln(4)-3<-1$ and $c_K\geq \ln(K)$ in the last inequality. Plugging \eqref{eq:concen1}-\eqref{eq:concen3} into \eqref{eq:BigOconcen} we get our desired result.
		
		\noindent (iv) Suppose $c_K = \frac{\varepsilon}{2c_n}.$ Then 
		\begin{align}
			\notag	\mathbb{P}\left(f(\bar{x}_K)-f(x^*)\geq \frac{c_{D,K,n}+\varepsilon+\sqrt{\frac{2nc_Lc_K}{\al}\varepsilon}}{2\sum_{k=0}^{K-1} \gamma_k}\right)& \leq \exp\left(-\varepsilon/3\right)+\left(\frac{\varepsilon}{c_nc_K}\right)^{c_K}\exp\left(c_K-\varepsilon/c_n\right) \\
			\notag  & \leq \exp\left(-\varepsilon/3\right)+2^{(\frac{\varepsilon}{2c_n})}\exp\left(-\varepsilon/2c_n\right) \\
			\notag & \leq \exp\left(-\varepsilon/3\right)\\
			\notag  &~~~+\left(\frac{2}{e}\right)^{(\frac{\varepsilon}{2c_n})}\exp\left(\varepsilon/2c_n\right)\exp\left(-\varepsilon/2c_n\right) \\
			\label{hp:b1}            & = \exp\left(-\varepsilon/3\right)+\left(\frac{e}{2}\right)^{-(\frac{\varepsilon}{2c_n})}.
		\end{align}
		Observe that since $\eta_k = \gamma_k$, \begin{align*}
			c_{D,K,n} & = D_X^2 + 2 L_0\sqrt{(n+1)}\sum_{k=0}^{K-1}\gamma_k \eta_k  +\sqrt{\frac{2nc_Lc_K}{\al}} 
			=D_X^2 + 2 L_0\sqrt{(n+1)}\underbrace{\sum_{k=0}^{K-1}\gamma_k^2}_{\triangleq c_K} + \sqrt{\frac{c_L n \varepsilon}{\al c_n}} \\
			& = D_X^2 + 2 L_0\sqrt{(n+1)}\frac{\varepsilon}{2c_n}+ \sqrt{\frac{c_Ln\varepsilon}{\al c_n}}. 
		\end{align*}
		Observe that 
		\begin{align}
			\frac{\varepsilon}{2c_n} = c_K = \sum_{k=0}^{K-1}\gamma_k^2 \le \ln(K) \implies K \ge e^{\frac{\varepsilon}{2c_n}}.  
		\end{align}
		We may then get a lower bound on the probability, as shown next.
		{\small
		\begin{align}
			\notag & \mathbb{P}\left(f(\bar{x}_K)-f(x^*)\geq \frac{c_{D,K,n}+\sqrt{\frac{c_L\varepsilon}{\al}}+\varepsilon+\sqrt{\frac{2nc_Lc_K}{\al}\varepsilon}}{2\sum_{k=0}^{K-1} \gamma_k}\right) \\
			\notag    &~~~\ge 
			\mathbb{P}\left(f(\bar{x}_K)-f(x^*)\geq \frac{D_X^2 +\sqrt{\frac{c_Ln\varepsilon}{\al c_n}}+ \varepsilon\left(\left(2 L_0\sqrt{(n+1)}\right)+1+ \left(\sqrt{\frac{2nc_L}{2\al c_n}}\right)\right)}{2e^{\frac{\varepsilon}{2c_n}}}\right) \\
			\label{hp:b2}    &\overset{e^{(-u)} \le u}{\ge} 
			\mathbb{P}\left(f(\bar{x}_K)-f(x^*)\geq \frac{\varepsilon}{4c_n}\left(D_X^2 +\sqrt{\tfrac{c_Ln\varepsilon}{\al c_n}}+ \varepsilon\left(\left(2 L_0\sqrt{(n+1)}\right)+1+ \left(\sqrt{\tfrac{2nc_L}{2\al c_n}}\right)\right)\right)\right). 
		\end{align}}
		{By leveraging \eqref{hp:b1} and \eqref{hp:b2},}
		\begin{align*}
			& \mathbb{P}\left(f(\bar{x}_K)-f(x^*)\geq \frac{\varepsilon}{4c_n}\left(D_X^2 +\sqrt{\tfrac{c_Ln\varepsilon}{\al c_n}}+ \varepsilon\left(\left(2 L_0\sqrt{(n+1)}\right)+1+ \left(\sqrt{\tfrac{2nc_L}{2\al c_n}}\right)\right)\right)\right) \\
			&~~~~\le \exp\left(-\varepsilon/3\right)+\left(\frac{e}{2}\right)^{-(\frac{\varepsilon}{2c_n})}. 
		\end{align*}

	\end{proof}
	\begin{remark}
		Note that one may be able to relax the assumption on the boundedness of $X$ in Proposition~\ref{prop:concentration}. In the non-decision dependent setting of Section~\ref{sec:esGsestimator}, for example, one can assume that $F$ is strongly convex in $x$, Lipschitz continuous on $\xi$, and $\xi$ satisfies a \emph{Gaussian concentration property}. In this case, \cite[Thm 2.2]{frikha2012concentration} yields a concentration inequality for {$\|x_k-x^*\|$}. 
	\end{remark}
	
	{We conclude this section with a result that leverages the Borel-Cantelli Lemma to obtain an almost-sure convergence rate guarantee.}
	
	\begin{corollary}
		Under the setting of Proposition~\ref{prop:concentration}, the following holds for any $K > 0.$
		\begin{equation*}
			f(\bar{x}_K)-f(x^*)\, = \, \mathcal{O}\left(\frac{n \ln K}{\sqrt{K}}\right) \quad a.s.
		\end{equation*}
	\end{corollary}
	\begin{proof}
	By \eqref{eq:concenfinal} from Proposition~\ref{prop:concentration}{(iii)}, we have
	\begin{equation}
		\mathbb{P}\left(f(\bar{x}_K)-f(x^*)\geq\frac{Cn\ln(K)}{\sqrt{K}}\right)\leq K^{-4/3}+K^{\ln(4)-3}.
	\end{equation}
	Observe that the right hand side of the above inequality is summable. Recall that $C$ is a constant independent of $n$ and $K$. Therefore, by the Borel-Cantelli lemma, we have
	\begin{equation*}
		\mathbb{P}\left(\, f(\bar{x}_K)-f(x^*)\geq \frac{Cn\ln K}{\sqrt{K}} \quad {\mbox{infinitely often} }\, \right)\, =\, 0.
	\end{equation*}
	That is, except for finitely many $K$, 
	\begin{equation*}
		f(\bar{x}_K)-f(x^*)=\mathcal{O}\left(\frac{n \ln K}{\sqrt{K}}\right) \quad a.s.
	\end{equation*}
	which is the desired result.
	\end{proof}

	\section{Addressing {nonsmooth nonconvex stochastic optimization}}\label{sec:Nonconv}
	In this section, we consider the unconstrained nonsmooth and nonconvex problem, applying the update rule \eqref{algstoc} with $X=\R^n$. {Specifically, given an $x_0 \in \R^n$, let $\{x_k\}$ be generated as per the following update rule, given an $x_0 \in \mathbb{R}^n.$}
	\begin{equation}\label{eq:updateRn}
		x_{k+1}\, = \, x_{k}-\gamma_{k}\tilde{g}_{\eta_k}(x_k, V_k, Z_k, \bxi_k).
	\end{equation}
	In Section~\ref{sec:5.1}, we introduce the nonconvex and nonsmooth setting and provide rate guarantees and asymptotic statements in Sections~\ref{sec:5.2} and ~\ref{sec:5.3}. respectively.
	
	\subsection{Background and assumptions}\label{sec:5.1}
	We recall some key results in Clarke's nonsmooth calculus~\cite{clarke98} which allow for articulating stationarity conditions in nonsmooth nonconvex regimes. At the heart of this analysis is the directional derivative,  defined next. \\
	\begin{definition}[Directional derivatives and Clarke generalized gradients~\cite{clarke98}]\label{Def:Calrkesubgrad}  
		The directional derivative of $h$ at $x$ in a direction $v$ is defined as 
		\begin{equation}
			h^{\circ}(x,v) \triangleq  \limsup_{y \to x, t \downarrow 0} \left(\frac{h(y+tv)-h(y)}{t}\right).
		\end{equation} 
		The Clarke generalized gradient (or Clarke subdifferential) of a function $h$ at $x$ is defined as follows.
		\begin{equation}
			\partial h(x) \, \triangleq \, \left\{ \, \zeta \, \in \, \mathbb{R}^n\, \mid\, h^{\circ}(x,v) \, \geq \,  \zeta^\top v, \quad \forall v \, \in \, \mathbb{R}^n\,\right\}.
		\end{equation}
		In other words, 
		$$h^{\circ}(x,v) =  \sup_{g \in \partial h(x)}  \ g^{\top}v.$$ 
	\end{definition}
	If $h$ is $C^1$ at $x$, the Clarke
	generalized gradient reduces to the standard gradient, {i.e.,} $\partial h(x)
	= \nabla_{x} h(x).$ If $x$ is a local minimizer of $h$, then we have that $0
	\in \partial h(x)$. In fact, this claim can be extended to convex constrained
	regimes, {i.e.,} if $x$ is a local minimizer of $\displaystyle \min_{ \in
		X} \ h(x)$, then $x$ satisfies $0 \in \partial h(x) + \mathcal{N}_X(x)$,
	where $\mathcal{N}_{X}(x)$ denotes the normal cone of $X$ defined at
	$x$~\cite{clarke98}. Furthermore, if $h$ is locally Lipschitz on an open set
	{$\mathcal{C}$} containing $X$, then $h$ is differentiable almost everywhere
	on $C$ by Rademacher's theorem~\cite{clarke98}. Suppose $C_h$ denotes the set
	of points where $h$ is not differentiable. We now provide some properties of
	the Clarke generalized gradient.   
	
	\begin{proposition}[Properties of Clarke generalized gradients~\cite{clarke98}] \em
		Suppose $h$ is $L_0$-Lipschitz continuous on $\mathbb{R}^n$. Then the following hold {for any $x \in \mathbb{R}^n$}.
		\begin{enumerate}
			\item[(i)] $\partial h(x)$ is a nonempty, convex, and compact  set and $\|g \| \leq {L_0}$ for any $g \in \partial h(x)$. 
			\item[(ii)] $h$ is differentiable almost everywhere.  (iii) $\partial h(x)$ is an upper-semicontinuous map defined as 
			$\partial h(x) \, \triangleq \, \mbox{conv}\left\{\, g \, \mid \, g \, = \, \lim_{k \to \infty} \nabla_{x} h(x_k), C_h \, \not \, \owns x_k \, \to \, x\, \right\}.$ $\Box$
		\end{enumerate}
	\end{proposition}
	We may also define the $\delta$-Clarke generalized gradient~\cite{goldstein77} as   
	\begin{equation}
		\partial_{\delta} h(x) \, \triangleq \, \mbox{conv}\left\{ \, \zeta\, \mid \, \zeta \, \in \, \partial h(y), \left\|\, x-y\, \right\|\, \leq \, \delta \,\right\} = \mbox{conv} \left\{ \, \bigcup_{\| y-x\| \le \delta } \partial h(y) \, \right\}. 
	\end{equation}
	Notice that $\delta$-Clarke generalized gradient of $h$ at $x$ is given by the convex hull of elements in the Clarke generalized gradient of vectors within a distance $\delta$ of $x$. {This has particular relevance in the current context since it has been shown that a stationarity point of the $\eta$-smoothed problem satisfies the enlarged stationarity conditions with parameter $\eta$. This is captured by the next lemma~\cite{goldstein77}.
		\begin{lemma}
			Consider  a problem, where $f$ is a locally Lipschitz continuous function. For any $\eta > 0$ and any $x \in \mathbb{R}^n$,  $\nabla f_{\eta}(x) \in \partial_{\eta} f(x).$ 
	\end{lemma}}
	
	Next, we consider the set of cluster points associated with sequences associated with the gradients of the smoothed functions.

	\begin{definition}\label{Def:subgradmoli}{\cite[Def 4.1.]{ermoliev1995minimization}}
		Let $f: \mathbb{R}^n \rightarrow \mathbb{R}$ be locally integrable and let $\left\{\, f_k:=f_{\eta_k} \, \mid \, k \in \mathbf{N}\right\}$ be a sequence of $C^1$ mollified functions obtained from $f$ by convolution with the sequence of mollifiers $\left\{\phi_k:=\phi_{\eta_k}: \mathbb{R}^n \rightarrow \mathbb{R}_{+} \, \mid \, k \in \mathbf{N}\right\}$, where $\eta_k \downarrow 0$ as $k \rightarrow \infty$.  The subgradient set of $\phi$-mollifier of $f$ at $x$ is denoted by $\partial_\phi f(x)$ and represents  the cluster points of all possible sequences $\left\{\nabla f_k\left(x_k\right)\right\}$ such that $x_k \rightarrow x$, as captured by the following definition\footnote{Here $\operatorname{Limsup}$ is the outer limit (limit superior) of sets. $\operatorname{Limsup}_{n\rightarrow\infty}C_n$ consists of all cluster points of all possible sequences $y_n$ where each $y_n\in C_n$.}.
		$$
		\partial_\phi f(x):=\mbox{Limsup}_{k \rightarrow \infty}\left\{\nabla f_k\left(x_k\right) \mid x_k \rightarrow x\right\} 
		$$
	\end{definition}
	
	\medskip
	
	Observe that that the mollifiers are such that the mollified functions $f_k$ are smooth (of class $C^1$ ) as would be the case if the mollifiers $\phi_k$ are smooth.
	We now borrow an important result from \cite[Thm 4.11]{ermoliev1995minimization} which is stated for compactly supported mollifiers $\phi$. This result clarifies the relationship between $\partial_{\phi} f(x)$ and $\partial f(x)$, where the latter subdifferential is in the Clarke sense. 
	However, we let the patient reader check that the result is readily extended to our setting with $\phi$ given by \eqref{eq:phiu}. 
	\begin{proposition}\label{prop:phiinclarke}
		If $f: \mathbb{R}^n \rightarrow \mathbb{R}$ is lower semicontinuous and locally integrable, then
		$$
		\mathrm{conv} \left(\partial_\phi f(x) \right) \,  \subset 
		\partial f(x).
		$$ 
	\end{proposition}
	{The above results are essential as we proceed to show that we may develop an efficient scheme for computing an approximate stationary point of $f_{\eta}$ and relating it to $f$ via the enlarged Clarke generalized gradient or by showing that the sequences converges asymptotically to} $\mbox{conv} \, {\left(\partial_\phi f(x)\right)} \,$ and via the above inclusion, it can be claimed that this element is indeed a Clarke stationary point of $f$. 
	
	\medskip
	
	We conclude with a review of the  assumptions employed in this section.
	\begin{assumption}\label{asm:non-convex} Assumption~\ref{asm:ascvgstochastic} (non-DD), Assumption~\ref{asm:decdep} (DD), or Assumption~\ref{asm:decdep} hold. 
	\end{assumption}
	Note that under Assumption~\ref{asm:non-convex} in the non-DD regime (Assumption~\ref{asm:non-convex} DD regime, respectively), by Remark~\ref{rem:EF-Lip} (Lemmas~\ref{lem:F-decdep1-Lip} or~\ref{lem:f-decdep2-Lip}, respectively) the function $f$ is $L_0$-Lipschitz continuous. Moreover, Proposition~\ref{prop:secmombd} (Propositions~\ref{prop:2mombd-DD1} or~\ref{prop:DDsecmombd}, respectively) implies that
	\begin{equation*}
		\mathbb{E}_{V_k, Z_k, \bxi_k} \left[\, \tilde{g}_{\eta_k}(x_k, V_k, Z_k, \bxi_k)\, \right] \, \leq \, \frac{4L_0^2 n}{\pi}.
	\end{equation*}

	\subsection{Convergence gurantees in nonsmooth nonconvex settings.}\label{sec:5.2}
	We now present rate guarantees for the expected stationarity residual. 
	\begin{table}[H]
		\scriptsize
		\centering
		\caption{}{{Complexity guarantees} in nonconvex settings}
		\renewcommand{\arraystretch}{2}
		\begin{tabular}{ c C{3cm} c c c c}
			\hline
			Parameters&Convergence metric&  Convergence rate &Iteration complexity&Improvement &Assumption\\
			\hline
			$\gamma_k = \frac{\eta}{L_0 \sqrt{n}\sqrt{k+1}}$& $\mathbb{E}_{R_K, \bxi}\left[\|\nabla f_{\eta}(x_{R_K})\|^2\right]$ &\eqref{eq:msefetak2}&${\tilde{\mathcal{O}}}\left(\, \left(L_0^2 n \eta^{-1} + L_0^4 n^2 \right) \varepsilon^{-4} \, \right)$&${\mathcal{O}(n)}$&Asm.~\ref{asm:non-convex}\\          \hline
			- &  $ x_k \rightarrow x^*$ in probability and $0 \in \partial f(x^*)\, a.s.$& - &-&-&Asm.~\ref{asm:non-convex}-\ref{asm:etagammanoncvx}\\
			\hline
		\end{tabular}
		\vspace{0.2em}
		
		{\textit{Remark:} Sample complexity scales as $n$ times the iteration complexity.}
	\end{table}
	\begin{proposition}\label{prop:gradgotozero}
		Suppose Assumption~\ref{asm:non-convex} holds. Consider the sequence $\left\{\, x_k\, \right\}$ generated by \eqref{eq:updateRn}. Then the following hold. 
		
		\noindent (a) For any $K > 0$,
		\begin{align}
			\notag	\frac{\sum_{k=0}^{K-1}\gamma_k\mathbb{E}\left[\, \left\|\nabla f_{\eta_k}(x_k)\right\|^2\, \right]}{\sum_{k=0}^{K-1}\gamma_k} &\leq \frac{1}{\sum_{k=0}^{K-1}\gamma_k} \left[f(x_0)-f^*+L_0\sqrt{n+1}\sum_{k=1}^{K-1}|\eta_k-\eta_{k-1}|\right.\\
			&~~~\left.+L_0\sqrt{n+1}(\eta_0+\eta_K)  +\frac{4L_0^3n^{\frac{3}{2}}}{\sqrt{2}\pi^{\frac{3}{2}}}\sum_{k=0}^{K-1}\frac{\gamma_k^2}{\eta_k} \right].
		\end{align}
		(b) Suppose $R_K$ denotes a random variable taking values in $\{0, 1, \cdots, K-1\}$ with probability mass function $\mathbb{P}[R_K = j] = \frac{\gamma_j}{\sum_{\ell=0}^{K-1} \gamma_\ell}$. Then we have for any $K > 0$,
		\begin{align} 
			\mathbb{E}_{R_j, \bxi} \left[ \| \nabla f_{\eta_{R_K}} (x_{R_K}) \|^2 \right] &\leq \frac{1}{\sum_{k=0}^{K-1}\gamma_k} \left[f(x_0)-f^*+L_0\sqrt{n+1}\sum_{k=1}^{K-1}|\eta_k-\eta_{k-1}|\right.\\
			&~~~\left.+L_0\sqrt{n+1}(\eta_0+\eta_K)  +\frac{4L_0^3n^{\frac{3}{2}}}{\sqrt{2}\pi^{\frac{3}{2}}}\sum_{k=0}^{K-1}\frac{\gamma_k^2}{\eta_k} \right],
		\end{align}
	\end{proposition}
	\begin{proof}
	\noindent (a) Since $f$ is $L_0$-Lipschitz continuous, by Lemma~\ref{Lemma:SmoothProperties}, $f_\eta$ has Lipschitz gradient with Lipschitz constant $L_1(f_\eta)= \frac{2L_0\sqrt{n}}{\eta\sqrt{2 \pi}}$. Then we have
	\begin{equation*}
		f_{\eta_k}(x_{k+1})\, \leq \, f_{\eta_k}(x_k) -  \gamma_k \nabla f_{\eta_k}^\top \tilde{g}_{\eta_k}(x_k,v_k,z_k,\xi_k) +\frac{L_1(f_{\eta_k})}{2}\gamma_k^2 \left\|\, \tilde{g}_{\eta_k}(x_k,v_k,z_k,\xi_k)\, \right\|^2.
	\end{equation*}
	Taking conditional expectation with respect to $\mathcal{F}_k$ and applying Proposition~\ref{prop:secmombd}, we have 
	\begin{align}\label{eq:expnoncvx}
		\mathbb{E}\left[f_{\eta_k}(x_{k+1}) \mid \mathcal{F}_k\right] &\leq f_{\eta_k}(x_k) -  \gamma_k \mathbb{E}\left[\, \left\|\nabla f_{\eta_k}(x_k)\right\|^2 \right] +\frac{L_0\sqrt{n}}{\eta_k\sqrt{2\pi}}\gamma_k^2 \mathbb{E} \left[\, \|\tilde{g}_{\eta_k}(x_k,V_k,Z_k,\bxi_k)\|^2 \, \mid \cf_k \right] \nonumber\\
		&\leq f_{\eta_k}(x_k) -  \gamma_k \mathbb{E} \left[\left\|\nabla f_{\eta_k}(x_k)\right\|^2 \right] +\frac{4L_0^3n^{\frac{3}{2}}}{\eta_k\sqrt{2}\pi^{\frac{3}{2}}}\gamma_k^2. 
	\end{align}
	Taking unconditional expectations and rearranging the order
	\begin{equation}
		\gamma_k \mathbb{E}\left[\, \left\|\nabla f_{\eta_k}(x_k)\right\|^2 \right] \,	\leq \, \mathbb{E}\left[f_{\eta_k}(x_k)\right]- \mathbb{E}\left[f_{\eta_k}(x_{k+1})\right]  +\frac{4L_0^3n^{\frac{3}{2}}}{\eta_k\sqrt{2}\pi^{\frac{3}{2}}}\gamma_k^2. 
	\end{equation}
	Summing both sides from $k=0$ to $K-1$, we have 
	\begin{align}
		\sum_{k=0}^{K-1}\gamma_k\mathbb{E}\left[\, \left\|\nabla f_{\eta_k}(x_k)\right\|^2 \, \right] \, \leq \,  &\left[\sum_{k=1}^{K-1}\mathbb{E}\left(f_{\eta_k}(x_k)-f_{\eta_{k-1}}(x_k)\right)\right.\nonumber\\
		&\left.+\left(f_{\eta_0}(x_0)-\mathbb{E}\left(f_{\eta_K}(x_{K+1})\right)\right)+\frac{4L_0^3n^{\frac{3}{2}}}{\sqrt{2}\pi^{\frac{3}{2}}}\sum_{k=0}^{K-1}\frac{\gamma_k^2}{\eta_k} \right].
	\end{align}
	Recall that $f(x) \geq f^*$ for any $ x \in \mathbb{R}^n$. By Lemma~\ref{Lemma:SmoothProperties} (\ref{Lemma:c}), we know 
	\begin{equation*}
		\left(f_{\eta_0}(x_0)-\mathbb{E}\left(f_{\eta_K}(x_{K+1})\right)\right) =  f_{\eta_0}(x_0)-f^*+L_0\sqrt{n+1}\eta_K \leq  f(x_0) - f^* + L_0 \sqrt{n+1}(\eta_0 +\eta_K).
	\end{equation*}
	Applying Lemma~\ref{Lemma:SmoothProperties} (\ref{Lemma:e}) we have
	\begin{align}\label{eq: sumgrad}
		\notag	\sum_{k=0}^{K-1}\gamma_k\mathbb{E}\left[\,  \left\|\nabla f_{\eta_k}(x_k)\right\|^2 \, \right] & \leq   \left[  f(x_0)-f^* + L_0\sqrt{n+1}\sum_{k=1}^{K-1}|\eta_k-\eta_{k-1}| \right. \\
		& ~~~+ \left.L_0\sqrt{n+1}(\eta_0+\eta_K) +\frac{4L_0^3n^{\frac{3}{2}}}{\sqrt{2}\pi^{\frac{3}{2}}}\sum_{k=0}^{K-1}\frac{\gamma_k^2}{\eta_k} \, \right].
	\end{align}
	Dividing both sides by $\sum_{k=0}^{K-1}\gamma_k$, we obtain
	\begin{align}\label{eq:msgraderror}
		\notag	\frac{\sum_{k=0}^{K-1}\gamma_k\mathbb{E}\left[\, \left\|\nabla f_{\eta_k}(x_k)\right\|^2\, \right]}{\sum_{k=0}^{K-1}\gamma_k} &\leq \frac{1}{\sum_{k=0}^{K-1}\gamma_k} \left[f(x_0)-f^*+L_0\sqrt{n+1}\sum_{k=1}^{K-1}|\eta_k-\eta_{k-1}|\right.\\
		&~~~\left. +L_0\sqrt{n+1}(\eta_0+\eta_K)+\frac{4L_0^3n^{\frac{3}{2}}}{\sqrt{2}\pi^{\frac{3}{2}}}\sum_{k=0}^{K-1}\frac{\gamma_k^2}{\eta_k} \right].
	\end{align}
	\noindent (b) Suppose $R_K$ denotes a random variable taking values in $\{0, 1, \cdots, K-1\}$ with probability mass function $\mathbb{P}[R_K = j] = \frac{\gamma_j}{\sum_{\ell=0}^{K-1} \gamma_\ell}$. Then we have for any $K > 0$, 
	\begin{align*} 
		\mathbb{E}_{R_K, \bxi} \left[ \| \nabla f_{\eta_{R_K}} (x_{R_K}) \|^2 \right] &\leq \frac{1}{\sum_{k=0}^{K-1}\gamma_k} \left[f(x_0)-f^*+L_0\sqrt{n+1}\sum_{k=1}^{K-1}|\eta_k-\eta_{k-1}|\right.\\
		&~~~\left.+L_0\sqrt{n+1}(\eta_0+\eta_K)  +\frac{4L_0^3n^{\frac{3}{2}}}{\sqrt{2}\pi^{\frac{3}{2}}}\sum_{k=0}^{K-1}\frac{\gamma_k^2}{\eta_k} \right],
	\end{align*}
	where $\mathbb{E}_{R_K,\bxi}[\bullet]$ denotes the appropriately defined expectation over the joint space.
	
	\end{proof}
	\medskip
	
	{We now utilize the above result to derive formal rate guarantees in nonsmooth nonsmooth regimes, where we employ a constant sequence $\{\eta_k\}$ with $\eta_k = \eta$ for any $k$}. In such a case, we may articulate an $(\eta,\varepsilon)$-stationary point as a vector $x$ satisfying 
	\begin{equation}
		\mathbb{E}_{R_K, \bxi}\left[\|\nabla f_{\eta}(x_{R_K})\|^2\right] \le \varepsilon^2.
	\end{equation}
		We denote by {$K_{\eta,\varepsilon}$ and $S_{\eta,\varepsilon}$} the corresponding minimum number of steps and {oracle calls} required to achieve this bound. We provide a supporting lemma (proven via the Lambert function) which allows us to prove the following result (cf.~\cite{jalilzdeh22smoothed}).\\
		
		\begin{lemma}\label{lambert} Suppose $\frac{a+b\ln(K)}{K} \le \varepsilon$ for some $a, b > 0$. Then $K = \mathcal{O}\left(\frac{1}{\varepsilon}\ln\left(\frac{1}{\varepsilon}\right)\right)$. 
		\end{lemma}
		
		\begin{proposition}
			Suppose Assumption~\ref{asm:non-convex} holds. 
			\allowdisplaybreaks
			Suppose $\gamma_k = \frac{\eta}{L_0 \sqrt{n}\sqrt{k+1}}$ and $\eta_k = \eta$ for any $k$, then the {iteration and sample-complexity} of computing an $(\eta,\varepsilon)$-stationary point are
			\begin{equation*}
				K_{\eta,\varepsilon} = {\tilde{\mathcal{O}}}\left(\, \left(L_0^2 n \eta^{-1} + L_0^4 n^2 \right) \varepsilon^{-4} \, \right) \mbox{ and  }  
				{S_{\eta,\varepsilon}} = {\tilde{\mathcal{O}}}\left(\, {n}\left(L_0^2 n \eta^{-1} + L_0^4 n^2 \right) \varepsilon^{-4} \, \right), \mbox{ respectively} \footnotemark.
			\end{equation*}
			\footnotetext{Note that $\tilde{\mathcal{O}}(\bullet)$ is a counterpart of $\mathcal{O}$ that suppresses poly-logarithmic factors.}
		\end{proposition}
		
		\begin{proof} 
		Recall equation~\eqref{eq:msgraderror}, by choosing $\gamma_k = \frac{\eta}{L_0 \sqrt{n}\sqrt{k+1}}$ and $\eta_k = \eta$ for any $k$, we have 
		\begin{align}\label{eq:msefetak2}
			\notag
			\mathbb{E}_{R_K,\bxi}\left[\, \left\|\nabla f_{\eta}(x_{R_K})\right\|^2\, \right]  \, & \leq \, \frac{1}{\sum_{k=0}^{K-1}\gamma_k} \left[f(x_0)-f^*+L_0\sqrt{n+1}\sum_{k=0}^{K-1}|\eta_k-\eta_{k-1}|\right. \\
			~~~~&\left. + L_0\sqrt{n+1}(\eta_0+\eta_K) +\frac{4L_0^3n^{\frac{3}{2}}}{\sqrt{2}\pi^{\frac{3}{2}}}\sum_{k=0}^{K-1}\frac{\gamma_k^2}{\eta_k} \right] \nonumber\\
			\notag& = \, \frac{1}{\sum_{k=0}^{K-1}\gamma_k} \left[f(x_0)-f^*+ 2L_0\sqrt{n+1}\eta +\frac{4L_0^3n^{\frac{3}{2}}}{\sqrt{2}\pi^{\frac{3}{2}}}\sum_{k=0}^{K-1}\frac{\gamma_k^2}{\eta} \right] \\
			\notag  & \le \, \frac{1}{ \frac{\eta}{L_0 \sqrt{n} }\sqrt{K}} \left[f(x_0)-f^*+ 2L_0\sqrt{n+1}\eta +\frac{4L_0^3n^{\frac{3}{2}}}{\sqrt{2}\pi^{\frac{3}{2}}}\sum_{k=0}^{K-1}\frac{\eta \ln(K)}{L_0^2n} \right] \\
			& = \, \frac{1}{ \frac{\eta}{L_0 \sqrt{n} }\sqrt{K}} \left[f(x_0)-f^*+ 2L_0\sqrt{n+1}\eta +\frac{4\ln(K)L_0\sqrt{n} \eta}{\sqrt{2}\pi^{\frac{3}{2}}} \right].
		\end{align}
		It follows that the iteration complexity required to ensure that  $\mathbb{E}_{R_K,\bxi}\left[\, \left\|\nabla f_{\eta}(x_{R_K})\right\|^2\, \right] \le \varepsilon^2$, is given by  
		\begin{equation*}
			K_{\eta,\varepsilon} = \mathcal{O}\left(\, \left(L_0^2 n \eta^{-1} + L_0^4 n^2 \right)\left(\ln\left(L_0^2 n \eta^{-1} + L_0^4 n^2 \right)\right)^2 \varepsilon^{-4}\left(\ln(\varepsilon^{-4})\right)^2\, \right),  
		\end{equation*}
		where we have invoked Lemma~\ref{lambert}. Consequently, the corresponding sample-complexity of computing an $(\eta,\varepsilon)$-stationary point 
		is $\tilde{\mathcal{O}}(nK_{\eta,\varepsilon}). $
		\end{proof}

				\begin{remark} {\em Observe that the dependence on $n$ is  significantly better than that obtained by Nesterov when employing Gaussian smoothing~\cite{nesterov2017random}.  Observe that the sample-complexity matches that in recent work leveraging spherical smoothing (see~\cite{marrinan2026zeroth} and the references therein), requiring a single-sample at every step, while the dimension dependence is slightly worse than that obtained via a mini-batch scheme reliant on spherical smoothing.  }
				\end{remark} 

				\subsection{{Asymptotic a.s. convergence}} \label{sec:5.3}
				{In much of the prior work, there appear to be no asymptotic almost sure convergence guarantees to the set of Clarke-stationary points. The key challenge lies in driving the sequence $\{\eta_k\}$ to zero and is precisely the concern that is addressed in this subsection.} At this stage, we remind the reader of the definition of subgradient set of $\phi$-mollifier of $f$ in Definition~\ref{Def:subgradmoli} {and proceed to prove almost sure convergence to the set of Clarke stationary points in a subsequential sense}. 
				\begin{assumption}\label{asm:etagammanoncvx}
					Suppose the sequences $\left\{ \gamma_k \right\}$ and $\left\{ \eta_k\right\}$ are such that  $\sum_{k=1}^\infty|\eta_k-\eta_{k-1}| < \infty$, $\sum_{k=0}^\infty\frac{\gamma_k^2}{\eta_k} < \infty$ and $\sum_{k=0}^\infty\gamma_k = \infty$. 
				\end{assumption}
				\begin{remark}\label{Rk: etagammanoncvx}
					For the choice of sequences $\left\{\eta_k\right\}$ and $\left\{\gamma_k\right\}$ for satisfying Assumption~\ref{asm:etagammanoncvx}, one option is to set $\eta_k=\tfrac{1}{(k+1)^\beta}, \gamma_k=\tfrac{1}{(k+1)^\al}$ such that $0 <\beta<1 , 0<\al < 1$ and $2\al-\beta>1$. In this case it is easy to see that $ \sum_{k=0}^\infty\tfrac{\gamma_k^2}{\eta_k} < \infty$ and $\sum_{k=0}^\infty\gamma_k = \infty$. Now let us proceed to show that $\sum_{k=1}^\infty \ |\eta_k-\eta_{k-1}| < \infty$. Since function $x^\beta$ has a decreasing derivative, by the mean value theorem we have that 
					\begin{equation*}
						\frac{u^\beta-(u-1)^\beta}{1} \leq \beta(u-1)^{\beta-1}
					\end{equation*}
					Therefore, by invoking the above bound in the deriving a bound on $\sum_{k=1}^\infty|\eta_k-\eta_{k-1}|$, we obtain
					\begin{align}\label{eq:etak-etak-1}
						\sum_{k=1}^\infty|\eta_k-\eta_{k-1}| & = \sum_{k=2}^{\infty}\left|\frac{1}{k^\beta}-\frac{1}{(k-1)^\beta}\right| \leq \sum_{k=2}^\infty \notag \frac{k^\beta-(k-1)^\beta}{k^\beta (k-1)^{\beta}} \\
						& \leq \sum_{k=2}^\infty\frac{\beta}{k^\beta (k-1)} < \sum_{k=2}^\infty\frac{\beta}{(k-1)^{1+\beta}} < \infty.
					\end{align}
				\end{remark}	
				\begin{proposition}
					Suppose Assumption~\ref{asm:non-convex} - \ref{asm:etagammanoncvx} hold. In addition, suppose any subsequence of $(x_k)$ generated from \eqref{eq:updateRn} has an almost sure convergent subsubsequence. Then there exists a cluster point $x^{\ast}$ of $(x_k)$ such that $0 \in \partial f(x^*)$ almost surely, that is, $x^*$ is a stationary point almost surely.
				\end{proposition}
				\begin{proof}
				By Proposition~\ref{prop:gradgotozero}.(a) we have 
				\begin{align}\label{eq:Enablafeta}
					\frac{\sum_{k=0}^{K-1}\gamma_k\mathbb{E}\left[\, \left\|\nabla f_{\eta_k}(x_k)\right\|^2\, \right]}{\sum_{k=0}^{K-1}\gamma_k} &\leq \frac{1}{\sum_{k=0}^{K-1}\gamma_k} \left[f(x_0)-f^*+L_0\sqrt{n+1}\sum_{k=1}^K|\eta_k-\eta_{k-1}|\right.\nonumber\\
					&~~~\left.+L_0\sqrt{n+1}(\eta_0+\eta_K)  +\frac{4L_0^3n^{\frac{3}{2}}}{\sqrt{2}\pi^{\frac{3}{2}}}\sum_{k=0}^{K-1}\frac{\gamma_k^2}{\eta_k} \right].
				\end{align}
				By Assumption~\ref{asm:etagammanoncvx} we know the numerator of the right hand side of \eqref{eq:Enablafeta} is bounded by a constant $C$. Therefore we have 
				\begin{equation}\label{eq:Enablafto0}
					\frac{\sum_{k=0}^{K-1}\gamma_k\mathbb{E}\left[\, \left\|\nabla f_{\eta_k}(x_k)\right\|^2\, \right]}{\sum_{k=0}^{K-1}\gamma_k} \leq\frac{C}{\sum_{k=0}^{K-1}\gamma_k}\rightarrow 0 ,
				\end{equation}
				as $K\rightarrow\infty$. We {may} conclude that there exists a subsequence {$\left\{x_{k_j}\right\}$} such that $\E[\|\nabla f_{\eta_{k_j}}(x_{k_j})\|^2]\rightarrow 0$. If not, then there exists a positive constant $r$ such that $$ \liminf_{k \rightarrow \infty}\E\left[\|\nabla f_{\eta_{k}}(x_{k})\|^2\right]=r>0.$$ This implies that 
				\begin{equation*}
					\lim_{k\rightarrow\infty}\frac{\sum_{k=0}^{K-1}\gamma_k\mathbb{E}\left[\, \left\|\nabla f_{\eta_k}(x_k)\right\|^2\, \right]}{\sum_{k=0}^{K-1}\gamma_k}\geq r>0,
				\end{equation*}
				which {leads to} a contradiction {with} \eqref{eq:Enablafto0}. By  the Markov inequality, we have for any $\varepsilon>0$, we have 
				\begin{equation*}
					\mathbb{P}\left(\|\nabla f_{\eta_{k_j}}(x_{k_j})\|>\varepsilon\right)\leq \frac{\E\left[\|\nabla f_{\eta_{k_j}}(x_{k_j})\|^2\right]}{\varepsilon^2}\rightarrow0,
				\end{equation*}
				as $j\rightarrow\infty$. Since $\varepsilon$ is arbitrary, it follows that $\nabla f_{\eta_{k_j}}(x_{k_j}) \rightarrow 0$ in probability. Therefore there exists a subsequence {with index set  $\{ l_j\}$, where $\left\{l_j\right\}  \, \subseteq \left\{ k_j\right\}$} such that 
				\begin{equation}\label{eq:nablafa.s.}
					\nabla f_{\eta_{l_j}}(x_{l_j}) \rightarrow 0, \quad\text{a.s.} 
				\end{equation}
				By 
				our assumption that every subsequence has an almost sure convergent subsubsequence, there exists a subsequence $(m_j)\subset(l_j)$ such that
				\begin{equation}\label{eq:xmja.s.}
					x_{m_j} \rightarrow x^*, \quad\text{a.s.}
				\end{equation}
				By \eqref{eq:nablafa.s.}-\eqref{eq:xmja.s.} we know that zero is a cluster point of $\nabla f_{\eta_{m_j}}(x_{m_j})$ such that $x_{m_j}\rightarrow x^{*}$. From Definition~\ref{Def:subgradmoli}, we have $	0 \in \partial_\phi f(x^*)$ a.s. Then by Proposition~\ref{prop:phiinclarke} we have $0 \in \partial f(x^*)$ a.s. 
				\end{proof}

				\section{Numerics}\label{sec:numerics}
				In this section, we {present numerics to examine the performance of the schemes in }  the non-decision-dependent setting in Section~\ref{subsec:nonDDnum} and extend our focus to decision-dependent setting in Section~\ref{subsec:DD}.
				\subsection{Non-Decision-dependent regime}\label{subsec:nonDDnum}
				In this section, we provide some numerics, comparing the proposed scheme with three ZO schemes. {Specifically, we compare our proposed {\bf esGs} scheme with Gaussian smoothing counterparts~\cite{nesterov2017random}, spherical smoothing~\cite{cui2023complexity}, and simultaneous perturbation stochastic approximation (SPSA)~\cite{spall1992multivariate}.} To ensure parity, we terminate
				all algorithms after the same number of calls to the function oracle; specifically, we ran $200$ iterations of our method and $200n$ steps  of the other methods. We generated $20$ replications of each scheme. For each problem considered, we provided a tabulation of error given by the empirical average of either $\mathbb{E}[f(\bar{x}_K)-f^*]$ or $\mathbb{E}[f(x_K) - f^*]$ over the replications while ``time'' represents the average runtime in \texttt{Matlab}. 
				\begin{figure}[htb]
					\centering
					\begin{minipage}{0.48\textwidth}
						\centering
						\begin{subfigure}{0.48\textwidth}
							\includegraphics[width=\linewidth]{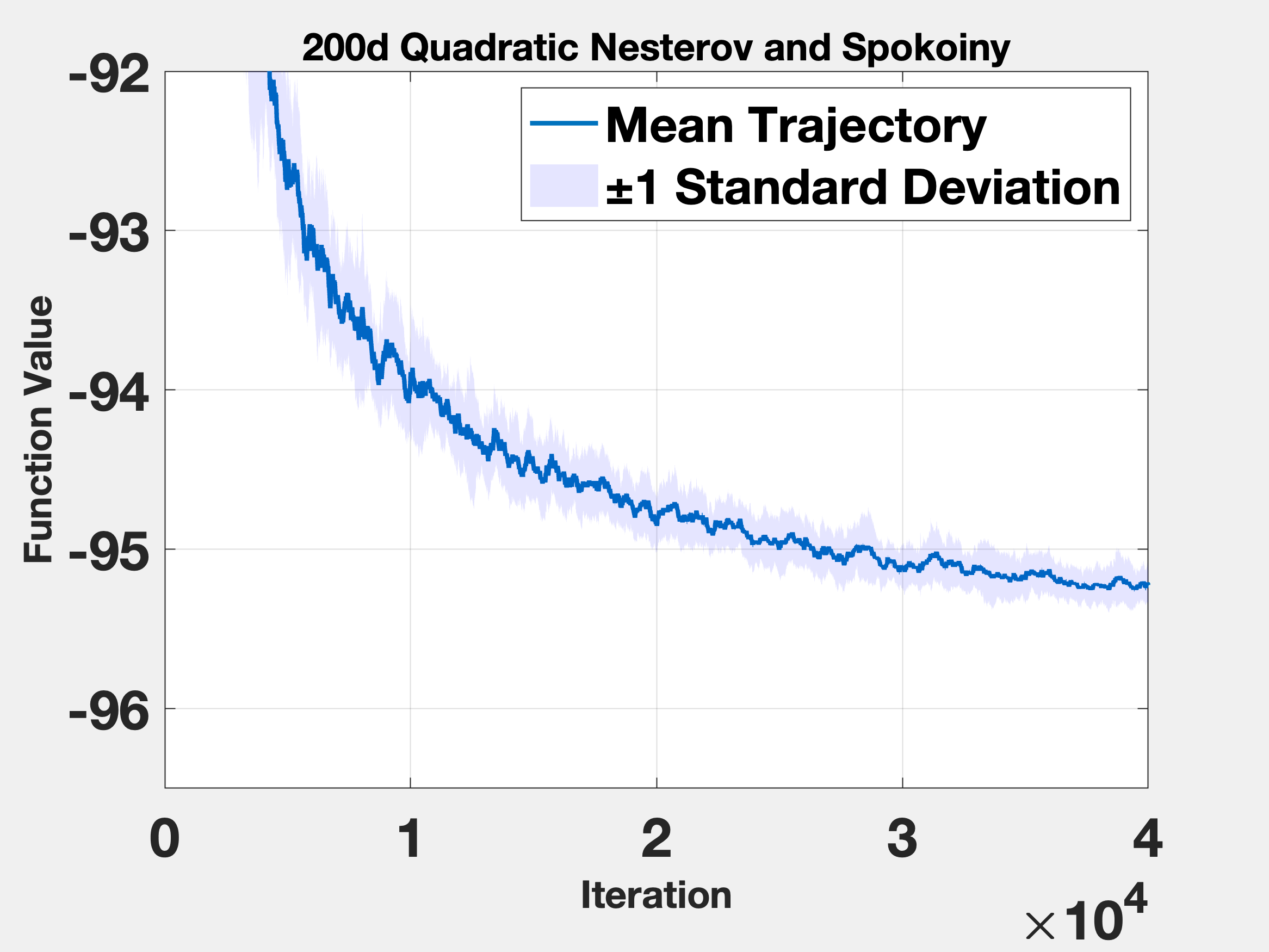}
							\caption{(a)}
							\label{fig:a}
						\end{subfigure}
						\hfill
						\begin{subfigure}{0.48\textwidth}
							\includegraphics[width=\linewidth]{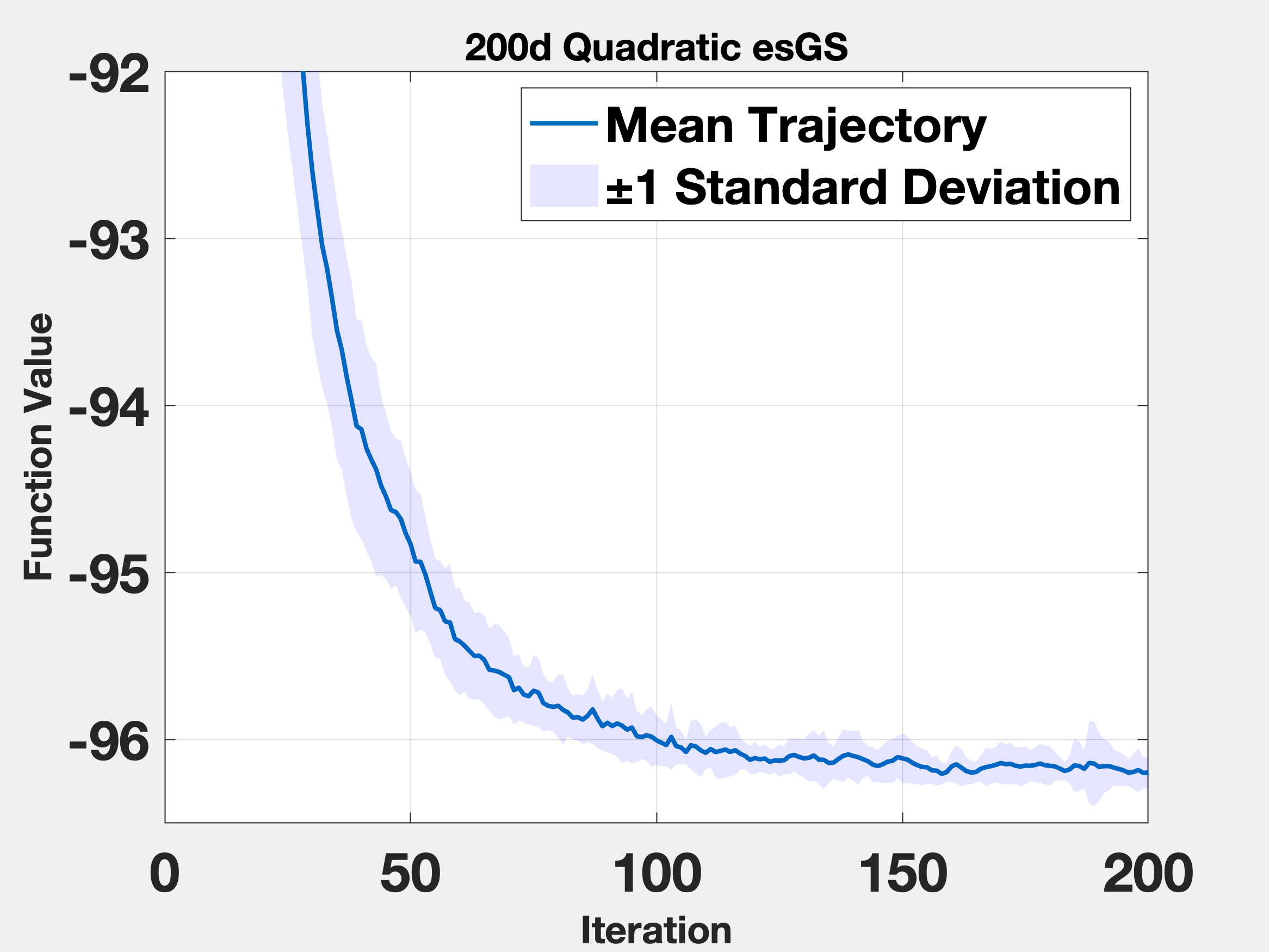}
							\caption{(b)}
							\label{fig:b}
						\end{subfigure}	
					\end{minipage}\hfill
					\begin{minipage}{0.51\textwidth}
						\centering
						\fontsize{5}{6}\selectfont
						\captionsetup{width=\textwidth}
						\captionof{table}{{$\ell_1$-regularized strongly} convex quadratic obj.}
						\label{tab:QuadP}
						\begin{tabular}{ |p{0.25cm} |p{0.53cm}|p{0.65cm}|p{0.53cm}|p{0.63cm}|p{0.53cm}|p{0.63cm}| p{0.53cm}| p{0.53cm}|}
							\hline
							&  \multicolumn{2}{C{1.6cm}|}{\citet{nesterov2017random}} & \multicolumn{2}{c}{\citet{cui2023complexity}}& \multicolumn{2}{|c|}{\citet{spall1992multivariate} }& \multicolumn{2}{c|}{esGs} \\
							\hline
							$n$ &error&time&error&time&error&time&error&time\\
							\hline
							10  & 0.0700   & 0.0096  & 0.0558&0.0092&0.0562&0.0114&0.0280&0.0044 \\
							\hline
							20 & 0.1580  & 0.0390 &0.1444 &0.0293&0.1334&0.0350&0.0285&0.0092 \\
							\hline
							50  &  0.4047  &0.5408   &0.3273 &0.4965&0.3432&0.5224&0.0383&0.0436\\
							\hline
							100 &0.6066 & 3.0719  &0.6360 &3.0741&0.6337&3.0199&0.0615&0.1178\\
							\hline
							150 & 0.8232   & 11.5668  & 0.8070&12.6310&0.8680&11.9770&0.0828&0.2913\\
							\hline
							200 &  1.0683  & 39.0659  & 1.0025&38.0015&1.0228&37.4888&0.0955&0.6088\\
							\hline
							250 & 1.1982   & 69.4609  &1.1731 &71.8910&1.1298&70.4420&0.1185&0.9991\\
							\hline
						\end{tabular}
					\end{minipage}
				\end{figure}
				
				\underline{Stochastic quadratic objectives {with $\ell_1$-regularization}.}
				Our first problem is a nonsmooth convex stochastic quadratic optimization problem in which $$F(x,\xi) =\tfrac{1}{2}x^\top \hat{Q}x+b(\xi)^\top x+ 0.5\|x\|_1,$$ with$X \triangleq [-1,1]^n$, $\hat{Q}=Q+W${($Q \in \mathbb{S}^n_{++}$)}, $b(\xi)=b+\xi$, {$\bxi \sim \mathcal{N}(0,1)$}, and $W=\tfrac{1}{n}B^\top B$ where the components of $B$ are generated from $\mathcal{N}(0,0.01)$. All algorithms utilize the same steplength and smoothing parameters $\gamma_k=\eta_k=1/(\text{norm}(Q)\sqrt{{k}})$ and a deterministic starting point $x_0=(5,5,5,5,5,0,0,\cdots)$.  

				\underline{Piecewise-linear strongly convex and convex objectives.}
				In our second and third problems, we consider a strongly convex and convex piecewise-linear objective given by $$F(x,\xi)=\phi\left(\sum_{i=1}^n\left(\frac{i}{n}+\xi_i\right)x_i\right)+\tfrac{\mu}{2}\|x\|^2$$
				 where {$\mu = 1$ (strongly convex) or $\mu = 0$ (convex)}, $X \triangleq \{ \, x \, \mid \,\|x\| \leq 1\,\}$, $\phi(t)=\max_{1 \leq j \leq m}(v_j+s_jt)$, and $\xi_i \sim \mathcal{N}(0,1)$. We choose $\{v_j\}_{j=1}^5=\{0.2,0.3,0.6,0.5,0.8\}$, and $ \{s_j\}_{j=1}^5 =\{0.9,0.2,0.1,0.5,0.5\}$. All of the algorithms employ the same steplength and smoothing sequences given by $\gamma_k=\eta_k=1/k^{0.52}.$ 
				\begin{figure}[htbp]
					\centering
					\begin{minipage}{0.48\textwidth}
						\centering
						\begin{subfigure}{0.48\textwidth}
							\includegraphics[width=\linewidth]{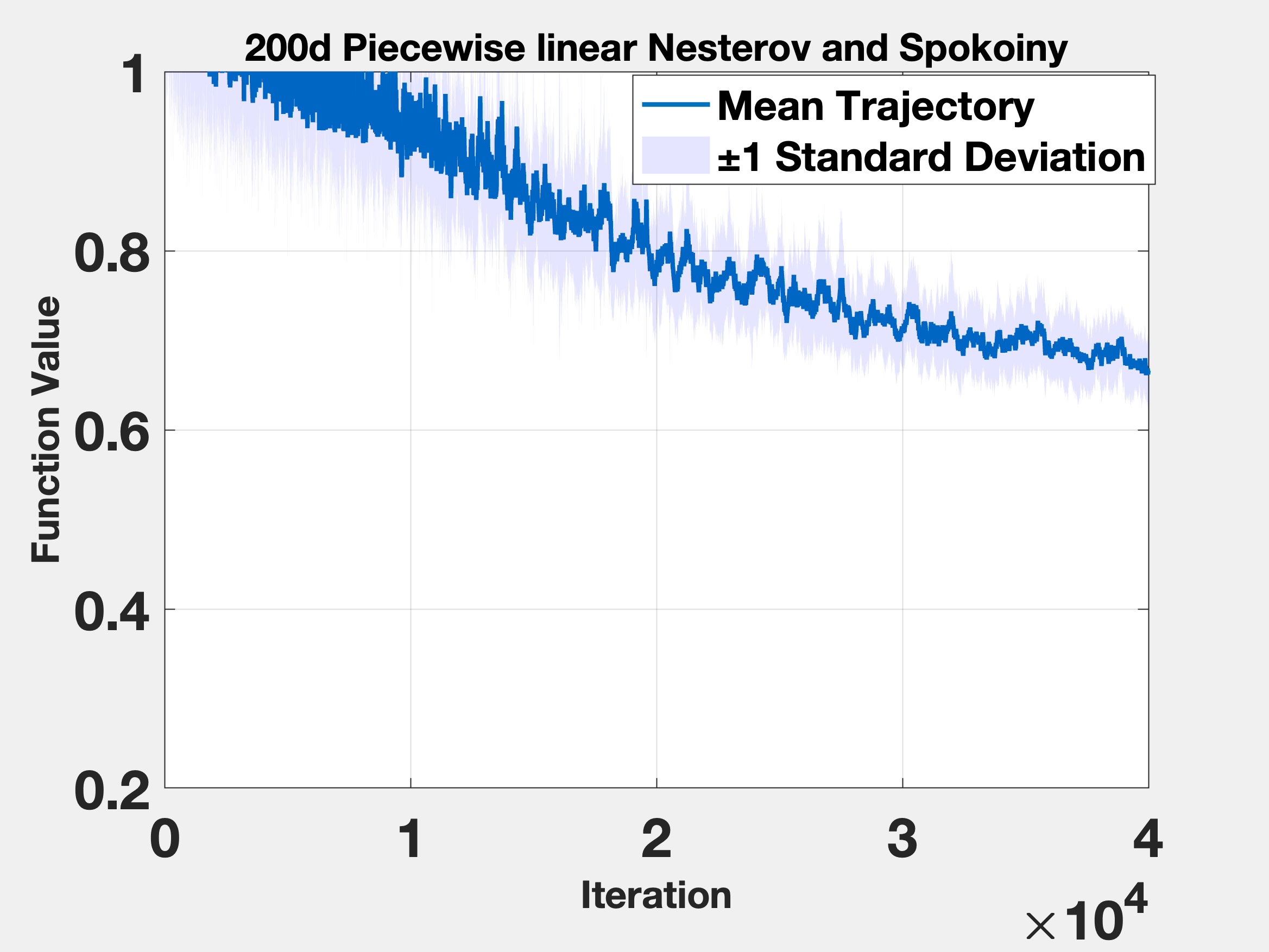}
							\caption{(c)}
							\label{fig:c}
						\end{subfigure}
						\hfill
						\begin{subfigure}{0.48\textwidth}
							\includegraphics[width=\linewidth]{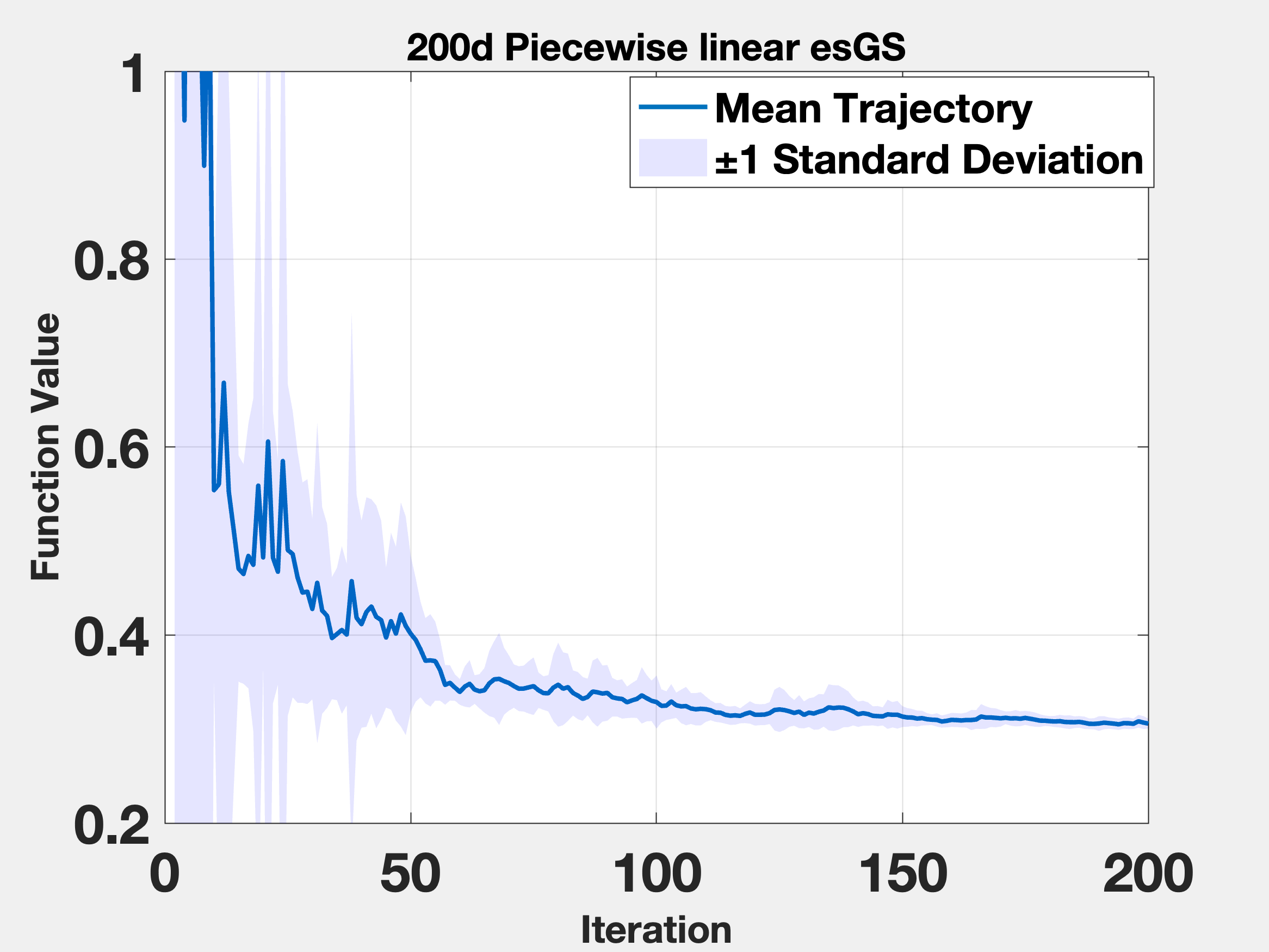}
							\caption{(d)}
							\label{fig:d}
						\end{subfigure}
					\end{minipage}\hfill
					\begin{minipage}{0.51\textwidth}
						\centering
						\captionsetup{width=\textwidth}
						\captionof{table}{{Strongly} Convex Piecewise-Linear Objective}
						\fontsize{5}{6}\selectfont
						\label{tab:PLP}
						\centering
						\begin{tabular}{ |p{0.35cm} |p{0.5cm}|p{0.5cm}|p{0.5cm}|p{0.5cm}|p{0.5cm}|p{0.7cm}| p{0.5cm}| p{0.6cm}|}
							\hline
							&  \multicolumn{2}{|C{1.6cm}|}{\citet{nesterov2017random}} & \multicolumn{2}{|C{1.6cm}|}{\citet{cui2023complexity}}& \multicolumn{2}{|C{1.2cm}|}{\citet{spall1992multivariate} }& \multicolumn{2}{|C{1.2cm}|}{esGs} \\
							\hline
							$n$ &error&time&error&time&error&time&error&time\\
							\hline
							10  & 0.0672 &0.0055  & 0.0844&0.0048&0.1008&0.0067&0.0205&0.0024 \\
							\hline
							100  & 0.2111&0.0938&0.5434&0.0860&0.5292&0.1190&0.0228&0.0241\\
							\hline
							150  & 0.3264&0.1709&0.6346&0.1589&0.6507&0.2206&0.0314&0.0739\\
							\hline
							200  & 0.4014&0.2697&0.7299&0.2539&0.7627&0.3466&0.0400&0.1217\\
							\hline
							500  & 0.9795&1.2842&0.9311&1.3229&0.9448&1.7489&0.1098&0.1431\\
							\hline
							1000  & 1.6041&4.9934&1.8245&4.8908&1.7411&6.3167&0.1870&2.1960\\
							\hline
							2000  &2.4267&16.9147&2.4379&17.1208&2.4572&22.4920&0.3253&2.5243\\
							\hline
							3000  &3.1195&38.0462&3.1155&38.4532&3.0184&50.4454&0.5203&5.8374\\
							\hline
							4000  &3.6107&92.6660&3.5918&101.2859&3.5943&133.5990&0.7237&10.1906\\
							\hline
						\end{tabular}
					\end{minipage}
				\end{figure}
				\begin{figure}
					\begin{minipage}{0.48\textwidth}
						\centering
						\begin{subfigure}{0.48\textwidth}
							\includegraphics[width=\linewidth]{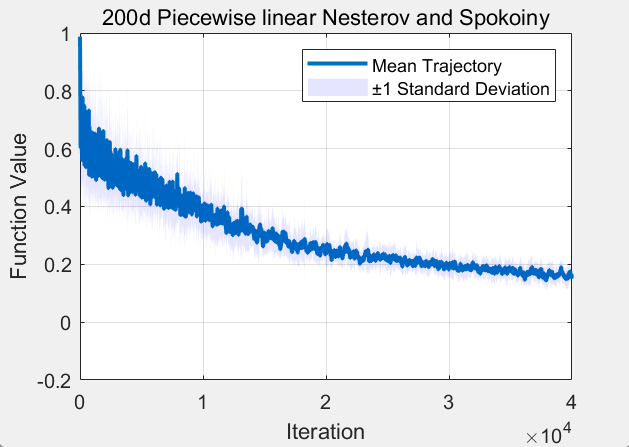}
							\caption{(e)}
							\label{fig:e}
						\end{subfigure}
						\hfill
						\begin{subfigure}{0.48\textwidth}
							\includegraphics[width=\linewidth]{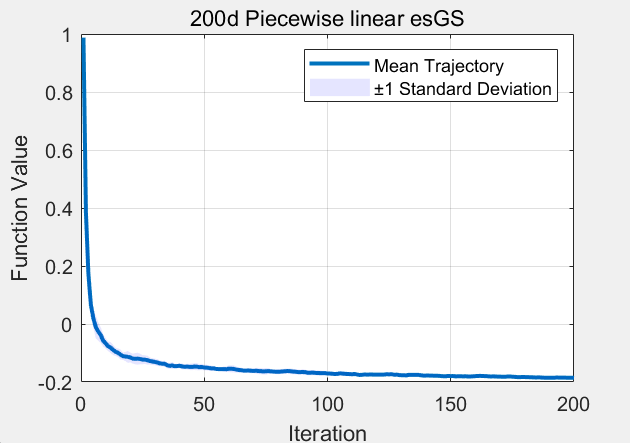}
							\caption{(f)}
							\label{fig:f}
						\end{subfigure}
					\end{minipage}\hfill
					\begin{minipage}{0.51\textwidth}
						\centering
						\captionof{table}{Convex Piecewise-Linear Objective}
						\fontsize{5}{6}\selectfont
						\label{tab:PLNstrong}
						\centering
						\begin{tabular}{ |p{0.35cm} |p{0.5cm}|p{0.5cm}|p{0.5cm}|p{0.5cm}|p{0.5cm}|p{0.7cm}| p{0.5cm}| p{0.6cm}|}
							\hline
							&  \multicolumn{2}{|C{1.6cm}|}{\citet{nesterov2017random}} & \multicolumn{2}{|C{1.5cm}|}{\citet{cui2023complexity}}& \multicolumn{2}{|C{1cm}|}{\citet{spall1992multivariate} }& \multicolumn{2}{|C{1cm}|}{esGs} \\
							\hline
							$n$ &error&time&error&time&error&time&error&time\\
							\hline
							10  & 0.0130 &0.0199  & 0.0104&0.0216&0.0105&0.0236&0.0038&0.0042 \\
							\hline
							100  & 0.1652&0.2184&0.1563&0.2271&0.1618&0.2586&0.0188&0.0148\\
							\hline
							150  & 0.2713&0.3605&0.2702&0.3710&0.2673&0.4319&0.0262&0.0488\\
							\hline
							200  & 0.3900&0.5004&0.3820&0.5193&0.3764&0.6260&0.0350&0.0623\\
							\hline
							500  & 0.9210&1.6481&0.9118&1.6893&0.9246&2.1700&0.0854&0.4169\\
							\hline
							1000  & 1.5934&4.7528&1.5831&4.7925&1.6127&6.4975&0.1516&1.7818\\
							\hline
							2000  &2.4821&15.0903&2.4869&15.4758&2.4473&21.9957&0.2837&6.6433\\
							\hline
							3000  &3.0783&33.3002&3.0420&33.5104&3.0753&47.9169&0.4184&14.4593\\
							\hline
							4000  &3.5804&77.9254&3.5594&85.7464&3.5568&102.9953&0.5313&26.9370\\
							\hline
						\end{tabular}
					\end{minipage}
				\end{figure}
				Tables~\ref{tab:QuadP}-\ref{tab:PLNstrong} represent the comparisons across the ZO schemes and our proposed framework with the ({\bf esGs}) estimator for the quadratic and piecewise-linear problem, respectively. Figures~\nameref{fig:a}-\nameref{fig:f} represent the trajectories and the spread across replications across the two problems, comparing our method with that by Nesterov and Spokoiny.
				\noindent {{\bf Insights.} (a) {\em Empirical error.} The empirical error
					produced by our scheme is \underline{nearly an order of magnitude} better than competing
					schemes in the first example, while similar benefits emerge in the
					piecewise-linear setting.  (b) {\em Computational time.} In terms of
					computational time, the distinctions are even more pronounced, a consequence of
					the iteration complexity of our scheme being $\mathcal{O}(n)$ better than its
					counterparts. For instance, in quadratic settings when $n = 200$, our scheme
					requires $0.609$s while the scheme developed by Nesterov and Spokoiny takes
					approximately $39$s, while producing an empirical error that is \underline{more than $10$
						times worse}. The distinctions are significant, if not as pronounced, in the
					piecewise-linear setting.} We further note that the distinctions in computational time and empirical error distinctions increase
				dramatically with dimension, in alignment with our theory.
				\begin{table}[H]
					\parbox{.54\linewidth}{
						\captionsetup{width=\textwidth}
						\captionof{table}{Strongly Convex Quadratic Objective}
						\tiny
						\label{tab:QuadPvar}
						\begin{tabular}{ |p{1.8cm} |p{1.1cm}|p{1.1cm}|p{1cm}|p{1cm}|}
							\hline
							&  \multicolumn{2}{|C{2.5cm}|}{\citet{nesterov2017random}} & \multicolumn{2}{|c|}{\bf esGs} \\
							\hline
							Var of $(\bxi,B)$ &error&time&error&time\\
							\hline
							$(1,0.01)$& 0.7048  &  77.5523  & 0.0580 & 1.0937\\
							\hline
							$(3,0.01)$& 6.0977 &  77.4978  & 0.4869 & 1.0997\\
							\hline
							$(3,0.09)$  & 5.7965  &  78.6783  & 0.4799 & 1.1006\\
							\hline
							$(10,1)$ &  60.4279 &  78.4632  & 5.5456 & 1.1048\\
							\hline
						\end{tabular}
					}
					\hfill
					\parbox{.45\linewidth}{
						\captionsetup{width=\textwidth}
						\captionof{table}{Strongly Convex Piecewise-Linear Objective}
						\tiny
						\label{tab:PLPvar}
						\centering
						\begin{tabular}{ |p{1cm} |p{1cm}|p{1cm}|p{1cm}|p{1cm}|p{1cm}|p{1cm}|}
							\hline
							&  \multicolumn{2}{|C{2.5cm}|}{\citet{nesterov2017random}} &  \multicolumn{2}{|c|}{\bf esGs} \\
							\hline
							$Var(\bxi)$ &error&time&error&time\\
							\hline
							$1$& 0.3750  &  0.5717  & 0.0299 &0.1383 \\
							\hline
							$3$& 0.8801  & 0.5715  & 0.4287 &0.1482 \\
							\hline
							$5$  & 1.0159  &0.5731   & 0.6405 & 0.1393\\
							\hline
							$10$ & 1.0251  & 0.5754  & 0.7498 & 0.1374\\
							\hline
						\end{tabular}
					}
				\end{table}
				
				\underline{Robustness of scheme.}
				We also {conducted} experiments for noise with {increasing} variance (see ). Under the same settings, we {provide the} same budget of oracle call to each method and {modify} the {variance in the} noise. 
				As {observed in Table~\ref{tab:QuadPvar}-\ref{tab:PLPvar}},  {\bf esGs} {continues to compete well with standard {\bf GS} schemes.}

				\underline{{Stochastic nonconvex optimization}.}	
				Let ${\bxi}\sim U[0,2]$. Define $f: \mathbb{R}^n \mapsto \mathbb{R}$ as 
				$$f(\mathbf{x})=\mathbb{E}\left[\min \left[f_1(\mathbf{x}, \bxi), f_2(\mathbf{x}, \bxi)\right]\right],$$ where $f_1(\mathbf{x}, \bxi)=\sum_i^n\left(x_i-\bxi\right)^2$ and $f_2(\mathbf{x}, \bxi)=\sum_i^n\left(x_i+\bxi\right)^2$. $X \triangleq [-10,10]^n$. The explicit form of $f$ is $f(\mathbf{x}) = ||\mathbf{x}||^2 + 4n/3 - 2|\sum_{i=1}^n x_i|$, which is non-convex. 
				For $\mathbf{x}$ such that $\sum_{i=1}^n x_i\neq 0$, $f$ is differentiable and $\nabla f(x)=2 x-2 \cdot \operatorname{sgn}\left(\sum_{i=1}^n x_i\right) \mathbf{1}$. For $\mathbf{x}$ such that $\sum_{i=1}^n x_i= 0$ we have the subgradient of $f$ is $\partial f(x)=\{2 x+v \cdot \mathbf{1} \mid v \in[-2,2]\}$. It is easy to see that $\mathbf{x}=\mathbf{1}$ and $\mathbf{x}=-\mathbf{1}$ are stationary points of $f$, and all three algorithm tend to converge to this two points. In the following error$=\mathbb{E}\|\nabla f(x_K)\|^2$. We choose $\gamma_k=1/k^{0.9},\eta_k=1/k^{0.3}$, satisfying Remark~\ref{Rk: etagammanoncvx}, and we ran $500$ iterations of our method and $500n$ steps  of the other methods.
				\begin{figure}[htbp]
					\centering
					\begin{minipage}{0.48\textwidth}
						\centering
						\begin{subfigure}[t]{0.49\textwidth}
							\includegraphics[width=\linewidth]{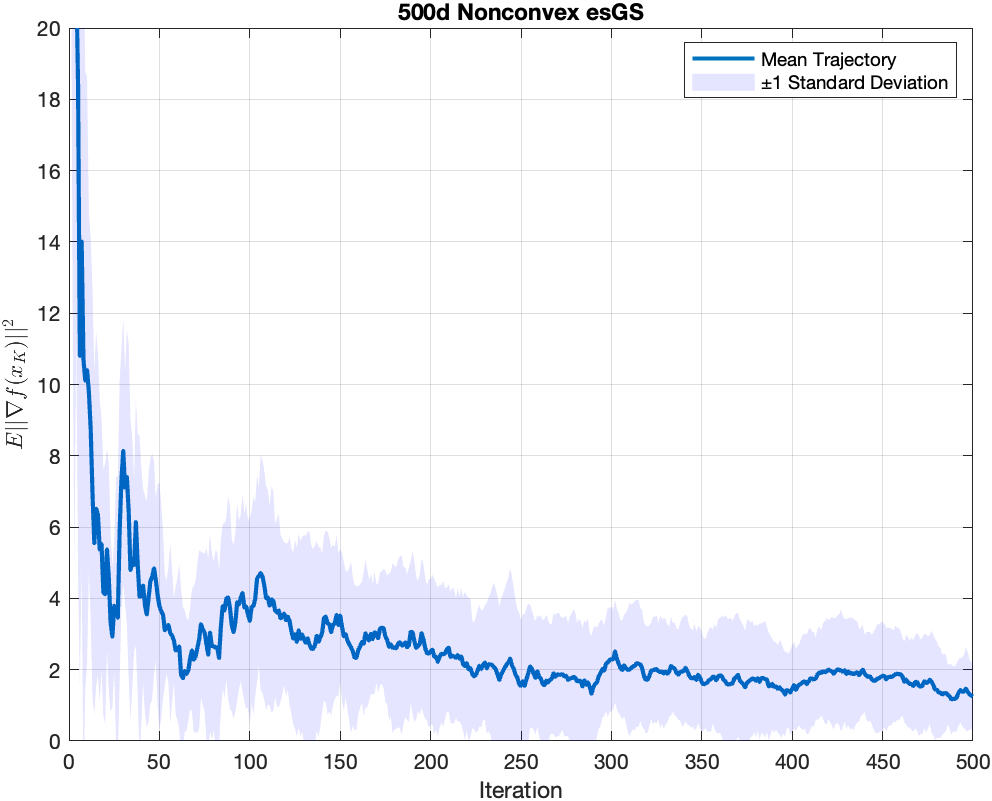}
							\captionsetup{labelformat=empty}
							\caption{\scriptsize esGs (nonconvex)}
							\label{esGsnon}
						\end{subfigure}
						\hfill
						\begin{subfigure}[t]{0.49\textwidth}
							\includegraphics[width=\linewidth]{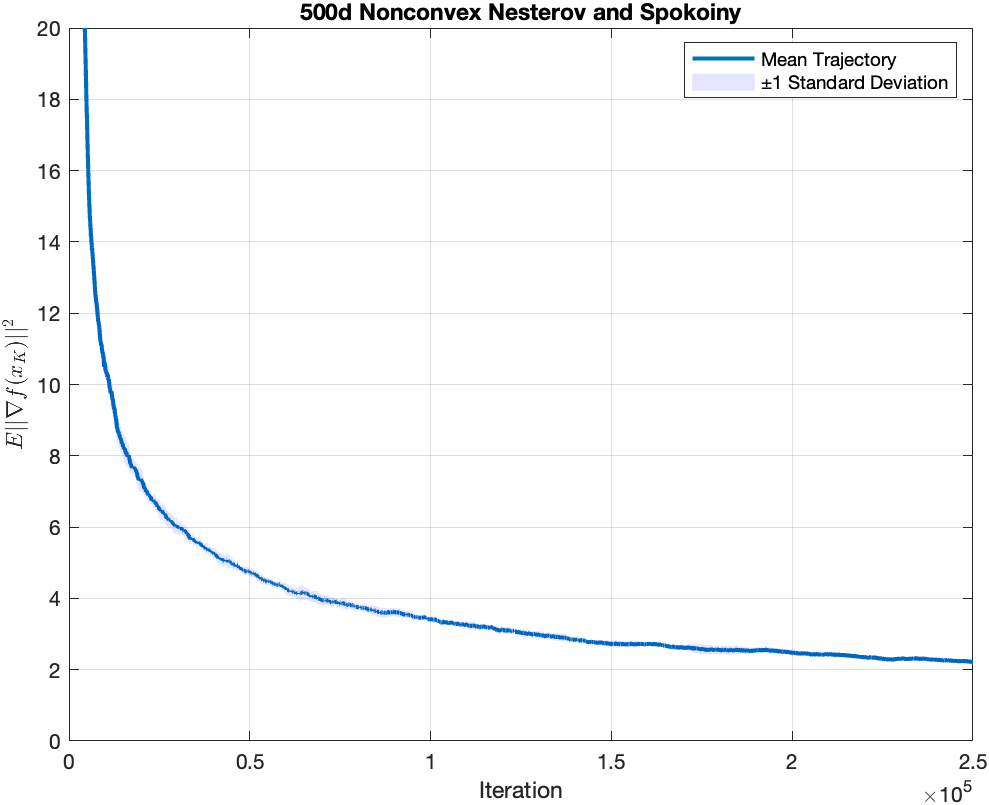}
							\captionsetup{labelformat=empty}
							\caption{\scriptsize Nesterov and Spokoiny (nonconvex)}
							\label{NSnon}
						\end{subfigure}	
					\end{minipage}\hfill
					\begin{minipage}{0.51\textwidth}
						\centering
						\captionof{table}{Nonconvex settings}
						\scriptsize
						\centering
						\begin{tabular}{ |p{0.5cm} |p{0.8cm}|p{1cm}|p{0.8cm}|p{1cm}| p{0.8cm}| p{1cm}|}
							\hline
							&  \multicolumn{2}{|C{1.8cm}|}{\citet{nesterov2017random}} & \multicolumn{2}{|c|}{\citet{cui2023complexity}}& \multicolumn{2}{|c|}{esGs} \\
							\hline
							$n$ &error&time&error&time&error&time\\
							\hline
							5  & 0.2294  &  0.0035 &0.1693  & 0.0038 & 0.1264  &  0.0044  \\
							\hline
							10 &  0.2918 &0.0103   & 0.2908 &  0.0132 &  0.1514 & 0.0086  \\
							\hline
							50  & 0.6392  &  0.0598 & 0.6520 & 0.0603  &  0.5381 &  0.0161  \\
							\hline
							500& 2.2249  &  1.9560 & 2.2505 & 2.0446  & 1.3091  &  0.4324  \\
							\hline
							1000&  3.2148 & 6.9023  & 3.2897 & 7.3248  & 2.1735  &  1.5807  \\
							\hline
							5000 & 7.8538  &  218.7707 & 7.8577 & 237.5124  & 5.1192  & 181.8080   \\
							\hline
						\end{tabular}
					\end{minipage}
				\end{figure}

				\subsection{Decision-dependent {settings}}\label{subsec:DD}
				Here we provide the following numerical experiments for the market problem in Section~\ref{subsec:marketproblem}, with $a=4.5,a_1=0.8, a_2=0.2,L_2=0.5,R_2=2.2,\varepsilon=0.1$ and $\sigma^2=10$. We apply the algorithms for both known and unknown structure of $p(\xi|x)$. The numerical experiments show that both algorithms we proposed converge to the true optimal point, instead of {a performatively stable point}.
				\begin{figure}[H]
					\centering
					\begin{subfigure}{0.2\textwidth}
						\includegraphics[width=\linewidth]{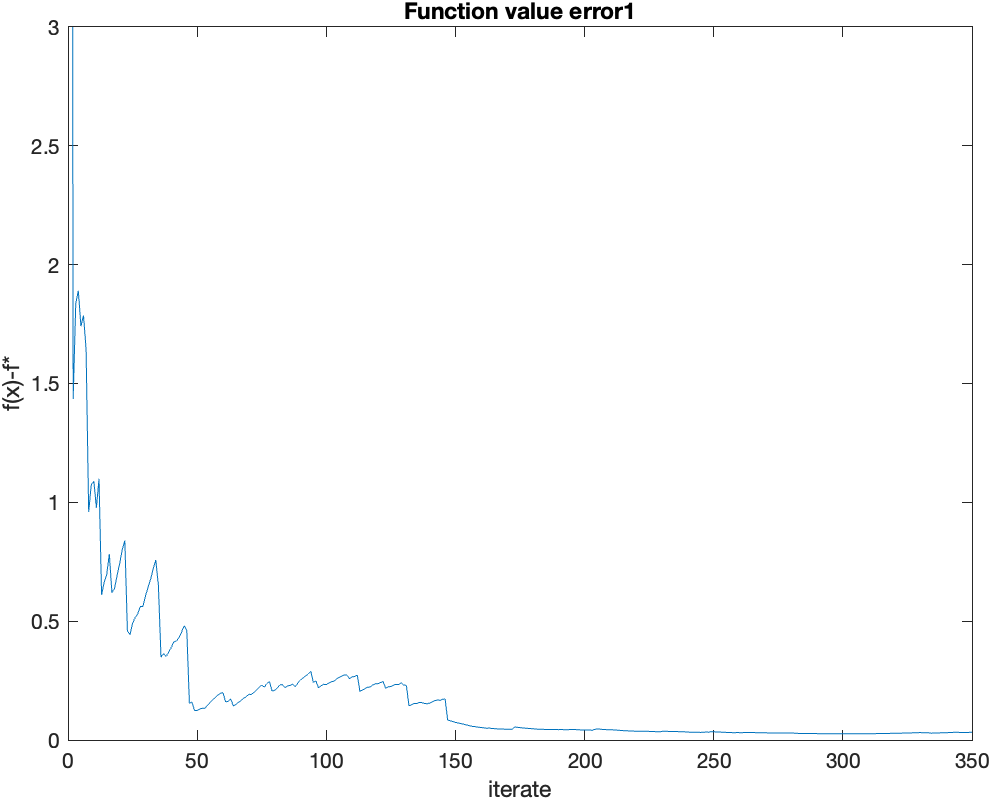}
						\caption{(g)}
						\label{fig:f1}
					\end{subfigure}
					\begin{subfigure}{0.2\textwidth}
						\includegraphics[width=\linewidth]{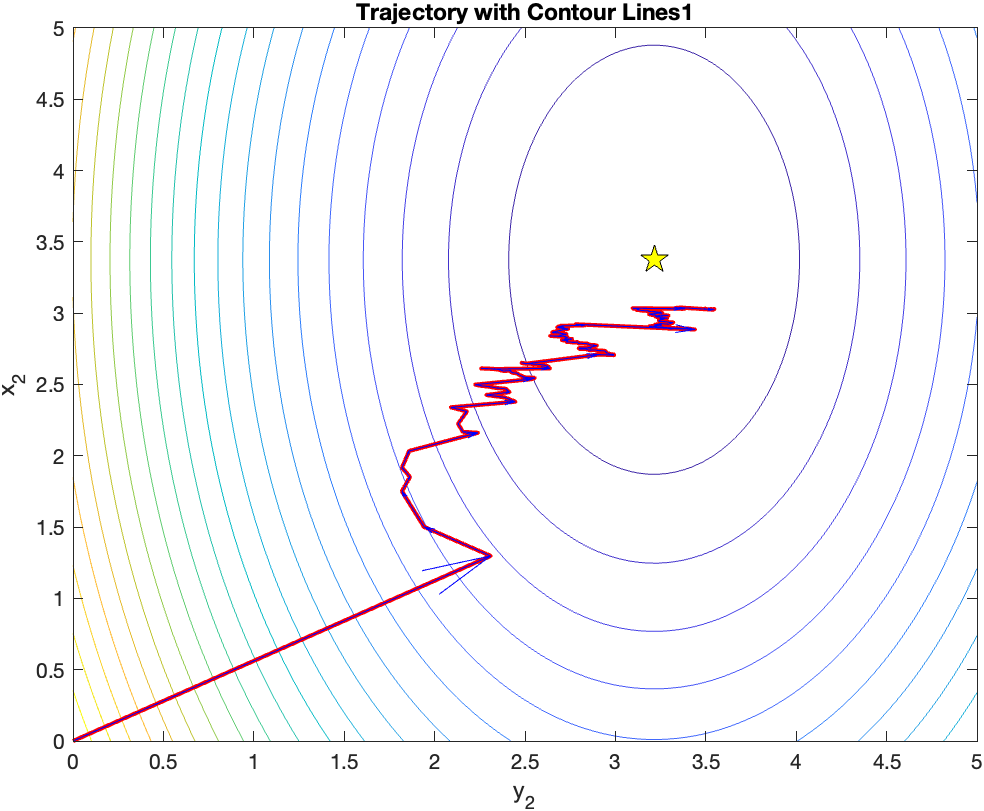}
						\caption{(h)}
						\label{fig:x1}
					\end{subfigure}	
					\begin{subfigure}{0.2\textwidth}
						\includegraphics[width=\linewidth]{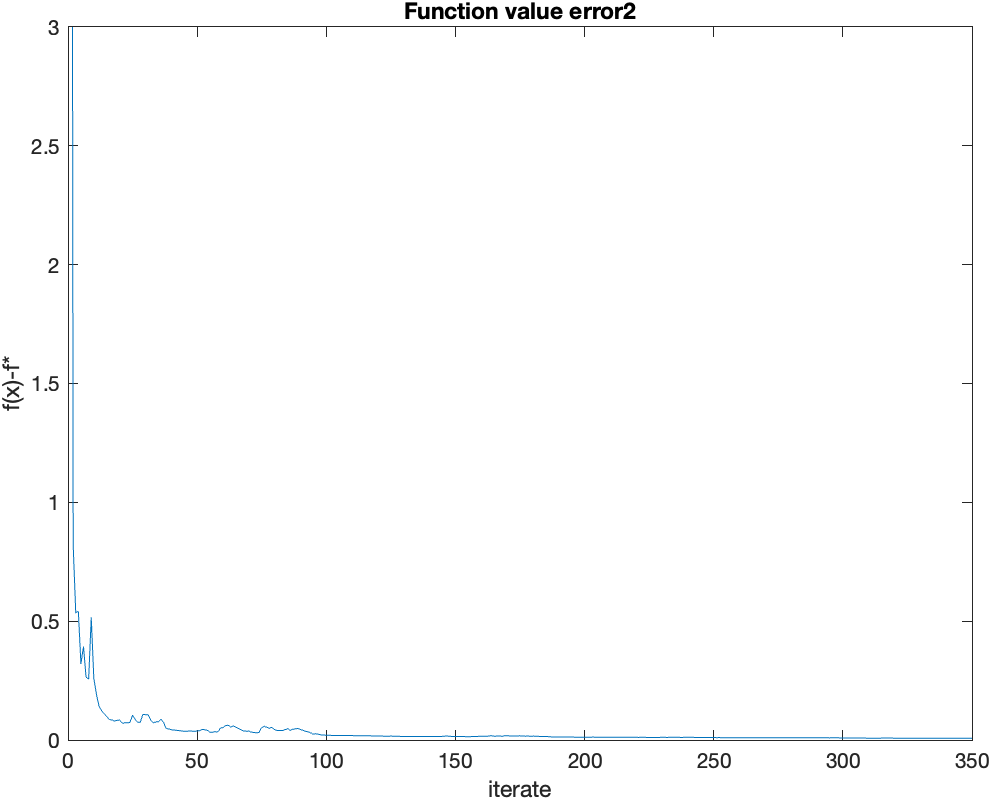}
						\caption{(i)}
						\label{fig:f2}
					\end{subfigure}
					\begin{subfigure}{0.2\textwidth}
						\includegraphics[width=\linewidth]{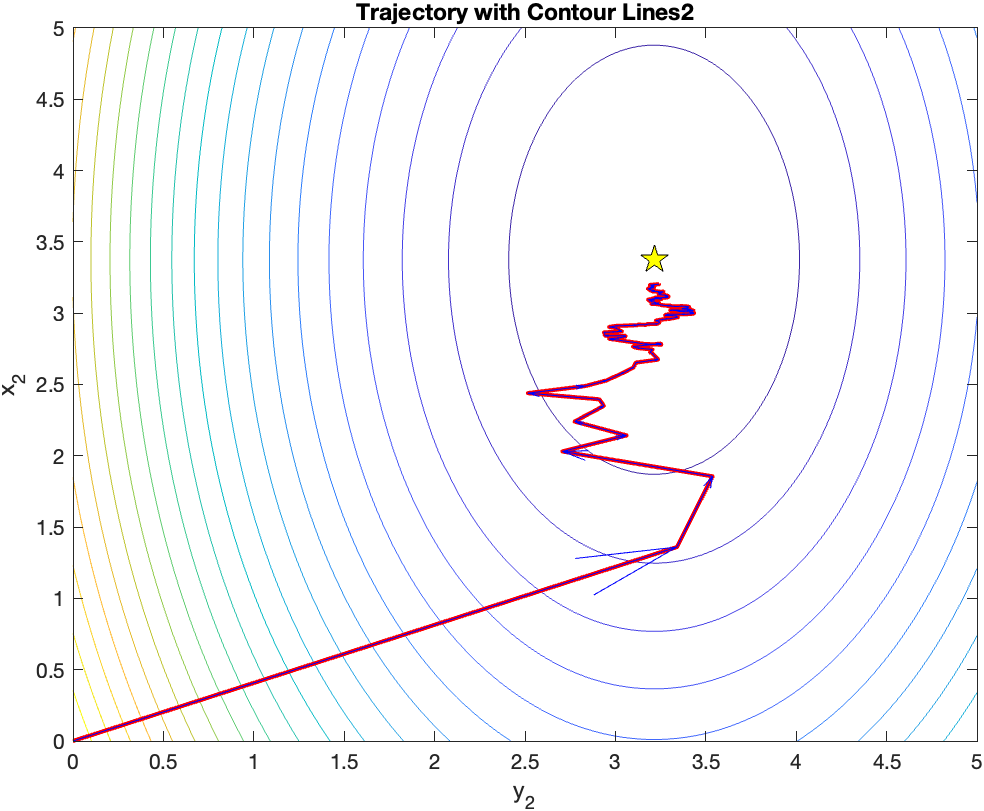}
						\caption{(j)}
						\label{fig:x2}
					\end{subfigure}
					\\
					{\scriptsize Figures \nameref{fig:f1} and \nameref{fig:x1} represent error of function value and the trajectory of $x$ by applying algorithm with known structure of $p(\xi|x)$, respectively. Figures \nameref{fig:f2} and \nameref{fig:x2} represent error of function value and the trajectory of $x$ by applying algorithm with unknown structure of $p(\xi|x)$, respectively.}
				\end{figure}

				\section{Concluding remarks}\label{sec:conc}
				{Extant techniques for addressing  nonsmooth stochastic optimization problems are afflicted by two key challenges. First, when applying stochastic zeroth-order methods enabled by Gaussian smoothing, the moment bounds display a quadratic dependence on dimension and this dependence emerges in the iteration complexity, impeding applicability of such schemes to high-dimensional problems. Second, there appear to be no unified analysis that can capture standard stochastic problems and their decision-dependent counterparts that allow for decision dependence on the probability distribution.} 
				
				{To contend with the first challenge, via a simple change-of-variable argument, we develop an
					exponentially shifted Gaussian smoothing ({\bf esGs}) estimator, reliant on
					shifting via exponential random variables, whose moment bound grows at the rate
					of $\mathcal{O}(n)$, rather that $\mathcal{O}(n^2)$ for Gaussian smoothing estimators. Motivated by the second challenge, we show that this {\bf esGs} estimator can be extended with suitable modifications to contend with two distinct settings of decision-dependence, while exhibiting similar linear dependence on dimension. This provides us with the foundation for developing a unified analysis for convex, strongly convex, and nonconvex regimes. In convex regimes, we provide improved iteration complexity guarantees of $\mathcal{O}(n \varepsilon^{-2})$ (by a factor of $n$) while matching the sample-complexity guarantees of $\mathcal{O}(n^2 \varepsilon^{-2})$ in such schemes. Similar benefits are obtained in nonconvex regimes where we obtain improved iteration complexity guarantees. Notably, in a convex setting, we also provide novel high probability guarantees  and sublinear rate guarantees in an a.s. sense, while a new subsequential a.s. convergence guarantee is provided in nonconvex regimes. Importantly, all of these findings can capture decision-dependent regimes. Preliminary numerics in convex and nonconvex settings support these findings and we observe that our {\bf esGs}-enabled ZO methods provide significant benefits in computational time and empirical accuracy.}
				

\bigskip

\bibliographystyle{plainnat}

\bibliography{refs}

\end{document}